\documentclass[12pt,reqno]{amsart}
\usepackage[numbers,sort&compress]{natbib}
\usepackage{amsfonts,bbm}
\usepackage{amssymb,color}
\usepackage{amssymb}
\usepackage{fancyhdr}
\usepackage[titletoc]{appendix}
\usepackage{enumitem}
\usepackage{amsgen}
\usepackage{amscd}
\usepackage{amsmath}
\usepackage{mathrsfs}
\usepackage{cases}
\usepackage[colorlinks=true]{hyperref}
\hypersetup{urlcolor=blue, citecolor=green, linkcolor=blue}

\usepackage[margin=3cm, a4paper]{geometry}

\newtheorem{thm}{Theorem}[section]
\newtheorem{cor}[thm]{Corollary}
\newtheorem{lem}[thm]{Lemma}
\newtheorem{prop}[thm]{Proposition}
\newtheorem{defn}[thm]{ \bf{Definition}}
\theoremstyle{remark}
\newtheorem{remark}[thm]{Remark}

\newcommand{\EQ}[1]{\begin{align*}\begin{split} #1 \end{split}\end{align*}}
\newcommand{\EQn}[1]{\begin{align}\begin{split} #1 \end{split}\end{align}}
\newcommand{\EQnn}[1]{\begin{align} #1 \end{align}}
\newcommand{\EQnnsub}[1]{\begin{subequations}\begin{align} #1 \end{align}\end{subequations}}

\newcommand{\Del}[1]{}

\def\norm#1{\left\|#1\right\|}
\def\normo#1{\|#1\|}
\def\normb#1{\big\|#1\big\|}

\def\abs#1{\left|#1\right|}

\def\absb#1{\big|#1\big|}

\def\brk#1{\left(#1\right)}
\def\brko#1{(#1)}
\def\brkb#1{\big(#1\big)}

\def\fbrk#1{\left\lbrace#1\right\rbrace}

\def\fbrkb#1{\big\lbrace#1\big\rbrace}

\def\jb#1{\langle#1\rangle}

\def\wt#1{\widetilde{#1}}
\def\wh#1{\widehat{#1}}
\def\wb#1{\overline{#1}}

\def\pd{\partial}

\newcommand{\ra}{{\rightarrow}}

\newcommand{\wra}{{\rightharpoonup}}
\def\loe{\le}
\def\goe{\ge}
\def\lsm{\lesssim}
\def\gsm{\gtrsim}

\newcommand{\N}{{\mathbb N}}

\newcommand{\R}{{\mathbb R}}
\newcommand{\C}{{\mathbb C}}
\newcommand{\Z}{{\mathbb Z}}
\newcommand{\PP}{{\mathbb P}}

\newcommand{\F}{{\mathcal{F}}}
\newcommand{\A}{{\mathcal{A}}}

\newcommand{\HH}{{\mathcal{H}}}

\newcommand{\E}{{\mathcal{E}}}

\newcommand{\U}{{\mathcal{U}}}
\newcommand{\V}{{\mathcal{V}}}
\newcommand{\W}{{\mathcal{W}}}
\newcommand{\dd}{{\mathrm{d}}}
\newcommand{\Q}{{\mathcal{Q}}}
\newcommand{\TT}{{\mathcal{T}}}

\def\dx{\mathrm{\ d} x}
\def\dy{\mathrm{\ d} y}
\def\ds{\mathrm{\ d} s}

\newcommand{\re}{{\mathrm{Re}}}
\newcommand{\im}{{\mathrm{Im}}}

\def\ep{\varepsilon}
\def\al{\alpha}

\def\Om{\Omega}
\def\om{\omega}
\def\ph{\varphi}
\def\th{\theta}

\def\de{\delta}
\def\De{\Delta}

\def\ta{\tau}
\def\la{\lambda}
\def\ga{\gamma}

\newcommand{\I}{\infty}
\def\rev#1{\frac{1}{#1}}
\def\half#1{\frac{#1}{2}}

\numberwithin{equation}{section}

\allowdisplaybreaks[4]

\begin{document}
\title[NLS]{Scattering for defocusing mass sub-critical NLS}

\subjclass[2010]{}

\author{Jia Shen}
\address{(J. Shen) Center for Applied Mathematics\\
	Tianjin University\\
	Tianjin 300072, China}
\email{shenjia@pku.edu.cn}

\author{Yifei Wu}
\address{(Y.Wu) Center for Applied Mathematics\\
	Tianjin University\\
	Tianjin 300072, China}
\email{yerfmath@gmail.com}
\thanks{}

\date{}

\begin{abstract}\noindent
In this paper, we consider the $L_x^2$-scattering of defocusing mass sub-critical nonlinear Schr\"odinger equations with low weighted initial condition. 

In prior work, Tsutsumi-Yajima \cite{TY84BAMS} established the $L^2$-scattering result for initial data $u_0\in H^1 \cap \F H^1$, and this condition was later refined to $\F H^1$ by Hayashi-Ozawa \cite{HO88poinc}. Furthermore, Lee \cite{Lee19IMRN} proved that the inverse wave operator does not exhibit H\"older's continuity solely within $L^2$, indicating the necessity of introducing a weighted condition for scattering. The determination of how many levels of weighting are required becomes a natural inquiry. Moreover, for large $\F H^s$-data with $s<1$, there only exists the wave operator result, but scattering results are lacking. 

In the first portion of our result, we establish the $L^2$-scattering within the framework of $\F H^s$, where $s<1$, thus improving the weighted condition proposed in earlier studies. To the best of our knowledge, this is the first result that surmounts the constraint of the customary condition $\mathcal FH^1$, which was hitherto employed to ensure the validity of the pseudo-conformal energy. Our approach hinges on the application of the pseudo-conformal transformation, which serves to convert the mass-subcritical Cauchy problem into a mass-supercritical final data problem. This allows us to implement a linear-nonlinear decomposition, leading to the establishment of an almost energy conservation, supported by the utilization of a localized interaction Morawetz estimate that controls $\|t^\frac{3p}{16}u\|_{L^4_{t,x}}$ in the framework of $\F H^s, s<1$.




To totally remove the weighted condition, one may consider the randomized data problem. Previously, the almost sure scattering results have been studied by Burq and Thomann \cite{BT20scattering} in one dimension, and by Latocca \cite{Lat20scattering} in the high dimensions for radial data. 

Our second result is the almost sure scattering by introducing a ``narrowed" Wiener randomization in physical space, which gives some improved weighted spacetime estimates for linear solution when the initial data is merely in $L^2$. In particular, our result does not rely on any smallness, radial, or weighted assumption. We apply the method in our first result in the probabilistic setting, which is different from the quasi-invariant measure approach as in \cite{BT20scattering,Lat20scattering}.

\end{abstract}

\maketitle

\tableofcontents

\vskip 1.5cm
\section{Introduction}
\vskip .5cm
In this paper, we consider the nonlinear Schr\"odinger equations (NLS):
\EQn{
	\label{eq:nls}
	\left\{ \aligned
	&i\pd_t u + \frac12\De u = \mu |u|^p u, \\
	& u(0,x) = u_0(x),
	\endaligned
	\right.
}
where $p>0$, $\mu=\pm1$, and
$u(t,x):\R\times\R^d\rightarrow \C$ is a complex-valued function. The positive sign ``$+$" in nonlinear term of \eqref{eq:nls} denotes defocusing source,   and the negative sign ``$-$" denotes the focusing one. The equation \eqref{eq:nls} has conserved mass
\EQ{
	M(u(t)) :=\int_{\R^d} \abs{u(t,x)}^2 \dx=M(u_0),
}
and energy
\EQ{
	E(u(t)) := \frac14 \int_{\R^d} \abs{\nabla u(t,x)}^2 \dx + \rev{p+2}\mu \int_{\R^d}  \abs{u(t,x)}^{p+2} \dx=E(u_0).
}

The class of solutions to equation (\ref{eq:nls}) is invariant under the scaling
\begin{equation}\label{eqs:scaling-alpha}
u(t,x)\to u_\lambda(t,x) = \lambda^{\frac{2}{p}} u(\lambda^2 t, \lambda x) \ \ {\rm for}\ \ \lambda>0.
\end{equation}
Denote
$$
s_c=\frac d2-\frac{2}{p},
$$
then the scaling  leaves  $\dot{H}^{s_{c}}$ norm invariant, that is,
\begin{eqnarray*}
	\|u(0)\|_{\dot H^{s_{c}}}=\|u_{\lambda}(0)\|_{\dot H^{s_{c}}}.
\end{eqnarray*}
This gives the scaling critical exponent $s_c$. Let
$$
2^*=\I, \mbox{ when } d=1 \mbox{ or } d=2;  \quad
2^*=\frac4{d-2}, \mbox{ when }  d\goe 3.
$$
Therefore, according to the conservation law, the equation is called mass or $L_x^2$ \textit{critical} when $p=\frac 4d$, and energy or $\dot H_x^1$ \textit{critical} when $p=\frac4{d-2}$. In this paper, we mainly concern the mass subcritical case when $p<\frac4d$. 

Now, we give the definition of $L_x^2$-scattering theory about NLS. Let $u\in C(\R;L_x^2(\R^d))$ be the global solution to \eqref{eq:nls}. We say that the solution $u$ \textit{scatters} forward in time (or backward in time), if there exists \textit{final data} or \textit{scattering state} $u_+$ (or $u_-$) in $L_x^2(\R^d)$, such that
\EQ{
\lim_{t\ra+\I}\normb{u(t)-e^{\frac12it\De}u_{+}}_{L_x^2} = 0\quad(\text{or }\lim_{t\ra-\I}\normb{u(t)-e^{\frac12it\De}u_{-}}_{L_x^2} = 0).
}
The operator $\Pi_\pm$ is said to be \textit{wave operator}, if there exists a subspace $X_{\pm}\subset L_x^2(\R^d)$ such that for all $u_\pm\in X$, there exists a unique $u_0\in L_x^2$ with the global solution $u$ of \eqref{eq:nls} scatters to $u_\pm$, and we define
\EQ{
\Pi_\pm(u_\pm) = u_0.
}
We also say that the equation \eqref{eq:nls} is \textit{asymptotic completeness} on $Y_\pm$, if there exists $Y_\pm\subset L_x^2(\R^d)$ such that the inverse wave operator $\Pi_{\pm}^{-1}$ is well-defined on $Y_\pm$.

\subsection{Historical background}
The asymptotic completeness and wave operator theory of NLS have been extensively studied. In this paper, we mainly focus on the mass subcritical case, namely $p<\frac 4d$. Notice that there are two important exponents that play the key role in the study of scattering theory: the short range critical exponent $\frac2d$ and the Strauss exponent 
\EQ{
\ga(d):=\frac{2-d+\sqrt{d^2+12d+4}}{2d},
}
which is defined as the larger solution to the quadratic equation $d\ga(\ga+1)=2(\ga+2)$.
Note that, for $d\goe 1$, we have the relation
\EQ{
	\frac 2d<\ga(d)<\frac4d.
}
We define $\F H^s(\R^d):=L^2(\R^d;\jb{x}^s\dd x)$, and $\Sigma^s:=H^s\cap \F H^s(\R^d)$, where $\F$ denotes the Fourier transform. 

The exponent $\ga(d)$ is first proposed by Strauss in \cite{Str81JFA}, and he proved the scattering for mass subcritical NLS with some smallness conditions in $L^{\frac{p+2}{p+1}}$ when $\ga(d)< p < \frac 4d$. Tsutsumi and Yajima \cite{TY84BAMS} proved the existence of scattering states for $\frac 2d< p <\frac 4d$ with initial data $u_0 \in \Sigma^1$ in the defocusing case, where the scattering holds in $L_x^2$ sense. This result was extended to $\F H^1$-data by Hayashi and Ozawa \cite{HO88poinc,HO89JFA}, see Remark 2 in \cite{HO88poinc}. In the very recently, the result was improved in the sense of $H_x^1$-scattering by Burq, Georgiev, Tzvetkov, and Visciglia \cite{BGTV21NLS}.

Additionally, the construction of inverse wave operator requires the uniqueness of scattering state. The wave and inverse wave operators were constructed on $\Sigma^1$ by Tsutsumi \cite{Tsu85poincare} for $\ga(d)<p<\frac4d$, and by Cazenave-Weissler \cite{CW92CMP} and Nakanishi-Ozawa \cite{NO02NoDEA} for the endpoint exponent $p=\ga(d)$. Furthermore, Cazenave-Weissler \cite{CW92CMP} also proved the inverse wave operator with small $\Sigma^1$-data when $p>\frac{4}{d+2}$, and the $\Sigma^s$-data case was studied by Nakanishi-Ozawa \cite{NO02NoDEA} for $p>\frac{4}{d+2s}$ with $0<s<\min\fbrk{\frac d2,p}$. In the case of $0 < s < 2$, the wave operator has been constructed by Ginibre-Ozawa-Velo \cite{GOV94poincare} on $\Sigma^s$, subject to the condition $p>\max\fbrk{\frac 2d,\frac{4}{d+2s},s}$. Nevertheless, a large data scattering result in $\Sigma^s$ with $0 < s < 1$ remains elusive.

While the above scattering results rely on the $\Sigma^s$-weighted or $L^{\frac{p+2}{p+1}}$ decaying condition,  Lee \cite{Lee19IMRN} demonstrated that the wave and inverse wave operators do not exhibit H\"older's continuity solely within $L^2$, indicating the necessity of introducing a weighted (or decaying) condition for scattering. 

The radial assumption is also another direction to study the mass subcritical scattering. Guo and Wang \cite{GW14JAM} first obtained the small data scattering with radial data in $\dot H^{s_c}$ subject to certain restriction on $p$.  As for the large data case, there is currently no general scattering result, while a large data global well-posedness for compact radial data in $\dot H^{s_c}$ was established in \cite{BDSW21Adv}. Moreover, Killip, Masaki, Murphy, and Visan \cite{KMMV19DCDS} proved that if the critical Sobolev norm of solution is bounded, then the solution must be global and scatter for defocusing case. They also proved that if the critical weighted $\F H^{-s_c}$-norm with $s_c=\frac d2-\frac2p$ of the solution is uniformly bounded, then the same result holds in both focusing and defocusing cases, see \cite{KMMV17NoDEA}. 

Another approach to studying the scattering for mass sub-critical NLS involves considering the randomized data problem. The almost sure scattering has been proved by Burq and Thomann \cite{BT20scattering} in one dimension case and by Latocca \cite{Lat20scattering} in the high dimensions for radial data. Their results concerned randomization based on the countable eigen-basis of the Schr\"odinger operator with the harmonic oscillator potential, also known as the Hermite operator. This approach was initially introduced by Burq, Thomann, and Tzvetkov \cite{BTT13Fourier}. They transformed the NLS into the ones with a harmonic potential through a lens transform, and then constructed a quasi-invariant Gaussian measure building on the countable eigen-basis of 1D Hermite operator. This theory currently has since been extended to higher dimensional case for radial data in \cite{Den12APDE,Lat20scattering}, and for algebraic nonlinear terms in \cite{BPT23NLS}.

The probabilistic method has also been applied to the study of wave operator. In \cite{Nak01siam}, Nakanshi observed the deterministic wave operator problem for mass sub-critical NLS in $L^2$ is supercritical, and demonstrated the existence of a final data problem in $L^2$ or $H^1$ for $\frac{2}{d} < p < \frac{4}{d}$ with $d \geq 3$, although uniqueness was not established (resulting in an indeterminate wave operator). However, Murphy \cite{Mur19ProAMS} introduced the spatial Wiener randomization and obtained the existence and uniqueness of the final data problem for NLS almost surely in $L_x^2$ when $\ga(d)<p<\frac4d$,  thus ensuring the well-defined nature of the wave operator. It is worth noting that in the deterministic setting, the existence in $L_x^2$ has been proven, but the uniqueness remains unknown (see \cite{Nak01siam}). Subsequently, Nakanishi and Yamamoto \cite{NY19MRL} further extended the almost sure wave operator below the Strauss exponent $\ga(d)$.

The exponent $\frac 2d$ is sharp for scattering results, since any $L_x^2$-scattering solution to \eqref{eq:nls} must be trivial in the long range case when $0<p\loe \frac2d$, as shown in \cite{Str74scattering,Bar84JMP}. For the critical exponent $p=\frac2d$, Ozawa \cite{Oza91CMP} first constructed the modified wave operator for 1D cubic NLS. The modified scattering was initially proved by Hayashi and Naumkin \cite{HN98AJM}.

\ 

Above all, the previous large data scattering results for mass subcritical NLS are divided into two settings.
\begin{itemize}
\item 
The first type of results relies on the $\F H^1$-\textit{weighted} condition (\cite{TY84BAMS,HO88poinc,HO89JFA,Tsu85poincare,CW92CMP,BGTV21NLS}).
\item
The second type relies on the \textit{radial} assumption when $d\goe 2$ in the small data setting (\cite{GW14JAM}), or with the a priori assumption (\cite{KMMV19DCDS}), or by utilizing invariant-measure method (\cite{Lat20scattering}).
\end{itemize}
Moreover, 
\begin{itemize}
\item When considering the initial data in $\F H^s$ with $s<1$, there are only results for small data scattering (\cite{NO02NoDEA}), conditional scattering (\cite{KMMV17NoDEA}), or wave operator (\cite{GOV94poincare}). 
\item Lee \cite{Lee19IMRN} showed that the inverse wave operator from $L^2$ to $L^2$ cannot be $\al$-H\"older continuous for some $\al$.
\end{itemize}
The 1D almost sure scattering in \cite{BT20scattering} appears to be the only result with non-decaying and non-radial data. Therefore, it is natural to inquire about the behavior between $\F H^1$ and $L^2$-data, which has not been studied yet.

In this paper, our \textit{main purpose} is to study the large data scattering for defocusing mass subcritical NLS with low-weighted and non-radial data. 

Our first result presents a systematic study on the scattering on $\F H^s$ for certain $s<1$, without any restrictions on smallness or radial symmetry. This extends the positive results of \cite{TY84BAMS, HO88poinc} to spaces with lower weights. 

Our second result is the almost sure scattering on $L^2$ by introducing a ``narrowed'' Wiener randomization in physical space. For mass subcritical NLS when $d\goe2$, this result represents the first scattering result without imposing any conditions related to smallness, radial symmetry, or weighted properties on the initial data.


\subsection{Main result I: fractional weighted scattering}
In this section, we present our first main result regarding the scattering of defocusing mass subcritical NLS in the fractional weighted space $\F H^s$ with some $s<1$. The result is divided into two parts: first, we provide a result in the 3D case with $2/3<p<4/3$ to show our main idea, and then apply our method to encompass more general dimensions and nonlinear terms.
\begin{thm}[Fractional weighted scattering: 3D case]\label{thm:frac-weighted-3d}
Suppose that $d= 3$, $\mu=1$, and $\frac23< p < \frac43$. For any $\frac{14}{15}<s<1$, take the initial data $u_0$ such that $\jb{x}^su_0 \in L_x^2(\R^3)$. Then, there exists a unique solution $u\in C(\R;L_x^2(\R^3))$ to the equation \eqref{eq:nls}. Moreover, the solution scatters in $L_x^2(\R^3)$, namely there exist $u_\pm\in L_x^2(\R^3)$ such that
	\EQ{
		\lim_{t\ra\pm\I}\normb{u(t)-e^{\frac12it\De }u_\pm}_{L_x^2(\R^3)} =0.
	}
\end{thm}
\begin{thm}[Fractional weighted scattering: general dimensions]\label{thm:frac-weighted-dd}
Suppose that $d\goe1$, $\mu=1$, and
\begin{enumerate}
	\item $\frac2d< p < \frac4d$, when $1\loe d\loe 8$;
	\item $\frac2d< p < \frac{2}{d-4}$, when $9\loe d\loe 11$;
	\item $\frac2d< p < \frac{2}{d-2}$, when $d\goe 12$.
\end{enumerate}
Then, there exists some $s_0=s_0(d,p)<1$ satisfying the following properties. For any $s_0<s<1$, take the initial data $u_0$ such that $\jb{x}^su_0 \in L_x^2(\R)$. Then, there exists a unique solution $u\in C(\R;L_x^2(\R^d))$ to the equation \eqref{eq:nls}. Moreover, the solution scatters in $L_x^2(\R^d)$, namely there exist $u_\pm\in L_x^2(\R^d)$ such that
\EQ{
\lim_{t\ra\pm\I}\normb{u(t)-e^{\frac12it\De }u_\pm}_{L_x^2(\R^d)} =0.
}
\end{thm}
\begin{remark} We highlight some important points regarding these two results.
\begin{enumerate}
\item 
All the previous large data scattering results in weighted space $\F H^s$ have typically required $s \goe 1$, whereas our results give the first large data scattering in weighted space with weight orders less than $1$. 
\item Our results do not rely on any assumptions of smallness or radial symmetry.
\item In Theorem \ref{thm:frac-weighted-3d}, the lower bound $\frac{14}{15}$ is uniformly in $p$. Notably, when $\frac23<p<\frac{60}{73}$, our assumption is ``supercritical" under scaling, as evidenced by $\F H^{-s_c}\subset \F H^{\frac{14}{15}}$. Furthermore, it is worth mentioning that this lower bound is not optimal.
\end{enumerate}
\end{remark}

The main difficulty in proving our theorems arises from the absence of a pseudo-conformal conservation law. We overcome this difficulty by showing that the pseudo-conformal energy for the nonlinear part of the solution is almost conserved. The key lies in transferring the first-order vector field from the linear part to the nonlinear remainder. Then, the scattering follows by employing a modification of  Tsutsumi-Yajima's argument in \cite{TY84BAMS}.

We further remark that our method is inspired by the study of dispersive equations when the energy conservation is not available. This line of research dates back to Bourgain's early work \cite{Bou98IMRN}. Recently, studies have focused on the application of this approach to problems involving randomized data, starting with the work of Burq and Tzvetkov \cite{BT14JEMS}. We would like to mention the pioneering contributions in this field, such as the first global well-posedness result for energy critical equation by Pocovnicu \cite{Poc17JEMS}, the first scattering result by Dodson-L\"uhrmann-Mendelson \cite{DLM20AJM}, and the first non-radial scattering result by Bringmann \cite{Bri21AJM}. Furthermore, our result can be viewed as a generalization of these studies on mass supercritical NLS to the case of pseudo-conformal energy for mass subcritical NLS, building on the observations made by Nakanishi and Ozawa \cite{NO02NoDEA} that there exists some symmetry between the mass supercritical and subcritical NLS. 

Despite both theories sharing the common objective of establishing an almost conservation law, a key distinction arises between these two approaches. In scenarios involving scattering of mass supercritical equations, it becomes imperative to encompass the entire time integral in the energy increment. This requirement is exemplified in works like \cite{DLM20AJM,Bri21AJM,Bri20APDE}. In contrast, our method, following the application of the pseudo-conformal transform, places a primary emphasis on addressing the singularity near the time origin, particularly when $p$ approximates $\frac{2}{d}$. 

\subsection{The randomization}
In order to study the scattering with $L^2$-data, we consider the random data problem. To begin with, we give the precise definition of randomization. Let $a\in\N$. First, for $N\in 2^{\N}$, we denote the cube sets 
\EQ{
	O_N=&\fbrkb{x\in \R^d:  |x_j| \le N, j=1,2,\cdots,d},
}
and then we define
\EQ{
	Q_N=& O_{2N}\setminus O_{N}.
}

Next, we make a further decomposition of $Q_N$.
Note that $a$ is a positive integer, we make a partition of the $Q_N$ with the essentially disjoint sub-cubes
\EQ{
	\A(Q_N):=\fbrk{Q: Q\text{ is a dyadic cube with length } N^{-a}\text{ and }Q\subset Q_N}.
}
Then, we have $\sharp \A( Q_N)=(2^d-1)N^{d(a+1)}$ and  $ Q_N=\cup_{Q\in\A(Q_N)}Q$.  

Now, we can define the final decomposition
\EQ{
	\Q:=\fbrk{O_1} \cup \fbrk{Q: Q\in \A( Q_N)\text{ and } N\in2^\N}.
}
By the above construction,
\EQ{
	\R^d=  O_1 \cup \brk{ \cup_{N\in 2^\N} Q_N} 
	= O_1\cup\brk{\cup_{N\in 2^\N}\cup_{Q\in\A(Q_N)}Q} =\cup_{Q\in\Q}Q.
}
Therefore, $\Q$ is a countable family of essentially disjoint caps covering $\R^d$. We can renumber the cubes in $\Q$ as follows:
\EQ{
	\Q=\fbrk{Q_j:j\in \N}.
}

\begin{defn}\label{defn:randomization}
	Let $d\goe 2$. Given $\frac2d<p<\frac4d$, we set the parameters $\ep>0$ and $a\in\N$ such that   
	\EQ{
		\ep= \frac{4-dp}{2} \text{, and }a\goe \max\fbrk{\frac{d+4}{d\ep},\frac{5(p+2)}{dp}}.
	}
	Let $\wt \psi_j \in C_0^\I(\R^d)$ be a real-valued function such that $0\loe \wt \psi_j\loe 1$ and
	\EQ{
		\wt \psi_j(x) =\left\{ \aligned
		&1\text{, when $x\in Q_j$,}\\
		&\text{smooth, otherwise,}\\
		&0\text{, when $x\notin 2Q_j$,}
		\endaligned
		\right.
	}
	where we use the notation that $2Q_j$ is the cube with the same center as $Q_j$ and with $\mathrm{diam}(2Q_j)=2\mathrm{diam}(Q_j)$. Now, let
	\EQ{
		\psi_j(x):=\frac{\wt \psi_j(x)}{\sum_{j'\in\N}\wt \psi_{j'}(x)}.
	} 
	Then, $\psi_j\in C_0^\I([0,+\I))$ is a real-valued function, satisfying $0\loe \psi_j\loe 1$ and for all $x \in\R^d$, $\sum_{j\in\N}\psi_j(x)=1$.
	
	Let $(\Om, \A, \PP)$ be a probability space. Let $\fbrk{g_j}_{j\in\N}$ be a sequence of zero-mean, complex-valued Gaussian random variables on $\Om$, where the real and imaginary parts of $g_j$ are independent. Then, for any function $f$, we define its randomization $f^\om$ by
	\EQn{\label{eq:randomization}
		f^\om=\sum_{j\in\N} g_j(\om)\psi_j f.
	}
\end{defn}

\subsection{Main result II: almost sure scattering}
Now, we consider the defocusing NLS with the randomized initial data:
\EQn{
\label{eq:nls-random}
\left\{ \aligned
&i\pd_t u + \frac12\De u = |u|^p u, \\
& u(0,x) = f^\om(x).
\endaligned
\right.
}
The main result in this section is as follows.
\begin{thm}[Probabilistic scattering in $L_x^2$]\label{thm:main}
Suppose that $d\goe 1$ and $\frac{2}{d}< p < \frac{4}{d}$. Let $f\in L_x^2(\R^d)$ and $f^\om$ be defined in Definition \ref{defn:randomization}. Then, for almost every $\om\in\Om$, Then, there exists a unique solution $u\in C(\R;L_x^2(\R^d))$ to the equation \eqref{eq:nls-random}. Moreover, the solution scatters in $L_x^2(\R^d)$, namely there exist $u_\pm\in L_x^2(\R^d)$ such that
\EQ{
\lim_{t\ra\pm\I}\normb{u(t)-e^{\frac12it\De}u_\pm}_{L_x^2(\R^d)} =0.
}
\end{thm}

\begin{remark}
We highlight some important points regarding this result.
\begin{enumerate}
\item
The randomization in Definition \ref{defn:randomization} does not improve the regularity, weighted properties, or decay characteristics of the initial data $f^\om$ itself, see the discussion in Remark 1.2 in \cite{BT08inventI}. Furthermore, the initial data does not possess any radial or size restrictions.
\item
Our theory differs from the results presented in the works of Burq-Thomann and Latocca's results \cite{BT20scattering,Lat20scattering}. Firstly, we consider a ``large'' subset of the $L_x^2$, while in their theory, the initial data is almost in $L_x^2$ but not precisely in $L_x^2$ itself. Secondly, our proof framework shares similarities with Theorems \ref{thm:frac-weighted-3d} and \ref{thm:frac-weighted-dd}, while the results in \cite{BT20scattering,Lat20scattering} are based on the lens transform and the quasi-invariant measure method.
\item
Our result can be compared to the previous almost sure wave operator results in \cite{Mur19ProAMS,NY19MRL}. They employed the Wiener randomization in physical space, while we utilize a ``narrowed'' Wiener randomization in physical space. Moreover, their result established the uniqueness of initial data, allowing for the definition of the wave operator. In contrast, our method only provides the existence of final data, which prevents us from verifying the inverse wave operator.
\end{enumerate}
\end{remark}

Now, we give a detailed explanation of the difference in initial data between our result and the previous almost sure scattering result in \cite{BT20scattering,Lat20scattering}. In Theorem \ref{thm:main}, the initial data is a "large" subset of the $L_x^2$ space without any decay assumption. In other words, the randomized data does not belong to any space $\F H^\ep$, $\ep>0$ or $L_x^p$, $p\ne 2$. Moreover, for each given $f\in L^2$, let $\mu_f$ be the image measure on $L^2$ of the mapping $\om\mapsto f^\om$, defined as:
\EQ{
	\mu_f:=\PP\circ (f^\om)^{-1}.
}
Naturally, $\mu_f(L^2)=1$ , and we consider the initial data on the support of $\mu_f$ for each $f\in L^2$.

In the theory in \cite{BT20scattering,Lat20scattering}, the initial data is essentially in $L_x^2$ up to logarithmic divergence in terms of the harmonic oscillator. More specifically, they consider the initial data on the support of a measure $\mu_0$ that satisfies:
\EQ{
\mu_0(\cap_{\ep>0} \HH^{-\ep})=1 \text{, and }\mu_0(L^2)=0.
}
The space $\HH^{-s}$ with $s\goe 0$ denotes the dual space of $\HH^s$, defined as:
\EQ{
\HH^{s} := \fbrk{f\in L^2:\jb{\nabla}^s f\in L^2\text{ and } \jb{x}^s f\in L^2}.
}

\subsection{Main ideas}
In this section, we give a general description of our method.
\subsubsection{Main framework: the adaptation of Tsutsumi-Yajima's argument} Our proofs for the fractional weighted scattering in Theorems \ref{thm:frac-weighted-3d} and \ref{thm:frac-weighted-dd}, and the almost sure scattering in Theorem \ref{thm:main} are all based on the combination of linear-nonlinear decomposition (which is usually referred as Bourgain's method or Da Prato-Debousche's method) and Tsutsumi-Yajima's argument. The common point is that the initial data does not belong to $\F H_x^{1}$, thus the usual pseudo-conformal energy fails.

After applying the pseudo-conformal transform, it suffices to consider the equation
\EQ{
i\pd_t \U + \frac12 \De \U = t^{\frac{dp}{2}-2}|\U|^p\U,
}
with final data $\U_+=\F^{-1}\wb u_0$. Moreover, this is an $L_x^2$-supercritical problem, and the weighted condition of $\F H^s$ is changed to derivative condition of $H^s$.

We decompose $\U=\V+\W$ with $\V$ being some linear solution, and $\W$ solves the nonlinear equation with the forced term $\V$. Even though $\nabla\V\notin L_x^2$, we can still prove the almost conservation law of the pseudo-conformal energy:
\EQ{
	\E(t) = \frac{1}{4}t^{2-\frac{dp}{2}}\int_{\R^d}|\nabla\W(t,x)|^2 \dx + \frac{1}{p+2} \int_{\R^d}|\U(t,x)|^{p+2} \dx,
}
and then the scattering follows by some adaptation of Tsutsumi-Yajima's argument.

For the forward-in-time case, the proof is divided in two parts. Firstly, we prove $\sup_{t\in[T_0,+\I)}\norm{\W(t)}_{\dot H^1}<+\I$, which is actually a local problem. Secondly, the main part of proof is the $(0,T_0]$ case, which relies on the almost conservation of $\E$. The main ingredient of the energy increment is 
\EQ{
\int_0^{T_0}\int_{\R^d} \nabla\W \cdot \nabla\V \W^p \dd x\dd t.
}
Then, it suffices to reduce the derivative with respect to $\nabla\V$ and close the energy estimate. Note that the $\dot H^1$-bound creates some singularity near $t=0$, namely
\EQ{
\norm{\nabla\W(t)}_{L_x^2} \lsm t^{\frac d4p-1} \E^{1/2},
}
then an additional main task is to remove this singularity from the time integral in the energy increment. Next, we state the main ingredients of the proof for Theorems \ref{thm:frac-weighted-3d} and \ref{thm:frac-weighted-dd}, as well as Theorem \ref{thm:main}.

\subsubsection{Main ingredient of the fractional weighted scattering in 3D case}
To improve the weighted condition, we utilize a combination of tools as outlined below.
\begin{enumerate}
\item (``High-low" decomposition). The high-low frequency decomposition is introduced by Bourgain \cite{Bou98IMRN}. Here, we apply the similar idea to the physical space.
\item
(Localized interaction Morawetz estimate). We have derived a mass subcritical spacetime estimate based on the almost pseudo-conformal conservation law, which differs from the usual interaction Morawetz estimate in \cite{CKSTT04CPAM}. Despite the presence of a time singularity near the origin, the key advantage of this estimate is its relatively small energy increase. This property makes it particularly suitable for handling the case where $p$ is close to $4/d$. To quantify the singularity, we consider the spacetime estimate over dyadic temporal intervals. In the original variable, the spacetime estimate can be roughly expressed as follows:
\EQ{
\normb{t^{\frac3{16}p}u}_{L_{t,x}^4([2^k,2^{k+1}]\times\R^3)} \lsm \E^{1/8},\quad\text{for any k}\ge0.
}
This estimate exhibits the same scaling as $L_t^\I \dot H_x^{\frac{2-3p}{8}}$ and demonstrates a significantly lower energy increase compared to all other known estimates. For more detailed information, please refer to Proposition \ref{prop:interaction-morawetz}.
\item
(Various spacetime estimate). To decrease the derivative of $\nabla\V$ in the energy increment, we need to apply the bilinear Strichartz estimate on $(0,T_0]$. This necessitates the use of dual Strichartz estimates for the nonlinear term $t^{\frac{3p}{2}-2}|\U|^p\U$. We observe that the dual $L_t^1L_x^2$-estimate exhibits better time decay near the origin, making it more suitable for the case when $p$ is close to $2/3$. Additionally, the dual $L_t^{2}L_x^{6/5}$-one gives better energy increase of $N_0$ and fewer loss in derivatives, providing an advantage for dealing with the case when $p$ is near $4/3$. We also invoke the aforementioned interaction Morawetz estimate into the proof of $L_t^{2}L_x^{6/5}$-estimate to further reduce the energy increase of $N_0$. Consequently, these two estimates are used in different cases of $p$.
\end{enumerate}

\subsubsection{Ideas of the fractional weighted scattering in general dimensional cases}
In addition, we intend to address more general cases, without pursuing a favorable weighted condition. However, we cannot find a universal method that can handle all dimensions in the whole range when $\frac2d<p<\frac4d$. Instead, it requires the use of various distinct ways to prove the main result. Notably, there exists a key difference between the proofs for the cases when $p<\frac{2}{d-2}$ and $p\goe \frac{2}{d-2}$. 

$\bullet$ \textit{Main idea when $p<\frac{2}{d-2}$.} We consider the limit case when $s=1$. The main term in the energy increment can be bounded as follows:
\EQ{
\int_0^T \int_{\R^d} \nabla\W \cdot \nabla\V \W^p \dd x\dd t \lsm \int_0^T \norm{\nabla\W}_{L_x^2}  \norm{\nabla\V}_{L_x^{\frac{2d}{d-2}}} \norm{\W}_{L_x^{dp}}^p\dd t.
}
Here, the natural choice for $\nabla\W$ is $L_x^2$, and for $\norm{\nabla\V}_{L_t^q L_x^r}$, we have $r\loe \frac{2d}{d-2}$. Therefore, the estimate for $\W$ is at least $L_x^{dp}$, which cannot exceed the Sobolev critical exponent $\frac{2d}{d-2}$. This leads to the restriction $p<\frac{2}{d-2}$. Since $\frac{2}{d-2}\goe \frac4d$ when $1\loe d\loe 4$, we are able to cover the entire range $\frac2d<p<\frac4d$ when $1\loe d\loe 4$, and $\frac2d<p<\frac{2}{d-2}$ when $d\goe 5$. 

To further reduce the weighted condition, we can apply the bilinear Strichartz estimate to handle the high-low interaction part:
\EQ{
& \int_0^T \int_{\R^d} \nabla\W \cdot \nabla\V_{hi} \W_{low}^p \dd x\dd t \\
\lsm & \norm{\nabla\W}_{L_t^\I L_x^2} \norm{\nabla\V_{hi}\W_{low}}_{L_{t,x}^2}^{\th} \norm{\nabla\V_{hi}}_{L_t^2 L_x^{\frac{2d}{d-2}}}^{1-\th} \norm{\W_{low}}_{L_t^{q}L_x^r}^{p-\th}, 
}
where $\th$ is some suitable parameter, and $(q,r)$ is some perturbation of the exponent $(2p,dp)$. This approach reduces the derivative of $\nabla \V$ and hence decreases the weighted condition less than $1$.

Additionally, we note that the proof in the 1D case slightly differs from the $d\ge 2$ case. The main reason is that in the 1D case, the Strichartz norm $L_t^q L_x^r$ requires $q\ge4$, and that $p$ may less than $1$ when $d\ge 2$. Therefore, we divide the proof into two subcases, see  Section \ref{sec:fra-energy-dd-1} and Section \ref{sec:fra-energy-dd-2}.

$\bullet$ \textit{Main idea when $p\goe\frac{2}{d-2}$.} In this case, we restrict the dimension to $d\goe 5$, which requires the use of some global spacetime estimates. Heuristically, we also examine the case when $s=1$. By Sobolev's inequality, for $\frac{2}{d-2}\loe p <\frac{2}{d-4}$,
\EQ{
\int_0^T \int_{\R^d} \nabla\W \cdot \nabla\V \W^p \dd x\dd t \lsm &  \int_0^T \norm{\nabla\W}_{L_x^2}  \norm{\nabla\V}_{L_x^{\frac{2d}{d-2}}} \norm{\W}_{L_x^{dp}}^p\dd t \\
\lsm & \int_0^T \norm{\nabla\W}_{L_x^2}  \norm{\nabla\V}_{L_x^{\frac{2d}{d-2}}} \norm{\jb{\nabla}\W}_{L_x^{\frac{2d}{d-2}}}^p\dd t.
}
Next, applying the equation of $\W$,
\EQ{
\norm{\jb{\nabla}\W}_{L_t^2 L_x^{\frac{2d}{d-2}}} \lsm & \norm{\jb{\nabla}\W(T)}_{L_x^2} + \normb{t^{\frac{3p}{2}-2}\nabla(|\U|^p\U)}_{L_t^2 L_x^{\frac{2d}{d+2}}} \\
\lsm & \norm{\jb{\nabla}\W(T)}_{L_x^2} + \normb{t^{\frac{3p}{2}-2}\nabla(|\W|^p\W)}_{L_t^2 L_x^{\frac{2d}{d+2}}} + \text{easier terms} \\
\lsm & \norm{\jb{\nabla}\W(T)}_{L_x^2} + \norm{\nabla\W}_{L_t^2 L_x^{\frac{2d}{d-2}}}^{p} \normb{t^{\frac{3p}{2}-2}\norm{\nabla\W}_{L_x^{2}}}_{L_t^{\frac{2}{1-p}}} + \text{easier terms}.
}
Since $d\goe 5$ and $p<\frac4d$, we have $p<1$. Therefore, we are able to obtain the estimate for $\norm{\jb{\nabla}\W}_{L_t^2 L_x^{\frac{2d}{d-2}}}$ and close the energy increment. Similar to the case when $p<\frac{2}{d-2}$, we also invoke the bilinear Strichartz estimate to reduce the weighted condition. This approach allows us to extend  to the case when $p>\frac{2}{d-4}$. Note that $\frac{2}{d-4}\goe\frac4d$ when $d\loe 8$, which means our method can cover the range $\frac2d<p<\frac4d$ when $5\loe d\loe 8$. However, this method only works in the range $\frac{2}{d-2}\loe p< \frac{2}{d-4}$ when $d\loe 11$, since the derivation of global spacetime estimate brings more temporal singularity near the origin and more energy increase of $N_0$. 

Moreover, the proof differs between $d\loe 7$ and $d\goe 8$. As we can see from the discussion above, the Sobolev's inequality $\norm{\W}_{L_x^{dp}} \lsm \norm{\jb{\nabla}\W}_{L_x^{\frac{2d}{d-2}}}$ plays a crucial role. However, when $d\loe 7$, $p$ cannot reach the upper bound $\frac{2}{d-4}$, specifically $p<\frac4d<\frac{2}{d-4}$. Then, applying the above Sobolev inequality leads to poor time singularity. Therefore, we apply two different Sobolev's inequality
\EQ{
\norm{\W}_{L_x^{dp}} \lsm \normb{\jb{\nabla}^{\frac{1-p}{p}}\W}_{L_x^{\frac{2d}{d-2}}}\text{, when }5\loe d\loe 7,
}
and
\EQ{
\norm{\W}_{L_x^{dp}} \lsm \normb{\jb{\nabla}^{1-}\W}_{L_x^{\frac{2d}{d-2}}}\text{, when }8\loe d\loe 11.
}
These inequalities effectively handle the temporal singularity and energy increase simultaneously, see Lemma \ref{lem:fra-energy-global-spacetime-estimate-sobolev} below.

\subsubsection{Ideas of almost sure scattering} The proof framework of almost sure scattering is similar to the fractional scattering. However, there are some new aspects specific to the randomized data problem, which enable us to obtain a scattering result without any decaying condition. 

$\bullet$ \textit{``Narrowed"-scale physical randomization.} 
To totally remove the weighted condition, we introduce ``narrowed"-scale physical space randomization. The aim is to give some improved global $L^2$-subcritical spacetime estimates, considering the initial data being merely in $L^2$, compared to the deterministic setting. 

The improvement of space-time estimate for the linear flow can be observed in two aspects. First, by the dispersive inequality, we can obtain for some $L^2$ subcritical exponents $(q,r)$, 
\EQ{
&\normb{(x+it\nabla)e^{\frac12it\De}f^\om}_{L_t^q L_x^r}<+\I\text{, a. e. }\om\in\Om.}
Second, using the observation that applying the pseudo-conformal transform to a linear flow is essentially equivalent to the Fourier transform, namely 
\EQ{
\TT e^{\frac12it\De} f = e^{\frac12it\De} \F^{-1}\wb f,
}
where $\TT$ denotes the pseudo-conformal transform, we can deduce that for some $L^2$ supercritical exponents $(q,r)$, 
\EQ{
&\normb{\Delta \mathcal T(e^{\frac12it\De}f^\om)}_{L_t^q L_x^r}<+\I\text{, a. e. }\om\in\Om.
}
The above estimates holds for $f$ merely in $L_x^2$. Therefore, the improvement is one of the key elements in proving the scattering with non-decaying data.

$\bullet$ \textit{The adaptation of the Tsutsumi-Yajima's argument.} We also need to apply the Tsutsumi-Yajima's argument in the probabilistic setting. More precisely, we consider $v=S(t)f^\om$ as the linear part and $w=u-v$ be the nonlinear part of the solution $u$. Furthermore, we introduce the $1$-measure space
\EQ{
	\fbrkb{\om\in\Om:\norm{f^\om}_{L_x^2} + \norm{v}_X <+\I},
}
where $X$ is some required space-time norm. Then, we consider each $\om$ separately, and prove the scattering for the given initial data $f^\om$.

\subsection{Organization of the paper} In Section \ref{sec:preliminary}, we first give some notation and lemmas. In Section \ref{sec:scattering-criterion}, we give a general scattering criterion that will be applied to proof of Theorems \ref{thm:frac-weighted-3d}, \ref{thm:frac-weighted-dd}, and \ref{thm:main}. Then, we prove Theorem \ref{thm:frac-weighted-3d} in Section \ref{sec:fractional}, Theorem \ref{thm:frac-weighted-dd} in Section \ref{sec:fractional-dd}, and Theorem \ref{thm:main} in Section \ref{sec:almost-sure-scattering}.

\vskip 1.5cm
\section{Preliminary}\label{sec:preliminary}
\vskip .5cm

\subsection{Notation}
For any $a\in\R$, $a\pm:=a\pm\epsilon$ for arbitrary small $\epsilon>0$. For any $z\in\C$, we define $\re z$ and $\im z$ as the real and imaginary part of $z$, respectively. For any $x\in\R^d$, we denote that $\jb{x}:=(1+|x|^2)^{1/2}$.

$C>0$ represents some constant that may vary from line to line. We write $C(a)>0$ for some constant depending on coefficient $a$. If $f\loe C g$, we write $f\lsm g$. If $f\loe C g$ and $g\loe C f$, we write $f\sim g$. Suppose further that $C=C(a)$ depends on $a$, then we write $f\lsm_a g$ and $f\sim_a g$, respectively. 

For complex-valued functions $f$ and $g$, we denote $O(fg)$ as the linear combination:
\EQ{
O(fg) := c_1 fg + c_2 f\wb g+c_3 \wb fg+ c_4 \wb f\wb g,
}
for some $c_i\in\C$, $i=1,2,3,4$. For the functions $f$, $g$ and $h$, $O(fgh)$ is defined similarly.

We write ``a.e. $\om\in\Om$'' to stand for ``almost every $\om\in\Om$''.

We use $\wh f$ or $\F f$ to denote the Fourier transform of $f$:
\EQ{
	\wh f(\xi)=\F f(\xi):= \rev{(2\pi)^{d/2}}\int_{\R^d} e^{-ix\cdot\xi}f(x)\rm dx.
}
We also define
\EQ{
	\F^{-1} g(x):= \rev{(2\pi)^{d/2}}\int_{\R^d} e^{ix\cdot\xi}g(\xi)\rm d\xi.
}
Using the Fourier transform, we can define the fractional derivative $\abs{\nabla} := \F^{-1}|\xi|\F $ and $\abs{\nabla}^s:=\F^{-1}|\xi|^s\F $.  

We also need the usual inhomogeneous dyadic Littlewood-Paley decomposition. Take a cut-off function $\phi\in C_{0}^{\infty}(\R)$ such that $\phi(r)=1$ if $0\loe r\loe1$ and $\phi(r)=0$ if $r>2$. For dyadic $N\in 2^\N$, let $\phi_{N}(r) := \phi(N^{-1}r)$. Then, we define
\EQ{
	\ph_1(r):=\phi(r)\text{, and }\ph_N(r):=\phi_N(r)-\phi_{N/2}(r)\text{, for any }N\goe 2.
}
We define the inhomogeneous Littlewood-Paley dyadic operator: for any $N\in2^\N$,
\EQ{
	f_{N}= P_N f := \mathcal{F}^{-1}\brko{ \ph_N(|\xi|) \hat{f}(\xi)}.
}
Then, by definition, we have $f=\sum_{N\in2^\N}f_N$. Moreover, we also need the following: for any $N\in2^\N$,
\EQ{
	f_{\loe N}=P_{\loe N} f := &\mathcal{F}^{-1}\brko{ \phi_{N}(|\xi|) \hat{f}(\xi)}, \\
	f_{\ll N}=P_{\ll N} f :=& \mathcal{F}^{-1}\brko{\phi_{N}(2^{5}|\xi|) \hat{f}(\xi)}, \\
	f_{\lsm N}=P_{\lsm N} f :=& \mathcal{F}^{-1}\brko{ \phi_{N}(2^{-5}|\xi|) \hat{f}(\xi)},
}
and
\EQ{
	f_{\sim N}=P_{\sim N}f:= f_{\lsm N}-f_{\ll N}.
}
We also denote that $f_{\goe N}=P_{\goe N} f := f- P_{\loe N} f$, $f_{\gg N}=P_{\gg N} f:=f-P_{\lsm N}f$, and $f_{\gsm N}=P_{\gsm N}f:= f-P_{\ll N}f$.

Given $1\loe p \loe \I$, $L^p(\R^d)$ denotes the usual Lebesgue space. We also define $\jb{\cdot,\cdot}$ as real $L^2$ inner product:
\EQ{
	\jb{f,g} := \re\int f(x)\wb{g}(x)\dx.
}

For any $1\loe p <\I$, define $l_N^p=l_{N\in 2^\N}^p$ by its norm
\EQ{
	\norm{c_N}_{l_{N\in 2^\N}^p}^p:=\sum_{N\in 2^\N}|c_N|^p.
}
The space $l_N^\I=l_{N\in2^\N}^\I$ is defined by $\norm{c_N}_{l_{N\in 2^\N}^\I}:=\sup_{N\in 2^\N}|c_N|$. In this paper, we use the following abbreviations
\EQ{
	\sum_{N:N\loe N_1}:=\sum_{N\in2^\N:N\loe N_1}\text{, and }\sum_{N\loe N_1}:=\sum_{N,N_1\in2^\N:N\loe N_1}.
}

We then define the vector-valued norms. For $1\loe q< \I$, $1\loe r\loe \I$, and the function $u(t,x)$, we define
\EQ{
	\norm{u}_{L_t^q L_x^r(\R\times \R^d)}^q:= \int_{\R}\norm{u(t,\cdot)}_{L_x^r}^q\dd t,
}
and for the function $u_N(x)$, we define
\EQ{
	\norm{u_N}_{l_N^q L_x^r(2^\N\times \R^d)}^q:= \sum_{N}\norm{u_N(\cdot)}_{L_x^r}^q.
}
The $q=\I$ case can be defined similarly.

For any $0\loe\gamma\loe1$, we call that the exponent pair $(q,r)\in\R^2$ is $\dot H^\ga$-$admissible$, if $\frac{2}{q}+\frac{d}{r}=\half d-\ga$, $2\loe q\loe\I$, $2\loe r\loe\I$, and $(q,r,d)\ne(2,\I,2)$. If $\ga=0$, we say that $(q,r)$ is $L^2$-$admissible$. 
For any time interval $I\subset\R$, we define $S^0(I)$ by
\EQn{\label{defn:s0}
\norm{u}_{S^0(I)} := \sup\fbrkb{\norm{u}_{L_t^q L_x^r(I\times\R^d)}:(q,r)\text{ is }L^2\text{-admissible}}.
}

Next, we recall some useful basic lemmas.
\begin{lem}[Schur's test]\label{lem:schurtest}
For any $a>0$, let sequences $\fbrk{a_N}$,  $\fbrk{b_N}\in l_{N\in2^\N}^2$, then we have
\EQ{
\sum_{N_1\loe N} \brkb{\frac{N_1}{N}}^a a_N b_{N_1} \lsm \norm{a_N}_{l_N^2} \norm{b_N}_{l_N^2}.
}
\end{lem}
\begin{lem}[Fractional chain rule,\cite{CW91JFA}]\label{lem:frac-chain}
	Suppose that $G\in C^1(\C)$, $s\in(0,1]$, and $1<p,p_1,p_2<\I$ are such that $\rev p=\rev{p_1}+\rev{p_2}$. Then,
	\EQ{
		\norm{\abs{\nabla}^sG(u)}_{L_x^p}\lsm\norm{G'(u)}_{L_x^{p_1}}\norm{\abs{\nabla}^su}_{L_x^{p_2}}.
	}
\end{lem}
Next, we also recall the sharp Gagliardo-Nirenberg type inequality, which was initially introduced in \cite{Nag41,Wei83CMP}. Although the implicit constant can be expressed in terms of the ground state, we will omit it here as it is not necessary for our purposes.
\begin{lem}[Sharp Gagliardo-Nirenberg inequality]\label{lem:GN}
Let $d\goe 1$. Suppose that $0<p<+\I$ if $d=1,2$, and $0<p<\frac{4}{d-2}$ if $d\ge3$. Then, we have
\EQ{
\norm{f}_{L_x^{p+2}(\R^d)}^{p+2} \lsm \norm{f}_{L_x^2(\R^d)}^{p+2-\frac{pd}{2}} \norm{\nabla f}_{L_x^2(\R^d)}^{\frac{dp}{2}}.
}
\end{lem}
\subsection{Schr\"odinger equations}
We denote that $e^{\frac12it\De}u_0$ as the solution of 
\EQ{
\left\{ \aligned
&i\pd_t u + \frac12\De u = 0, \\
& u(0,x) = u_0,
\endaligned
\right.
}
and we denote that $S(t):=e^{\frac12it\De}$. Then, we have the explicit formula
\EQn{\label{linear-explicit}
S(t)u_0 = \frac{1}{(2\pi it)^{d/2}} \int_{\R^d} e^{i\frac{|x-y|^2}{2t}} u_0(y)\dy.
}
For the Schr\"odinger equation, we have the following classical Strichartz estimates:
\begin{lem}[Strichartz estimate, \cite{KT98AJM}]\label{lem:strichartz}
	Let $I\subset \R$. Suppose that $(q,r)$ and $(\wt{q},\wt{r})$ are $L_x^2$-admissible.
	Then,
	\EQ{
		\norm{S(t)\ph}_{L_t^qL_x^r(\R\times\R^d)} \lsm \norm{\ph}_{L_x^2},
	}
	and
	\EQ{
		\normb{\int_{t_0}^TS(t-\ta) F(\ta)\dd\ta}_{L_t^qL_x^r(\R\times \R^d)} \lsm \norm{F}_{L_t^{\wt{q}'} L_x^{\wt{r}'}(\R\times\R^d)}.
	}
\end{lem}

The bilinear Strichartz estimate was first introduced by Bourgain \cite{Bou98IMRN}, and further extended in \cite{CKSTT08Annals,Vis07Duke}, when $q=r=2$. The $q,r<2$ case was referred to bilinear restriction estimates for paraboloid, first obtained by Tao \cite{Tao03GAFA}, based on the method developed by Wolff \cite{Wol01Annals}. We will frequently use the version of bi-linear estimate for general functions, see  \cite{Can19MathAnn,Vis07Duke}.
\begin{lem}[Bilinear Strichartz estimate]\label{lem:bilinearstrichartz}
Let $I\subset\R$, $1\loe q,r \loe 2$, $\rev q + \frac{d-1}{r}<d-1$, and  $M,N\in2^\N$ satisfy $M\ll N$. Let also $(\wt q,\wt r)$ be some $L^2$-admissible pair with $\wt q\ne 2$. Suppose that for any $t\in I$,  $\wh{u}(t,\cdot)$ is supported on $\fbrk{\xi:\abs{\xi}\sim N}$, and $\wh{v}(t,\cdot)$ is supported on $\fbrk{\xi:\abs{\xi}\sim M}$. Then,
\EQn{\label{eq:bilinearstrichartz}
\norm{uv}_{L_t^qL_x^r(I\times\R^d)} \lsm &  \frac{M^{d+1-\frac{d+1}{r}-\frac{2}{q}}}{N^{1-\rev r}}\norm{u}_{S^*(I)} \norm{v}_{S^*(I)},
}
where $a\in I$, and
\EQn{
\norm{u}_{S^*(I)}:=\norm{u(a)}_{L_x^2}+\normb{(i\pd_t \pm\frac12\De)u}_{L_t^{\wt q'}L_x^{\wt r'}(I\times\R^d)}.
}
\end{lem}

Next, we recall the standard $L_x^2$ well-posedness theory of $L_x^2$-subcritical NLS, which was first obtained by Tsutsumi \cite{Tsu87Funk}. 
\begin{lem}\label{lem:local-l2}
Let $d\goe 1$, $0<p<\frac 4d$ and $\mu=\pm 1$. Assume that the initial data $u_0\in L_x^2(\R^d)$. Then, we have the following conclusions.
\begin{enumerate}
\item (Local well-posedness)
There exists $T=T(\norm{u_0}_{L_x^2(\R^d)})>0$ such that \eqref{eq:nls} admits a unique solution $u\in C_t([0,T];L_x^2(\R^d)) \cap L_{t,x}^{\frac{2(d+2)}{d}}([0,T]\times\R^d)$. Moreover,
\EQ{
\norm{u}_{S^0([0,T])} \lsm \norm{u_0}_{L_x^2(\R^d)}.
}
\item (Global well-posedness)
There exists a unique solution $u\in C(\R;L_x^2(\R^d))$ to \eqref{eq:nls}.
\end{enumerate}
\end{lem}

\subsection{Pseudo-conformal transform}
Now, we recall some notation and the properties for the pseudo-conformal transform. 
We define the pseudo-conformal  transform $\mathcal T$ to be 
\EQn{\label{def:PC-transform}
	\mathcal T f(t,x) := \frac{1}{(it)^{d/2}} e^{\frac{i|x|^2}{2t}} \bar f(\frac{1}{t},\frac{x}{t}).
}
Then we have  that  
$$
\mathcal T^{-1}=\mathcal T.
$$
If $u$ is the solution of \eqref{eq:nls}, then $\U=\TT u$ solves
\EQn{\label{eq:nls-pc}
i\pd_t \U + \frac12 \De \U = t^{\frac{dp}{2}-2}|\U|^p\U.
}
Now, we define
\EQ{
	M(t) := e^{\frac{i|x|^2}{2t}}.
}
Then, we define the vector field
\EQ{
J(t):=x+it\nabla,
}
which also has following presentation
\EQ{
J(t)=S(t)xS(-t)=M(t)(it\nabla)M(-t).
}
Then, we define the fractional vector field for any $s>0$,
\EQ{
|J(t)|^{s}:= S(t)|x|^sS(-t).
}
Moreover, we also have
\EQ{
|J(t)|^{s}=M(t)|t\nabla|^sM(-t).
}
We usually omit $t$ and write $J,|J|^s$ for short.

Then, we gather the following classical properties, which can be checked directly. 
\begin{lem}\label{lem:J-T} The following statements hold:
\begin{itemize}
\item[(1)] $\nabla \mathcal T=i \mathcal T J$ and $|\nabla|^s \mathcal T= \mathcal T |J|^s$.
\item[(2)] $J(i\partial_t+\frac12\Delta)=(i\partial_t+\frac12\Delta)J$, and the same property also holds for $|J|^s$.
\item[(3)] $|J(|u|^pu)|\lsm_p |u|^p|Ju|$.
\item[(4)]
$\TT S(t) f = S(t)\F^{-1}\wb f$.
\end{itemize}
\end{lem}

As a direct consequence of Lemma \ref{lem:J-T} (1) and the definition \eqref{def:PC-transform}, for any function $f(t,x)\in L_t^\I L_x^2$ and any $t\in\R\backslash\fbrk{0}$,
\begin{align}\label{H1-T}
\|\mathcal Tf(t)\|_{L^2_x}=\big\|f(\frac1t)\big\|_{L^2_x}\text{, and }
\|\nabla \mathcal Tf(t)\|_{L^2_x}=\big\|Jf(\frac1t)\big\|_{L^2_x}.
\end{align}
Finally, we give a general fact for spacetime norms under the pseudo-conformal transform: 
\begin{remark}\label{rem:spacetime-exponent-transform}
Let $u(t,x)$ be a function and
\EQ{
\U(s,y)=\TT u(s,y)=  \frac{1}{(is)^{d/2}} e^{\frac{i|y|^2}{2s}} \wb u(\frac{1}{s},\frac{y}{s}).
}
Then, for any $1\le q,r\le\I$, we have
\EQ{
\norm{u}_{L_t^q L_x^r(\R\times\R^d)} = \normb{s^{\frac d2-\frac2q-\frac dr}\U(s,y)}_{L_s^q L_y^r(\R\times\R^d)}.
}
Particularly, for any $\de>0$, we have
\EQ{
\norm{u}_{S^0((0,\de])} = \norm{\U}_{S^0([1/\de,+\I))}\text{, and }\norm{|J|^su}_{S^0((0,\de])} \sim \norm{\abs{\nabla}^s\U}_{S^0([1/\de,+\I))}.
}
\end{remark}
\subsection{Probabilistic theory}
We recall the large deviation estimate, which holds for the random variable sequence $\fbrk{\re g_k,\im g_k}$ in the Definition \ref{defn:randomization}.
\begin{lem}[Large deviation estimate, \cite{BT08inventI}]\label{lem:large-deviation}
	Let $(\Om, \A, \PP)$ be a probability space. Let $\fbrk{g_n}_{n\in\N^+}$ be a sequence of real-valued, independent, zero-mean random variables with associated distributions $\fbrk{\mu_n}_{n\in\N^+}$ on $\Om$. Suppose $\fbrk{\mu_n}_{n\in\N^+}$ satisfies that there exists $c>0$ such that for all $\ga\in\R$ and $n\in\N^+$
	\EQ{
		\absb{\int_{\R}e^{\ga x}\mathrm d \mu_n(x)}\loe e^{c\ga^2},
	}
	then there exists $\al>0$ such that for any $\la>0$ and any complex-valued sequence $\fbrk{c_n}_{n\in\N^+}\in l_n^2$, we have
	\EQ{
		\PP\brkb{\fbrkb{\om:\absb{\sum_{n=1}^\I c_n g_n(\om)}>\la}}\loe 2\exp\fbrkb{-\al\la\norm{c_n}_{l_n^2}^{-2}}.
	}
	Furthermore, there exists $C>0$ such that for any $2\loe \rho<\I$ and complex-valued sequence $\fbrk{c_n}_{n\in\N^+}\in l_n^2$, we have
	\EQ{
		\normb{\sum_{n=1}^\I c_n g_n(\om)}_{L_\om^\rho(\Om)} \loe C\sqrt{\rho} \norm{c_n}_{l_n^2}.
	}
\end{lem}
Then, we recall the following lemma in \cite{LM14CPDE}, which can be proved by the method in \cite{Tzv10gibbs}.
\begin{lem}\label{lem:probability-estimate}
	Let $F$ be a real-valued measurable function on a probability space $(\Om, \A, \PP)$. Suppose that there exists $C_0>0$, $K>0$ and $\rho_0\goe 1$ such that for any $\rho\goe \rho_0$, we have
	\EQ{
		\norm{F}_{L_\om^\rho(\Om)} \loe \sqrt{\rho} C_0 K.
	}
	Then, there exist $c>0$ and $C_1>0$, depending on $C_0$ and $p_0$ but independent of $K$, such that for any $\la>0$,
	\EQ{
		\PP\brkb{\fbrkb{\om\in\Om:|F(\om)|>\la}} \loe C_1 e^{-c\la^2K^{-2}}.
	}
	Particularly, we have
	\EQ{
		\PP\brkb{\fbrkb{\om\in\Om:|F(\om)|<\I}}=1.
	}
\end{lem}

\vskip 1.5cm
\section{A general scattering criterion}\label{sec:scattering-criterion}
\vskip .5cm

In this section, we give a scattering criterion when the pseudo-conformal conservation law is not available, which will be applied to both the fractional weighted problem and the randomized data problem. The proof is basically a modification of Tsutsumi-Yajima's argument in \cite{TY84BAMS}.
\begin{prop}[Scattering criterion]\label{prop:scattering-criterion}
Suppose that $d\goe 1$, $\mu=1$, and $\frac2d< p < \frac4d$. Let $u$ be the solution to \eqref{eq:nls}, and denote that $u_0=v_0+w_0$, $v=S(t)v_0$, and $w=u-v$.  Assume that for some $0\loe s<1$ and $T_0>0$,
\EQn{\label{eq:scattering-criterion-assumption-initial}
\jb{x}^sv_0 \in L_x^2(\R^d),\quad \jb{x}w_0 \in L_x^2(\R^d),
}
and
\EQn{\label{eq:scattering-criterion-assumption-energy}
\sup_{0<t\le T_0}\fbrk{t^{1-\frac{dp}{4}}\norm{\TT w(t)}_{\dot H_x^1} + \norm{\TT w(t)}_{L_x^{p+2}}+ \norm{\TT v(t)}_{L_x^{p+2}} }  < C,
}
where the constant $C$ may depend on $T_0$, $\norm{\jb{x}^sv_0}_{L_x^2}$, and $\norm{\jb{x}w_0}_{L_x^2}$. Then, there exists a unique solution $u\in C(\R;L_x^2(\R^d))$ to the equation \eqref{eq:nls}. Moreover, the solution scatters forward in $L_x^2$, namely there exists $u_+\in L_x^2(\R^d)$ such that
\EQ{
\lim_{t\ra+\I}\norm{u(t)-S(t)u_+}_{L_x^2(\R^d)} =0.
}
\end{prop}
\begin{proof}
By $u_0=w_0+v_0\in L_x^2(\R^d)$ and Lemma \ref{lem:local-l2}, we have the global well-posedness of \eqref{eq:nls}, then it suffices to prove the scattering. We apply the pseudo-conformal transform in \ref{def:PC-transform}, 
\EQ{
\U:=\TT u, \quad \W:=\TT w,\quad\text{and}\quad\V:=\TT v.
}
Then, $\U$ satisfies \eqref{eq:nls-pc}, and $\W$ satisfies the equation
\EQ{
i\pd_t \W + \frac12 \De \W = t^{\frac{dp}{2}-2}|\U|^p\U,
}
with asymptotic condition:
\EQ{
\lim_{t\ra+\I}\norm{\W(t)-S(t)\F^{-1}\wb w_0}_{L_x^2}=0.
}
By the assumption \eqref{eq:scattering-criterion-assumption-energy}, we have for any $0<t<T_0$,
\EQn{\label{esti:scattering-criterion-assumption-energy}
\norm{\nabla \W(t)}_{L_x^2(\R^d)}\lsm t^{\frac{dp}{4}-1}\text{, and } \norm{\W(t)}_{L_x^{p+2}(\R^d)}+ \norm{\U(t)}_{L_x^{p+2}(\R^d)} \lsm 1.
}
	
Using the equation of $\W$, for any $\ph \in H_x^1(\R^d)$ and $0<t_2<t_1$,
\EQ{
\jb{\W(t_1)-\W(t_2),\ph} = & -i\int_{t_2}^{t_1} \jb{\nabla_x \W(\ta), \nabla_x \ph} \dd\ta \\
& -i\int_{t_2}^{t_1} \ta^{\frac{dp}{2}-2}\jb{|\U|^p\U,\ph}\dd\ta.
}
Then, by $H^1(\R^d) \subset L^{p+2}(\R^d)$ and \eqref{esti:scattering-criterion-assumption-energy},
\EQ{
\abs{\jb{\W(t_1)-\W(t_2),\ph}} \lsm &  \brkb{\norm{\ph}_{H_x^1}\absb{t_1^{\frac{dp}{4}}-t_2^{\frac{dp}{4}}} + \norm{\ph}_{L_x^{p+2}} \absb{t_1^{\frac{dp}{2}-1}-t_2^{\frac{dp}{2}-1}}}\\
\lsm &  \norm{\ph}_{H_x^1} \brkb{\absb{t_1^{\frac{dp}{4}}-t_2^{\frac{dp}{4}}} + \absb{t_1^{\frac{dp}{2}-1}-t_2^{\frac{dp}{2}-1}}}.
}
Since $p>\frac2d$, this implies that $\fbrk{\W(t)}$ is weakly convergent in $H_x^{-1}(\R^d)$ when $t\ra 0$. Using the density argument, there exists $\W_0\in L_x^2(\R^d)$, such that
\EQ{
\W(t)\wra \W_0\text{, when }t\ra 0^+\text{, weakly in }L_x^2(\R^d).
}
Particularly, taking $\ph=\W(t_1)$ in the above argument and by \eqref{esti:scattering-criterion-assumption-energy}, 
\EQ{
\abs{\jb{\W(t_1)-\W(t_2),\W(t_1)}}\lsm &   \norm{\W(t_1)}_{H_x^1}\absb{t_1^{\frac{dp}{4}}-t_2^{\frac{dp}{4}}} + \norm{\W(t_1)}_{L_x^{p+2}} \absb{t_1^{\frac{dp}{2}-1}-t_2^{\frac{dp}{2}-1}} \\
\lsm & \absb{t_1^{\frac{dp}{2}-1}-t_1^{\frac{dp}{4}-1}t_2^{\frac{dp}{4}}} + \absb{t_1^{\frac{dp}{2}-1}-t_2^{\frac{dp}{2}-1}}.
}
Noting that $\frac{dp}{2}-1>0$ and $\W(t_2)\wra \W_0$, let $t_2\ra0^+$,
\EQ{
\abs{\jb{\W(t_1)-\W_0,\W(t_1)}} \lsm & t_1^{\frac{dp}{2}-1}.
}
Therefore, by $\W(t_1)\wra \W_0$,
\EQ{
\lim_{t\ra 0^+}\norm{\W(t)-\W_0}_{L_x^2(\R^d)}^2= \lim_{t_1\ra 0^+} \brkb{ \jb{\W(t_1)-\W_0,\W(t_1)} - \jb{\W(t_1)-\W_0,\W_0}} =0,
}
then by Plancherel's identity,
\EQ{
\lim_{t\ra +\I}\normb{w(t)-S(t) \F^{-1}\wb \W_0}_{L_x^2(\R^d)}^2 =0,
}
which implies
\EQ{
\lim_{t\ra +\I}\normb{u(t)-S(t) \brk{\F^{-1}\wb \W_0 +v_0}}_{L_x^2(\R^d)}^2 =0.
}
The scattering statement follows by taking $u_+=\F^{-1}\wb \W_0 +v_0$. 
\end{proof}

\vskip 1.5cm
\section{Scattering in fractional weighted space: 3D case}\label{sec:fractional}
\vskip .5cm

\subsection{Fractional weighted condition}\label{sec:fra-initial}
We first make some basic settings for Theorem \ref{thm:frac-weighted-3d} and \ref{thm:frac-weighted-dd}, and the contents in Section \ref{sec:fra-initial} and \ref{sec:fra-local} work for general dimensions. 

Now in Section \ref{sec:fra-initial}, we make a ``high-low'' decomposition of the initial data. Fix some $0<s<1$, and let $\jb{x}^su_0\in L_x^2$. Throughout the proof of these two theorems, the norm $\norm{\jb{x}^su_0}_{L_x^2(\R^d)}$ can be viewed as a constant, then we usually omit this dependence, and write
\EQ{
C=C(\norm{\jb{x}^su_0}_{L_x^2(\R^d)})
}
for short. Fix a sufficiently small constant $\de_0>0$. Since $\jb{x}^su_0\in L_x^2$, then there exists $N_0=N_0(\de_0)\in2^\N$ such that
\EQn{
\norm{\jb{x}^s\chi_{\goe N_0}u_0}_{L_x^2(\R^d)} \loe \de_0.
}
Take the initial data as
\EQ{
v_0= \chi_{\goe N_0}u_0 \text{, and }w_0= \chi_{\loe N_0}u_0.
}
Then, define
\EQ{
v=S(t)v_0\text{, and }w=u-v.
}
Therefore, $w$ solves the equation
\EQ{
\left\{ \aligned
&i\pd_t w +\frac12\De w=|u|^pu, \\
& w(0,x) = w_0.
\endaligned
\right.
}
Then, we apply the pseudo-conformal transform, and define
\EQ{
\V:=\TT v\text{, }\W:=\TT w\text{, and }\U:=\TT u.
}
Then, denote the final data as
\EQ{
\U_+:= \F^{-1} \wb u_0\text{, } \V_+ := \F^{-1} \wb v_0\text{, and }\W_+ := \F^{-1} \wb w_0.
}
Therefore, we have
\EQ{
\V = S(t)\V_+,
}
and
\EQn{\label{eq:nls-w-pc-duhamel}
\W = S(t)\W_+ +i\int_t^{+\I}S(t-\ta)(\ta^{\frac{dp}{2}-2}|\U|^p\U) \dd \ta,
}
where the integral is in the sense that
\EQ{
\lim_{t'\ra+\I}\int_t^{t'}S(t-\ta)(\ta^{\frac{dp}{2}-2}|\U|^p\U) \dd \ta\quad\text{in }L^2.
}

Now, we give the proof of Theorem \ref{thm:frac-weighted-3d}. Define the modified pseudo-conformal energy for the nonlinear part $\W$:
\EQn{\label{defn:pseudo-conformal-energy}
	\E(t) := \frac{1}{4}t^{2-\frac{dp}{2}}\int_{\R^d}|\nabla\W(t,x)|^2 \dx + \frac{1}{p+2} \int_{\R^d}|\U(t,x)|^{p+2} \dx.
}
\begin{remark}\label{rem:reason-for-modified-energy}
We now explain why we use the modified energy as in \eqref{defn:pseudo-conformal-energy}, rather than the following 
\EQ{
\wt\E(t) := \frac{1}{4}t^{2-\frac{dp}{2}}\int_{\R^d}|\nabla\W(t,x)|^2 \dx + \frac{1}{p+2} \int_{\R^d}|\W(t,x)|^{p+2} \dx.
}
In fact, this modification is necessary from the perspective of energy estimate. The energy increment of $\wt\E$ includes the term
\EQn{\label{eq:reason-for-modified-energy}
\int_0^T\int_\R \nabla\W_{hi}\cdot\nabla\W_{hi} O(\V_{low}\W_{low}^p)\dd x\dd t.
}
First, if we simply use energy bound for $\nabla \W$, then
\EQ{
\int_0^T t^{\frac d2p-2} \wt \E(t)  \norm{\V_{low}}_{L_x^\I}\norm{\W_{low}}_{L_x^\I}^p\dd t.
}
Since $\norm{\V}_{L_x^\I}$ and $\norm{\W}_{L_x^\I}$ also create additional time singularity near the origin,  the above integral in $t$ would be divergent when $p$ is close to $2/d$. Second, if we attempt to apply the bilinear Strichartz estimate to $\nabla\W_{hi}\V_{low}$, then
\EQ{
\eqref{eq:reason-for-modified-energy} \lsm & \norm{\nabla\W_{hi}\V_{low}}_{L_{t,x}^2} \norm{\nabla\W_{hi} \W^p}_{L_{t,x}^2} \\
\lsm & \norm{\nabla\W_{hi}\V_{low}}_{L_{t,x}^2}  \normb{t^{\frac d4p-1} \wt \E(t) \norm{\W}_{L_x^{\I}}^p}_{L_{t}^2}.
}
We can see that the latter integral in $t$ would still be divergent when $p$ is close to $2/d$. Then, it would be very difficult to handle the term \eqref{eq:reason-for-modified-energy} using our method. Therefore, we make the modification in \eqref{defn:pseudo-conformal-energy} to avoid such difficulties.
\end{remark}  

First, we give the estimates for final data:
\begin{lem}\label{lem:fra-initial-data}
Let $\de_0$, $N_0$, $\V_+$, and $\W^+$ be defined as above. Suppose that $0<s<1$, and $\jb{x}^su_0\in L_x^2$. Then,
\EQn{\label{eq:fra-bound-v0}
\norm{\V_+}_{H_x^s}\lsm\de_0,
}
and
\EQn{\label{eq:fra-bound-w0}
\norm{\W_+}_{H_x^s} + N_0^{s-1}\norm{\W_+}_{\dot H_x^1} \lsm 1.
}
\end{lem}
\begin{remark}
Recall that the implicit constant may depend on $\norm{\jb{x}^su}_{L_x^2(\R^d)}$.
\end{remark}
This lemma follows by the definition of $\V_+$ and $\W_+$. By \eqref{eq:fra-bound-v0} and the Strichartz estimate, we obtain:
\begin{lem}[Linear estimates]\label{lem:fra-linear}
Let $\de_0$, $N_0$, $\V_+$, and $\W_+$ be defined as above. Suppose that $0<s<1$, and $\jb{x}^su_0\in L_x^2$. Then,
\EQn{\label{eq:fra-bound-v}
\norm{\jb{\nabla}^s\V}_{S^0(\R)} \lsm\de_0,
}
where $S^0$ is defined in \eqref{defn:s0}.
\end{lem}
\begin{remark}
Note that $\V$ is a high-frequency cut-off function. In fact, by the decomposition of initial data, we can observe that for any $t>0$, $\wh \V(t,\cdot)$ is supported on $\fbrk{\xi\in\R:|\xi|\gsm N_0}$. Therefore, we can express $V$ as the sum:
\EQ{
\V=\sum_{N:N\gsm N_0} \V_N.
}
In the following, we will keep this point in mind, but for brevity, we will usually omit the lower bound $N_0$ when summing over $N$.
\end{remark}

\subsection{Local estimate for fractional weight}\label{sec:fra-local}
In this section, we start with the following local result on $|J|^su$ with general dimensions $d\goe 1$.
\begin{lem}\label{lem:Ju-local}
Let $0\loe s\loe 1$, $\langle x\rangle^su_0\in L_x^2(\R^d)$, and $u$ be the solution to \eqref{eq:nls}. Then, there exists $\de=\de(d,p,\norm{u_0}_{L_x^2(\R^d)})>0$ such that $|J|^su\in C([0,\de]; L_x^2(\R^d))$ with
\EQn{\label{eq:bound-Ju-local}
\norm{|J|^su}_{S^0([0,\de])} \lsm 1.
}
\end{lem}
\begin{proof}
Let $\de>0$ be some small parameter that will be defined later. In this lemma, we will restrict the variable $(t,x)$ on $[0,\de]\times\R^d$. First, applying the operator $|J|^s$ to the equation \eqref{eq:nls}, and by Lemma \ref{lem:J-T},
\EQ{
(i\pd_t + \frac12\De)|J|^su = |J|^s(|u|^pu).
}
Then, we have the Duhamel's formula:
\EQ{
|J|^su
=S(t)(|x|^su_0)-i\int_0^t   S(t-\ta)|J|^s(|u(\ta)|^pu(\ta))\dd \ta.
}
Denote that
\EQ{
u_1(t):=M(-t)u(t).	
}
By $|J|^s=M(t)|t\nabla|^sM(-t)$,
\EQn{\label{eq:Js-act-on-nonlinearterm}
|J|^s(|u(t)|^pu(t)) = & M(t)|t\nabla|^sM(-t) (|u(t)|^pu(t)) \\
= & M(t)|t\nabla|^s (|M(-t)u(t)|^pM(-t)u(t)) \\
= & M(t)|t\nabla|^s (|u_1|^pu_1).
}
Note that
\EQ{
S^0\subset L_t^\I L_x^2 \cap L_{t,x}^{\frac{2(d+2)}{d}}.
}
By Lemma \ref{lem:strichartz}, \eqref{eq:Js-act-on-nonlinearterm}, fractional chain rule in Lemma \ref{lem:frac-chain}, and H\"older's inequality in $t$,
\EQ{
\norm{|J|^su}_{S^0([0,\de])} \lsm & \||x|^su_0\|_{L^2} +\norm{M(t)|t\nabla|^s (|u_1|^pu_1)}_{L_t^{q_0'} L_x^{r_0'}}\\
\sim & \||x|^su_0\|_{L^2} +\norm{|t\nabla|^s (|u_1|^pu_1)}_{L_t^{q_0'} L_x^{r_0'}} \\
\lsm & \||x|^su_0\|_{L^2} +\normb{ \norm{u_1}_{L_x^r}^{p}\norm{|t\nabla|^su_1}_{L_x^{r}}}_{L_t^{q_0'}} \\
\lsm & \||x|^su_0\|_{L^2} + \norm{u_1}_{S^0}^p \norm{|t\nabla|^su_1}_{L_t^{q_1} L_x^{r}} \\
\sim & \||x|^su_0\|_{L^2} + \norm{u}_{S^0}^p \norm{|J|^su}_{L_t^{q_1} L_x^{r}},
}
where the parameters involved satisfy that 
\begin{align*}
\frac1{r_0}=\frac{d+4-dp}{2(d+2)};\quad 
\frac1{q_0}=\frac{d(dp-2)}{4(d+2)};\quad 
\frac1{q_1}=\frac{6d+8-(d+2)dp}{4(d+2)};\quad 
\frac1q=\frac1r=\frac{d}{2(d+2)}.
\end{align*}
Moreover, when $\frac2d<p<\frac4d$, we find that $q_1<r$. Using H\"older's inequality in $t$ and choosing $\delta<1$, we further have that 
\EQn{\label{421-2}
\norm{|J|^su}_{S^0}
\lsm & \||x|^su_0\|_{L^2} + \de^{\frac{1}{q_1}-\frac{1}{r}} \norm{u}_{S^0}^p \norm{|J|^su}_{S^0}.
}
By the local results in Lemma \ref{lem:local-l2}, we can set $\delta=\de(d,p,\norm{u_0}_{L_x^2})>0$ small enough such that 
$$
C\de^{\frac{1}{q_1}-\frac{1}{r}}\norm{u}_{S^0}^p \loe C\de^{\frac{1}{q_1}-\frac{1}{r}}\norm{u_0}_{L_x^2}^p\loe \frac12,
$$ 
then from \eqref{421-2}, we get that
\EQ{
\norm{|J|^su}_{S^0} 
\lsm & \|\jb{x}^su_0\|_{L^2}.
}
This finishes the proof of the lemma.
\end{proof}
We also need Lemma \ref{lem:Ju-local} with the pseudo-conformal variable:
\begin{cor}\label{cor:fra-local-u}
Suppose that $0\loe s\loe 1$, $\U_+\in H_x^s(\R^d)$, and $\U(t,x)$ is the solution to \eqref{eq:nls-pc}. Then, there exists some constant $T=T(d,p,\norm{u_0}_{L_x^2(\R^d)})>0$ such that
\EQ{
\norm{\jb{\nabla}^s\U}_{S^0([T,+\I))} \lsm 1.
}
\end{cor}
\begin{proof}
Note that $\U_+\in H_x^s(\R^d)$ implies $\jb{x}^su_0\in L^2$, thereby Lemma \ref{lem:Ju-local} holds.  Taking $T=1/\de$, where $\de$ is denoted in Lemma \ref{lem:Ju-local}, then this Corollary follows by
\EQ{
\norm{|J|^su}_{S^0((0,\de])} \sim \norm{\abs{\nabla}^s\U}_{S^0([T,+\I))},
}
see Remark \ref{rem:spacetime-exponent-transform}.
\end{proof}

The above preparation holds for general dimensions. Now, in the rest of Section \ref{sec:fractional}, we will give the proof of scattering in 3D case, namely Theorem \ref{thm:frac-weighted-3d}. We only consider the forward-in-time case, since the backward case can be derived similarly by taking the transform
\EQ{
u(t,x) \quad\ra\quad \wb u(-t,x),
}
and then consider the related forward scattering problem. 

Next, we apply the scattering criterion in Proposition \ref{prop:scattering-criterion} to prove Theorem \ref{thm:frac-weighted-3d}. Note that by Hausdorff-Young's, H\"older's inequalities, and $s>\frac{p}{2(p+2)}$ for $\frac23<p<\frac43$,
\EQ{
\sup_{t\in\R}\norm{\V(t)}_{L_x^{p+2}} = \sup_{t\in\R} \normb{\F(e^{-\frac12it|\cdot|^2}\wb v_0)}_{L_x^{p+2}} \lsm  \norm{v_0}_{L_x^{\frac{p+2}{p+1}}} \lsm \norm{\jb{x}^sv_0}_{L_x^2}.
}
Therefore, it suffices to prove that
\EQ{
\sup_{0<t\le T_0}\E(t)<C,
}
where $\E$ is defined in \eqref{defn:pseudo-conformal-energy}, and for some $T_0>0$. This will be finished in Proposition \ref{prop:fra-energy-3d} below.
\subsection{Energy estimates near the infinity time}
We have proved that $|J|^su(t)\in L_x^2(\R^3)$ locally, and in this subsection, we give an improved estimate for $w=u-v$, that is $Jw(t)\in L_x^2(\R^3)$ locally. Note that $Jv\notin L_x^2(\R^3)$, then in order to transfer the vector field $J$, it is convenient to work with the pseudo-conformal variable. After applying the transform $\TT$, we observe that $J$ becomes $\nabla$, then it suffices to use bilinear Strichartz estimate to reduce the derivative. The main result in this section is as follows:

\begin{prop}\label{prop:fra-local-w-3d}
Suppose that the assumptions in Theorem \ref{thm:frac-weighted-3d} hold. Then, there exists some constant $T_0=T_0(p,\norm{\jb{x}^su_0}_{L_x^2(\R^3)})>0$ such that
\EQ{
\sup_{T_0\loe t <+\I}\norm{\W(t)}_{H_x^1(\R^3)} \loe C N_0^{1-s}.
}
\end{prop}
We reduce the proof of Proposition \ref{prop:fra-local-w-3d} to the following lemma:
\begin{lem}\label{lem:fra-local-nonlinear-3d}
Suppose that the assumptions in Theorem \ref{thm:frac-weighted-3d} hold, and fix any $T_0>T$, where $T$ is defined in Corollary \ref{cor:fra-local-u}. Then, there exists some constant $\th=\th(p)>0$ such that
\EQ{
\normb{\int_t^{+\I} S(t-\ta)(\ta^{\frac{3p}{2}-2}|\U|^p\U) \dd \ta}_{L_t^\I\dot H_x^1([T_0,+\I)\times\R^3)} \lsm T_0^{-\th}(1+\norm{\W}_{L_t^\I\dot H_x^1([T_0,+\I)\times\R^3)}) .
}
\end{lem}
Now, we prove that Lemma \ref{lem:fra-local-nonlinear-3d} implies Proposition \ref{prop:fra-local-w-3d}:
\begin{proof}[Proof of Proposition \ref{prop:fra-local-w-3d}]
Let $T$ be as in Corollary \ref{cor:fra-local-u}, and take some $T_0>T$ that will be defined later. We restrict the spacetime variable $(t,x)$ on $[T_0,+\I)\times\R^3$. The $L_x^2$-bound follows from $H^s$-bound of $\V(t)$ in Lemma \ref{lem:fra-linear} and the conservation of mass
\EQ{
\norm{\U(t)}_{L_x^2} \equiv \norm{\U_+}_{L_x^2}.
}
Now, it suffices to prove the $\dot H_x^1$-bound of $\W(t)$. By the Duhamel formula \eqref{eq:nls-w-pc-duhamel} of $\W$ and Lemma \ref{lem:fra-local-nonlinear-3d},
\EQ{
\norm{\W}_{L_t^\I \dot H_x^1} \lsm & \norm{\W_+}_{\dot H_x^1} + \normb{\int_t^{+\I} S(t-\ta)(\ta^{\frac{3p}{2}-2}|\U|^p\U) \dd \ta}_{L_t^\I\dot H_x^1} \\
\lsm & N_0^{1-s} + T_0^{-\th} + T_0^{-\th} \norm{\W}_{L_t^\I\dot H_x^1([T_0,+\I)\times\R^3)}.
}
Then, noting that the implicit constant in the above inequality may depend on $p$ and $\norm{\jb{x}^su_0}_{L_x^2}$, the proposition follows by taking $T_0=T_0(p,\norm{\jb{x}^su_0}_{L_x^2})>0$ be some large constant such that
\EQ{
C(p,\norm{\jb{x}^su_0}_{L_x^2})T_0^{-\th} =\frac12.
}
\end{proof}
Now, it suffices to prove Lemma \ref{lem:fra-local-nonlinear-3d}:
\begin{proof}[Proof of Lemma \ref{lem:fra-local-nonlinear-3d}]
Let $T_0>T$ be some large constant, then Corollary \ref{cor:fra-local-u} holds on $[T_0,+\I)$. In the proof of this lemma, we restrict the variable $(t,x)$ on $[T_0,+\I)\times\R^3$. By the Strichartz estimate,
\EQ{
\normb{\int_t^{+\I} S(t-\ta)(\ta^{\frac{3p}{2}-2}|\U|^p\U) \dd \ta}_{L_t^\I\dot H_x^1} \lsm & \normb{t^{\frac{3p}{2}-2}|\U|^p\nabla \U }_{L_t^{2}L_x^{\frac65}+L_t^1L_x^2}.
}
Then, we decompose 
\EQ{
|\U|^p|\nabla \U| \lsm & |\U|^p|\nabla\W| + \sum_{N\in2^\N} |\U_{\gsm N} + \U_{\ll N}|^{p}|\nabla \V_N| \\
\lsm & |\U|^p|\nabla\W| + \sum_{N\in2^\N} |\U_{\gsm N}|^{p}|\nabla \V_N| + \sum_{N\in2^\N} | \U_{\ll N}|^{p}|\nabla \V_N|.
}
It suffices to consider
\begin{subequations}
\EQnn{
\normb{t^{\frac{3p}{2}-2}|\U|^p\nabla \U }_{L_t^{2}L_x^{\frac65}+L_t^1L_x^2} \lsm & \normb{t^{\frac{3p}{2}-2}|\U|^p\nabla\W }_{L_t^{2}L_x^{\frac65}} \label{esti:frac-local-nonlinear-term-3d-1}\\
& + \normb{t^{\frac{3p}{2}-2}\sum_{N\in2^\N} |\U_{\gsm N}|^{p}|\nabla \V_N| }_{L_t^{2}L_x^{\frac65}}\label{esti:frac-local-nonlinear-term-3d-2}\\
& + \normb{t^{\frac{3p}{2}-2}\sum_{N\in2^\N} | \U_{\ll N}|^{p}|\nabla \V_N| }_{L_t^1L_x^2}.\label{esti:frac-local-nonlinear-term-3d-3}
}
\end{subequations}

We first estimate \eqref{esti:frac-local-nonlinear-term-3d-1}. Noting that $(\frac{4p}{3p-2},3p)$ is $L^2$-admissible, by H\"older's inequality,
\EQ{
\eqref{esti:frac-local-nonlinear-term-3d-1} 
\lsm &  \normb{t^{\frac{3p}{2}-2}}_{L_t^{\frac{4}{4-3p}}} \norm{\U}_{L_t^{\frac{4p}{3p-2}}L_x^{3p}}^p \norm{\nabla\W}_{L_t^\I L_x^2}.
}
Note that
\EQ{
\normb{t^{\frac{3p}{2}-2}}_{L_t^{\frac{4}{4-3p}}} \lsm  \brkb{\int_{T_0}^{+\I} t^{-2} \dd t}^{\frac{4-3p}{4}} \lsm T_0^{-\frac{4-3p}{4}},
}
then combining Corollary \ref{cor:fra-local-u},
\EQn{\label{eq:frac-local-nonlinear-term-3d-1}
\eqref{esti:frac-local-nonlinear-term-3d-1} 
\lsm T_0^{-\frac{4-3p}{4}} \norm{\nabla\W}_{L_t^\I L_x^2}.
}

Then, we estimate \eqref{esti:frac-local-nonlinear-term-3d-2}. Noting that $s>\frac{1}{1+p}$, by H\"older's inequality and Corollary \ref{cor:fra-local-u},
\EQn{\label{eq:frac-local-nonlinear-term-3d-2}
\eqref{esti:frac-local-nonlinear-term-3d-2} 
\lsm & \sum_{N\in2^\N} \normb{t^{\frac{3p}{2}-2}}_{L_t^{\frac{4}{4-3p}}} \norm{\U_{\gsm N}}_{L_t^{\frac{4p}{3p-2}}L_x^{3p}}^p \norm{\nabla\V_N}_{L_t^\I L_x^2} \\
\lsm & T_0^{-\frac{4-3p}{4}} \sum_{N\in2^\N} N^{1-(1+p)s} \normb{N^s\U_{\gsm N}}_{L_t^{\frac{4p}{3p-2}}L_x^{3p}}^p \norm{\jb{\nabla}^s\V}_{L_t^\I L_x^2} \\
\lsm & T_0^{-\frac{4-3p}{4}}  \normb{\jb{\nabla}^{s}\U}_{L_t^{\frac{4p}{3p-2}}L_x^{3p}}^p \norm{\jb{\nabla}^s\V}_{L_t^\I L_x^2} \\
\lsm & T_0^{-\frac{4-3p}{4}}.
}

Finally we estimate \eqref{esti:frac-local-nonlinear-term-3d-3}. By H\"older's inequality and Corollary \ref{cor:fra-local-u},
\EQn{\label{eq:frac-local-nonlinear-term-3d-3-1}
\eqref{esti:frac-local-nonlinear-term-3d-3} 
\lsm & \sum_{N\in2^\N} \norm{\U_{\ll N}\nabla \V_N}^{\frac23}_{L_t^{1+}L_x^2} \normb{t^{\frac{3p}{2}-2}| \U_{\ll N}|^{p-\frac23}|\nabla \V_N|^{\frac13}}_{L_t^{3-}L_x^6} \\
\lsm & \sum_{N_1\ll N} \norm{\U_{N_1}\nabla \V_N}^{\frac23}_{L_t^{1+}L_x^2} \norm{\U_{\ll N}}_{L_t^2 L_x^\I}^{p-\frac23} \normb{t^{\frac{3p}{2}-2}}_{L_t^{\frac{6}{4-3p}-}} \normb{\nabla \V_N}_{L_t^\I L_x^2}^{\frac13} \\
\lsm & T_0^{-\frac{4-3p}{3}+} \sum_{N_1\ll N} \norm{\U_{N_1}\nabla \V_N}^{\frac23}_{L_t^{1+}L_x^2} \normb{\jb{\nabla}^{\frac12+\ep}\U_{\ll N}}_{L_t^2 L_x^6}^{p-\frac23}  N^{\frac13(1-s)} \\
\lsm & T_0^{-\frac{4-3p}{4}} \sum_{N_1\ll N} N^{\frac13(1-s)} \norm{\U_{N_1}\nabla \V_N}^{\frac23}_{L_t^{1+}L_x^2}. 
}
Note that $(\frac{4p}{3p-2},3p)$ and $(\frac{2}{\ep},\frac{6}{3-2\ep})$ are $L^2$-admissible for some $0<\ep\ll 1$. By the bilinear Strichartz estimate in Lemma \ref{lem:bilinearstrichartz}, for $N_1\ll N$,
\EQn{\label{eq:frac-local-nonlinear-term-3d-3-bi}
\norm{\U_{N_1}\nabla \V_N}_{L_t^{1+}L_x^2} \lsm &\frac{N_1^{0+}}{N^{1/2}} \brkb{\norm{P_{N_1}\U_+}_{L_x^2} + \normb{P_{N_1}\brkb{t^{\frac{3p}{2}-2}|\U|^p\U}}_{L_t^{\frac{2}{1+\ep}}L_x^{\frac{6}{5-2\ep}}}}\norm{\nabla P_N\V_+}_{L_x^2} \\
\lsm & N_1^{-s+}N^{\frac12-s} \brkb{\norm{\U_+}_{H_x^s} + \normb{t^{\frac{3p}{2}-2}\jb{\nabla}^s\brko{|\U|^p\U}}_{L_t^{\frac{2}{1+\ep}}L_x^{\frac{6}{5-2\ep}}}}\norm{\V_+}_{H_x^s} \\
\lsm & N_1^{-s+}N^{\frac12-s} \brkb{1 + \normb{t^{\frac{3p}{2}-2}}_{L_t^{\frac{4}{4-3p}}} \norm{\U}_{L_t^{\frac{4p}{3p-2}}L_x^{3p}}^p \norm{\jb{\nabla}^s\U}_{L_t^{\frac{2}{\ep}} L_x^{\frac{6}{3-2\ep}}}} \\
\lsm & N_1^{-s+}N^{\frac12-s}\brkb{1 + T_0^{-\frac{4-3p}{4}}} \\
\lsm & N_1^{-s+}N^{\frac12-s}.
}
Then, by \eqref{eq:frac-local-nonlinear-term-3d-3-1}, \eqref{eq:frac-local-nonlinear-term-3d-3-bi}, and $s>\frac23$,
\EQn{\label{eq:frac-local-nonlinear-term-3d-3}
\eqref{esti:frac-local-nonlinear-term-3d-3} 
\lsm & T_0^{-\frac{4-3p}{4}} \sum_{N_1\ll N} N^{\frac13(1-s)} \norm{\U_{N_1}\nabla \V_N}^{\frac23}_{L_t^{1+}L_x^2} \\
\lsm & T_0^{-\frac{4-3p}{4}} \sum_{N_1\ll N} N_1^{-\frac23s+} N^{\frac23-s} \\
\lsm & T_0^{-\frac{4-3p}{4}}.
}
By \eqref{eq:frac-local-nonlinear-term-3d-1}, \eqref{eq:frac-local-nonlinear-term-3d-2}, and \eqref{eq:frac-local-nonlinear-term-3d-3}, we have
\EQ{
\normb{t^{\frac{3p}{2}-2}|\U|^p\nabla \U }_{L_t^{2}L_x^{\frac65}+L_t^1L_x^2} \lsm T_0^{-\frac{4-3p}{4}} (1+\norm{\nabla\W}_{L_t^\I L_x^2}).
}
This finishes the proof.
\end{proof}
\begin{cor}\label{cor:fra-local-w-3d}
Suppose that the assumptions in Theorem \ref{thm:frac-weighted-3d} hold. Then, there exists a constant $A=A(p,\norm{\jb{x}^su_0}_{L_x^2(\R^3)},T_0)>0$ such that
\EQ{
\E(T_0) \loe AN_0^{2(1-s)}.
}
\end{cor}
\begin{proof}
We restrict the variable $(t,x)$ on $[T_0,+\I)\times\R^3$. By Gagliardo-Nirenberg's inequality in Lemma \ref{lem:GN},
\EQ{
\norm{f}_{L_x^{p+2}}^{p+2} \lsm \norm{f}_{L_x^2}^{p+2-\frac{3p}{2}} \norm{f}_{\dot H_x^1}^{\frac{3p}{2}}.
}
Since $s>\frac{14}{15}$ and $\frac{3p}{2(p+2)}<\frac{9}{22}$ for $\frac23<p<\frac43$, we also have
\EQ{
\norm{f}_{L_x^{p+2}} \lsm \norm{f}_{H_x^{\frac{3p}{2(p+2)}}} \lsm \norm{f}_{H_x^{s}}.
}
Then, combining the above two inequalities, and by Proposition \ref{prop:fra-local-w-3d}, $p<\frac43$, Corollary \ref{cor:fra-local-u}, and Lemma \ref{lem:fra-linear},
\EQ{
\E(T_0) \lsm & T_0^{2-\frac32p} \norm{\W(T_0)}_{\dot H_x^1}^2 + \norm{\W(T_0)}_{ L_x^{p+2}}^{p+2} + \norm{\V(T_0)}_{ L_x^{p+2}}^{p+2} \\
\lsm & T_0^{2-\frac32p} \norm{\W}_{L_t^\I \dot H_x^1([T_0,\I)\times\R^3)}^2 + \norm{\W}_{L_t^\I \dot H_x^1([T_0,\I)\times\R^3)}^{\frac{3p}{2}} + \norm{\V}_{L_t^\I H_x^s([T_0,\I)\times\R^3)}^{p+2} \\
\lsm & T_0^{2-\frac32p}N_0^{2(1-s)} + N_0^{\frac{3p}{2}(1-s)} + \de_0^{p+2} \\
\lsm &N_0^{2(1-s)},
}
where the implicit constant depends on $p,\norm{\jb{x}^su_0}_{L_x^2(\R^3)}$, and $T_0$. Denote the constant of the right hand side of the above inequality as $A=A(p,\norm{\jb{x}^su_0}_{L_x^2(\R^3)},T_0)$, then the corollary follows.  
\end{proof}

\subsection{Energy estimate towards the origin}
Now, to complete the proof of Theorem \ref{thm:frac-weighted-3d}, it suffices to prove:
\begin{prop}\label{prop:fra-energy-3d}
	Suppose that the assumptions in Theorem \ref{thm:frac-weighted-3d} hold. Let $A$ be the constant in Corollary \ref{cor:fra-local-w-3d}. Then, we have
	\EQ{
		\sup_{t\in(0,T_0]} \E(t) \loe 2A N_0^{2(1-s)}. 
	}
\end{prop}
First, we establish the bootstrap framework. Denote that $I=(0,T_0]$. Then it suffices to prove: under the bootstrap hypothesis,
\EQn{\label{eq:fra-energy-bootstrap-hypothesis-3d}
	\sup_{t\in I} \E(t) \loe 2A N_0^{2(1-s)},
}
then
\EQn{\label{eq:fra-energy-bootstrap-conclusion-3d}
	\sup_{t\in I} \E(t) \loe \frac32A N_0^{2(1-s)}.
}
In the following, note that $A,T_0$ are constant, then we omit its dependence.

Before the start of proof, we first gather some useful estimates. 
\begin{lem} Suppose that the assumptions in Theorem \ref{thm:frac-weighted-3d} and the bootstrap hypothesis \eqref{eq:fra-energy-bootstrap-hypothesis-3d} holds. Then, we have that:
\begin{enumerate}
\item (Energy bound) For any $t\in I$,
\EQn{\label{eq:fra-energy-bound-w-3d} 
&\norm{\nabla \W(t)}_{L_x^2} \lsm t^{\frac {3p}4-1} N_0^{1-s}\text{, and }\norm{\W(t)}_{L_x^{p+2}} \lsm N_0^{\frac{2}{p+2}(1-s)} .
}
\item (Mass bound) For any $t\in I$,
\EQn{\label{eq:fra-mass-bound-w-3d}
\norm{\W(t)}_{L_x^2} \lsm 1. 
}
\end{enumerate}
\end{lem}
This lemma follows from \eqref{eq:fra-energy-bootstrap-hypothesis-3d}, Lemma \ref{lem:fra-linear}, and the mass conservation low directly, and we omit the details of proof. 

Furthermore, we also need the estimates in the following subsections.

\subsection{Localized interaction Morawetz estimate}
In this subsection, we derive the interaction Morawetz estimate for the nonlinear Schr\"odinger equation after the pseudo-conformal transform in \eqref{eq:nls-pc}, namely
\EQ{
i\pd_t \U + \frac12\De\U = t^{\frac{3p}{2}-2}|\U|^p\U,
}
under the bootstrap hypothesis on the almost conservation law in \eqref{eq:fra-energy-bootstrap-hypothesis-3d}. The main feature is that the estimate for $\norm{\nabla\W(t)}_{L_x^2}$ has singularity near the temporal origin, thereby we consider the time-localized interaction Morawetz estimate, namely the estimate on each dyadic time interval.

We denote that $I_k:=[2^{k-1},2^k]$, for any $k\in\Z$ with $2^k\loe T_0$. 

\begin{prop}[Localized interaction Morawetz estimate]\label{prop:interaction-morawetz}
Suppose that the assumptions in Theorem \ref{thm:frac-weighted-3d} and the bootstrap hypothesis \eqref{eq:fra-energy-bootstrap-hypothesis-3d} holds. Then, for any $k\in\Z$ with $2^k\loe T_0$,
\EQn{\label{eq:fra-energy-bound-w-l4l4-3d}
\norm{\U}_{L_{t,x}^4(I_k\times\R^3)} + \norm{\W}_{L_{t,x}^4(I_k\times\R^3)} \lsm 2^{\frac{1}{4}\brko{\frac {3p}4-1}k} N_0^{\frac{1}{4}(1-s)},
}
where the implicit constant may depend on $A$ and $T_0$.
\end{prop}
\begin{remark}\label{rem:morawetz}
\begin{enumerate}
\item 
We remark that this estimate works for the scattering in the mass subcritical case, since it provides a mass subcritical spacetime estimate for the original solution $u$. Roughly speaking, for $k<0$,
\EQ{
	\normb{t^{-\frac{1}{4}(\frac34p-1)}\U}_{L_{t\sim2^k}^4L_x^4} \lsm \E^{1/8}.
}
Then, by Remark \ref{rem:spacetime-exponent-transform},
\EQ{
\normb{t^{\frac3{16}p}u}_{L_{t\sim2^{-k}}^4L_x^4} \lsm \E^{1/8}.
}
We can see that $\normb{t^{\frac3{16}p}u_\la}_{L_{t,x}^4}$ coincides with the same scaling of $\norm{u_\la}_{L_t^\I \dot H_x^{\frac{2-3p}{8}}}$, where $u_\la=\la^{\frac2p}u(\la^2 t,\la x)$.
\item The above interaction Morawetz estimate for \eqref{eq:nls-pc} only holds for positive time. When considering the backward scattering, we can transform the original solution $u$ to $\wb u(-t)$, thus change it into a forward scattering problem, and then consider the related pseudo-conformal transform and interaction Morawetz estimate.
\item We compare the Morawetz estimate \eqref{eq:fra-energy-bound-w-l4l4-3d} with the energy estimate. On each dyadic interval, heuristically, \eqref{eq:fra-energy-bound-w-l4l4-3d} can be viewed as
\EQ{
\norm{\W}_{L_x^4} \lsm 2^{\brkb{-\frac{1}{4}+\frac{1}{4}\brko{\frac {3p}4-1}}k} N_0^{\frac{1}{4}(1-s)}.
}
On the other hand, by Sobolev's inequality, interpolation,  \eqref{eq:fra-energy-bound-w-3d}, and \eqref{eq:fra-mass-bound-w-3d},
\EQ{
\norm{\W}_{L_x^4} \lsm \norm{\W}_{\dot H_x^{\frac34}} \lsm 2^{\frac34(\frac34p-1)k} N_0^{\frac{3}{4}(1-s)}.
}
We can see that the Morawetz estimate has much less energy increase in terms of $N_0$, but it creates more temporal singularity for all $\frac23<p<\frac43$.
\end{enumerate}
\end{remark}
\begin{proof}[Proof of Proposition \ref{prop:interaction-morawetz}]
The proof is essentially equivalent to the classical case. Denote that
\EQ{
m(t,x)=\frac12|\U(t,x)|^2\text{, and }p(t,x)=\frac12\im \brko{\wb \U(t,x)\nabla \U(t,x)}.
}
Then, we have
\EQn{\label{eq:local-mass-flow}
\pd_tm=-\nabla\cdot p,
}
and
\EQn{\label{eq:local-momentum-flow}
\pd_tp = & -\frac12\re \nabla\cdot \brko{\nabla \wb \U \nabla \U}  -\frac{p}{2(p+2)}t^{\frac{3p}{2}-2} \nabla \brkb{|\U|^{p+2}} + \frac14 \nabla \De m.
}
Moreover, we note that 
\EQ{
\pd_j\brkb{\frac{x_k}{|x|}}=\frac{\de_{jk}}{|x|}-\frac{x_jx_k}{|x|^3}\text{, }
\nabla\cdot \frac{x}{|x|}=\frac2{|x|}\text{, and }\De\nabla\cdot \frac{x}{|x|}=-8\pi\delta(x).
}
Let 
\EQ{
	M(t):= \int\!\!\int_{\R^{3+3}} \frac{x-y}{|x-y|}\cdot p(t,x)\> m(t,y)\dx\dy.
}
Then by \eqref{eq:local-mass-flow} and \eqref{eq:local-momentum-flow}, we have the interaction Morawetz identity
\begin{subequations}
\EQnn{
\pd_t M(t)
= &\iint_{\R^{3+3}} \frac{x-y}{|x-y|}\cdot \partial_t p(t,x) m(t,y)\dd x\dd y\nonumber\\
& \quad + \iint_{\R^{3+3}} \frac{x-y}{|x-y|}\cdot p(t,x) \partial_t m(t,y)\dd x\dd y\nonumber\\
= &\iint_{\R^{3+3}} \frac{x-y}{|x-y|}\cdot \brkb{ -\frac12\re \nabla\cdot \brko{\nabla \wb \U \nabla \U}  -\frac{p}{2(p+2)}t^{\frac{3p}{2}-2} \nabla \brkb{|\U|^{p+2}}}(t,x) m(t,y)\dd x\dd y\label{esti:inter-mora-angular-1}\\
& \quad -\iint_{\R^{3+3}} \frac{x-y}{|x-y|}\cdot p(t,x) \nabla\cdot p(t,y)\dd x\dd y\label{esti:inter-mora-angular-2}\\
&\quad +\frac14\iint_{\R^{3+3}} \frac{x-y}{|x-y|}\cdot \nabla \De m(t,x) m(t,y)\dd x\dd y. \label{esti:inter-mora-angular-3}
}
\end{subequations}
Integrating by parts,
\EQ{
\eqref{esti:inter-mora-angular-1} = & 
\iint_{\R^{3+3}} \frac{1}{|x-y|}|\nabla \U(t,x)|^2 m(t,y)\dd x\dd y \\
& + 
\frac{p}{p+2}\iint_{\R^{3+3}} \frac{1}{|x-y|} t^{\frac{3p}{2}-2} |\U(t,x)|^{p+2}m(t,y)\dd x\dd y.
}
Using the classical argument in \cite{CKSTT08Annals},
\EQ{
\eqref{esti:inter-mora-angular-2} \goe - \iint_{\R^{3+3}} \frac{1}{|x-y|}|\nabla \U(t,x)|^2 m(t,y)\dd x\dd y.
}
Then, we have
\EQ{
\eqref{esti:inter-mora-angular-1} + \eqref{esti:inter-mora-angular-2} \ge \frac{p}{p+2}\iint_{\R^{3+3}} \frac{1}{|x-y|} t^{\frac{3p}{2}-2} |\U(t,x)|^{p+2}m(t,y)\dd x\dd y.
}
We also have
\EQ{
\eqref{esti:inter-mora-angular-3} = &  \frac14\iint_{\R^{3+3}} \frac{x-y}{|x-y|}\cdot \nabla \De m(t,x) m(t,y)\dd x\dd y \\
= & 2\pi \int_{\R^3} |\U(t,x)|^4 \dd x.
}
Therefore, 
\EQ{
\pd_tM(t) \goe & 2\pi \int_{\R^3} |\U(t,x)|^4 \dd x + \frac{p}{p+2}\iint_{\R^{3+3}} \frac{1}{|x-y|} t^{\frac{3p}{2}-2} |\U(t,x)|^{p+2}m(t,y)\dd x\dd y.
}
Integrating in $t$ on $I_k$,
\EQn{\label{est:U-4}
&\int_{I_k}\int_{\R^3} |\U(\ta,x)|^4 \dd x \dd \ta \\
&+ \int_{I_k}\iint_{\R^{3+3}} \frac{1}{|x-y|} \ta^{\frac{3p}{2}-2} |\U(\ta,x)|^{p+2}|\U(\ta,y)|^{2}\dd x\dd y\dd \ta\\
\lsm & \sup\limits_{t\in [t_0,T]} \big| M(t)\big|.
}
Note that by the conservation of mass, Lemma \ref{lem:fra-linear}, \eqref{eq:fra-energy-bound-w-3d}, and \eqref{eq:fra-mass-bound-w-3d}, we get that for any $t\in I_k$,
\begin{align*}
\big| M(t)\big|
\lesssim  &
 \norm{\U(t)}_{L_x^2}^2 \norm{\U(t)}_{\dot H_x^{1/2}}^2\\
\lesssim   &
\norm{\V(t)}_{\dot H_x^{1/2}}^2+\norm{\W(t)}_{L_x^2}\norm{\W(t)}_{\dot H_x^1}\\
\lesssim &
2^{\brko{\frac {3p}4-1}k} N_0^{1-s}.
\end{align*}
Then, combining with this inequality and \eqref{est:U-4}, we have that
\EQ{ 
\norm{\U}_{L_{t,x}^4(I_k\times\R^3)} \lsm 2^{\frac{1}{4}\brko{\frac {3p}4-1}k} N_0^{\frac{1}{4}(1-s)}.
}
By Sobolev's inequality and Lemma \ref{lem:fra-linear}, \EQ{
\norm{\V}_{L_{t,x}^4(I_k\times\R^3)} \lsm \de_0,
}
then the estimate for $\W$ follows by noting that $2^{\frac{1}{4}\brko{\frac {3p}4-1}k} N_0^{\frac{1}{4}(1-s)}\gsm 1$.
\end{proof}

\subsection{Bilinear Strichartz estimate}
In this subsection, we establish a type of the bilinear Strichartz estimate on dyadic time interval. We first derive the dual Strichartz estimates for the nonlinear term $t^{\frac {3p}2 -2}(|\U|^p\U)$ in equation \eqref{eq:nls-pc}. 

We need to derive different kinds of spacetime estimates, since the difficulty of covering the energy increment varies for different nonlinear exponent $p$. Note that by \eqref{eq:fra-energy-bound-w-3d}, the energy gives
\EQ{
\norm{\nabla \W(t)}_{L_x^2} \lsm t^{\frac {3p}4-1} N_0^{1-s}.
}
If $p$ is close to $2/3$, the main issue lies in covering the estimate in $t$ near the origin, and if $p$ is close to $4/3$, this temporal singularity can be neglected, but the increase of $N_0$ will become more significant.

Therefore, we intend to obtain two kinds of spacetime estimate: the first one focuses on better time singularity near the origin, without considering the derivative loss and energy increase. The second type emphasizes better energy increase in terms of $N_0$, regardless of the temporal singularity. The former will be used to deal with the case when $p$ is near $2/3$, and the latter works for larger $p$. Moreover, the Morawetz estimate in Proposition \ref{prop:interaction-morawetz} exhibits better energy increase than the energy estimate in \eqref{eq:fra-energy-bound-w-3d}, which is helpful to the large $p$ case. 

Based on the above observations, we first derive the $L_t^1L_x^2$-estimate, using $L_t^\I L_x^{p+2}$-estimate for $U$, which does not possess any time singularity. For the second purpose, we consider the $L_t^{2-} L_x^{6/5+}$-estimate (we make an $\ep$-perturbation, since the bilinear Strichartz estimate is not available for endpoint $L_t^{2} L_x^{6/5}$ in Lemma \ref{lem:bilinearstrichartz}). Here, we also utilize the aforementioned interaction Morawetz estimate rather than energy estimate to minimize the increase of $N_0$. The estimates are summarized as follows:
\begin{lem}\label{lem:est-FU}
Suppose that the assumptions in Theorem \ref{thm:frac-weighted-3d} and the bootstrap hypothesis \eqref{eq:fra-energy-bootstrap-hypothesis-3d} holds. Let $0<\ep\ll 1$ and $N\in2^\N$,  then we have
\EQn{\label{eq:fra-energy-bound-bilinear-nln-l1l2}
\normb{P_Nt^{\frac {3p}2 -2}(|\U|^p\U)}_{L_{t\in I_k}^1 L_x^2} 
\lsm 
2^{(\frac32p-1)k} N^{\frac{3p}{2(p+2)}} N_0^{\frac{2(p+1)}{p+2}(1-s)}.
}
and
\EQn{\label{eq:fra-energy-bound-bilinear-nln-l2l6/5}
\normb{P_Nt^{\frac {3p}2 -2}(|\U|^p\U)}_{L_{t\in I_k}^{\frac{2}{1+\ep}} L_x^{\frac{6}{5-2\ep}}} 
\lsm 
2^{\brkb{(\frac12p+\frac53)(\frac34p-1)-\frac12p+\frac56}k}  N_0^{\frac{3p-2+2\ep}{6}(1-s)}.
}
\end{lem}
\begin{proof}
For the first inequality, by Bernstein's inequality and \eqref{eq:fra-energy-bound-w-3d},
\EQ{
\normb{P_Nt^{\frac {3p}2 -2}(|\U|^p\U)}_{L_{t\in I_k}^1 L_x^2} 
\lsm & \normb{t^{\frac32p-2}P_N(|\U|^p\U)}_{L_t^1 L_x^2} \\
\lsm & 2^{(\frac32p-1)k} \norm{P_N(|\U|^p\U)}_{L_t^\I L_x^2} \\
\lsm & 2^{(\frac32p-1)k} N^{\frac{3p}{2(p+2)}} \norm{P_N(|\U|^p\U)}_{L_t^\I L_x^{\frac{p+2}{p+1}}} \\
\lsm & 2^{(\frac32p-1)k} N^{\frac{3p}{2(p+2)}} \norm{\U}_{L_t^\I L_x^{p+2}}^{p+1} \\
\lsm & 2^{(\frac32p-1)k} N^{\frac{3p}{2(p+2)}} N_0^{\frac{2(p+1)}{p+2}(1-s)}.
}
This gives \eqref{eq:fra-energy-bound-bilinear-nln-l1l2}.

For the second inequality, by interpolation and Proposition \ref{prop:interaction-morawetz},
\EQ{
\normb{P_Nt^{\frac {3p}2 -2}(|\U|^p\U)}_{L_{t\in I_k}^{\frac{2}{1+\ep}} L_x^{\frac{6}{5-2\ep}}} 
\lsm & \normb{t^{\frac32p-2}|\U|^p\U}_{L_{t\in I_k}^{\frac{2}{1+\ep}} L_x^{\frac{6}{5-2\ep}}} \\
\lsm & \normb{t^{\frac32p-2}\normo{\U}_{L_x^{\frac{6}{5-2\ep}(p+1)}}^{p+1}}_{L_t^{\frac{2}{1+\ep}}(I_k)} \\
\lsm & \normb{t^{\frac32p-2}\normo{\U}_{L_x^{4}}^{2p-\frac{4}{3}+\frac43\ep} \normo{\U}_{L_x^{2}}^{\frac73-p-\frac43\ep}}_{L_t^{\frac{2}{1+\ep}}(I_k)} \\
\lsm & \normb{t^{\frac32p-2} }_{L_t^{\frac{6}{5-3p+\ep}}(I_k)} \normo{\U}_{L_{t\in I_k}^{4}L_x^4}^{2p-\frac{4}{3}+\frac43\ep} \normo{\U}_{L_t^\I L_x^{2}}^{\frac73-p-\frac43\ep} \\
\lsm & 2^{\brkb{\frac32p-2 -\frac12p+\frac56+\frac16\ep}k}\normo{\U}_{L_{t\in I_k}^{4}L_x^4}^{2p-\frac{4}{3}+\frac43\ep} \\
\lsm & 2^{\brkb{(\frac12p+\frac53+\frac13\ep)(\frac34p-1)-\frac12p+\frac56+\frac16\ep}k}  N_0^{\frac{3p-2+2\ep}{6}(1-s)}.
}
Then, since $p>\frac23$, we have $(\frac34p-1)\frac13\ep+\frac16\ep>0$, which gives
\EQ{
2^{(\frac34p-1)\frac13\ep+\frac16\ep} \loe C(T_0).
}
This finishes the proof of \eqref{eq:fra-energy-bound-bilinear-nln-l2l6/5}.
\end{proof}

Based on the above two spacetime estimates, our main result in this subsection is the following bilinear Strichartz estimate on dyadic intervals:
\begin{lem}[Localized bilinear Strichartz estimate]\label{lem:fra-energy-bound-bilinear-3d}
Suppose that the assumptions in Theorem \ref{thm:frac-weighted-3d} and the bootstrap hypothesis \eqref{eq:fra-energy-bootstrap-hypothesis-3d} hold. Then,  for any $0<\ep\ll1$, $k\in\Z$ with $2^k\loe T_0$, and $N_1,N\in2^\N$ such that $N_1\ll N$,
\EQn{\label{eq:fra-energy-bound-bilinear-3d}
\norm{\nabla \V_{N} \W_{N_1}}_{L_{t\in I_k}^2L_x^2} 
\lsm \min\big\{
&\de_0 N^{\frac12-s} 2^{(\frac32p-1)k} N_1^{\frac{5p+4}{2(p+2)}} N_0^{\frac{2(p+1)}{p+2}(1-s)}, \\
&\de_0 N_1 N^{\frac12-s} 2^{\brkb{(\frac12p+\frac53)(\frac34p-1)-\frac12p+\frac56}k}  N_0^{\frac{3p-2+2\ep}{6}(1-s)}\big\}.
}
\end{lem}
\begin{remark}\label{rem:bilinear-strichartz}
As pointed above, the advantage of the first upper bound is that it has better time estimate near the origin $t=0$, while the second bound has better energy increase in terms of $N_0$.
\end{remark}
\begin{proof}
This lemma follows directly from Lemma \ref{lem:est-FU} and the bilinear Strichartz estimate in Lemma \ref{lem:bilinearstrichartz}, namely
\EQ{
\norm{\nabla \V_{N} \W_{N_1}}_{L_{t\in I_k}^2 L_x^2}  \lsm & \frac{N_1}{N^{1/2}}\norm{\nabla \V_+}_{L_x^2}\\
& \cdot \brkb{\norm{P_{N_1}\W_+}_{L_x^2} + \normb{P_{N_1}t^{\frac {3p}2 -2}(|\U|^p\U)}_{L_t^1 L_x^2+L_{t}^{\frac{2}{1+\ep}} L_x^{\frac{6}{5-2\ep}}(I_k\times\R^3)}}.
}
\end{proof}

\subsection{Proof of Proposition \ref{prop:fra-energy-3d}}

Now, we reduce the proof of Proposition \ref{prop:fra-energy-3d} to Lemma \ref{lem:fra-energy-3d-main} and \ref{lem:fra-energy-3d-other} below.
 Recall the equation for $\W$:
\EQ{
i\pd_t \W + \frac12\De \W = t^{\frac{3p}{2}-2} |\U|^p\U.
}
Multiply the equation with $t^{2-\frac{3p}{2}}\wb{\W_t}$, integrate in $x$, and take the real part, then we obtain 
\EQ{
\frac12t^{2-\frac {3p}2}\re\int_{\R^3} \De \W\wb\W_t \dd x = \re\int_{\R^3} |\U|^p\U \wb\W_t \dd x.
}
Then, integrate by parts,
\EQ{
- \frac12t^{2-\frac {3p}2}\re\int_{\R^3} \nabla \W\cdot\nabla\wb\W_t \dd x = \re\int_{\R^3} |\U|^p\U \wb\U_t \dd x - \re\int_{\R^3} |\U|^p\U \wb\V_t \dd x,
}
which gives
\EQ{
- \frac14t^{2-\frac {3p}2}\pd_t\brkb{\int_{\R^3} |\nabla \W|^2 \dd x} = \frac{1}{p+2} \pd_t\brkb{\int_{\R^3} |\U|^{p+2} \dd x} - \re\int_{\R^3} |\U|^p\U \wb\V_t \dd x.
}
Thus, by $p<\frac43$,
\EQ{
	\pd_t\E(t) = & \re\int_{\R^3}|\U|^p\U\wb{\V_t}\dx + \frac14(2-\frac {3p}2)t^{1-\frac {3p}2} \int_{\R^3} |\nabla \W|^2 \dd x \\
	\goe & \re\int_{\R^3}|\U|^p\U\wb{\V_t}\dx.
}
Now, we integrate from $t$ to $T_0$,
\EQ{
\E(t) \loe \E(T_0) - \re\int_t^{T_0}\int_{\R^3}|\U|^p\U\wb{\V_\ta}\dd x\dd \ta.
}
Taking supremum in $t$,
\EQn{\label{eq:fra-energy-3d-identity-1}
	\sup_{t\in(0,T_0]}\E(t)\loe \E(T_0) + |\int_{0}^{T_0}\int_{\R^3}|\U|^p\U\wb{\V_t}\dd x\dd t|.
}
Using $i\pd_t\V=-\De \V$ and integration by parts,
\EQn{\label{eq:fra-energy-3d-identity-2}
	|\int_{0}^{T_0}\int_{\R^3}|\U|^p\U\wb{\V_t}\dd x\dd t| \lsm & \int_{0}^{T_0}\int_{\R^3}|\U|^p |\nabla\U\cdot\nabla\V|\dd x\dd t \\
	\lsm & \int_{0}^{T_0}\int_{\R^3}|\U|^p |\nabla\W\cdot\nabla\V|\dd x\dd t \\
	& + \int_{0}^{T_0}\int_{\R^3}|\U|^p|\nabla\V\cdot\nabla\V|\dd x\dd t.
}
Now, note that for $\frac23<p<\frac43$, we claim that
\EQn{\label{eq:bound-for-up}
	|\U|^p \lsm \sum_{N\in2^\N}N^{0+}|\U_N|^p.
}
In fact, we first make a dyadic frequency decomposition $|\U|^p = |\sum_{N\in2^\N}\U_N|^p$. Next, if $\frac23<p\le1$, then by $l_N^p\subset l_N^1$,
\EQ{
|\sum_{N\in2^\N}\U_N|^p \lsm \sum_{N\in2^\N}|\U_N|^p.
}
If $1<p<\frac43$, then by H\"older's inequality,
\EQ{
|\sum_{N\in2^\N}\U_N|^p = |\sum_{N\in2^\N}N^{0-}N^{0+}\U_N|^p \lsm \sum_{N\in2^\N}N^{0+}|\U_N|^p.
}
This proves \eqref{eq:bound-for-up}. Now, using \eqref{eq:bound-for-up}, we can make the frequency decomposition
\EQn{\label{eq:fra-energy-3d-identity-decomposition-1}
|\nabla\W||\nabla\V||\U|^p \lsm & \sum_{N\lsm N_1} N_1^{0+} |\nabla \W| |\nabla \V_N||\U_{N_1}|^{p} + \sum_{N\gg N_1} N_1^{0+} |\nabla \W| |\nabla \V_N||\U_{N_1}|^{p}\\
\lsm & \sum_{N\lsm N_1} N_1^{0+} |\nabla \W| |\nabla \V_N||\V_{N_1}|^{p} + \sum_{N\lsm N_1} N_1^{0+} |\nabla \W| |\nabla \V_N||\W_{N_1}|^{p} \\
& + \sum_{N\gg N_1} N_1^{0+} |\nabla \W| |\nabla \V_N||\V_{N_1}|^{p} + \sum_{N\gg N_1} N_1^{0+} |\nabla \W| |\nabla \V_N||\W_{N_1}|^{p}.
}
Moreover, we also have
\EQn{\label{eq:fra-energy-3d-identity-decomposition-2}
|\nabla\V|^2 |\U|^p\le & 2\sum_{N_1\loe N_2} |\nabla\V_{N_1}||\nabla\V_{N_2}||\U|^{p} \\
\lsm & \sum_{N_1\loe N_2} |\nabla\V_{N_1}||\nabla\V_{N_2}||\U_{\gsm N_2}|^{p} \\
& + \sum_{N_1\loe N_2} |\nabla\V_{N_1}||\nabla\V_{N_2}||\V_{\ll N_2}|^{p} \\
& + \sum_{N_1\loe N_2} |\nabla\V_{N_1}||\nabla\V_{N_2}||\W_{\ll N_2}|^{p}.
}
Therefore, applying \eqref{eq:fra-energy-3d-identity-decomposition-1} and \eqref{eq:fra-energy-3d-identity-decomposition-2} to \eqref{eq:fra-energy-3d-identity-2},
\EQnnsub{
|\int_{0}^{T_0}\int_{\R^3}|\U|^p\U\wb{\V_t}\dd x\dd t| \lsm & \sum_{N\lsm N_1} N_1^{0+} \int_{0}^{T_0}\int_{\R^3} |\nabla \W| |\nabla \V_N||\V_{ N_1}|^p \dd x \dd t \label{eq:fra-energy-bound-wv-high-high-v-3d}\\
& + \sum_{N\lsm N_1} N_1^{0+} \int_{0}^{T_0}\int_{\R^3} |\nabla \W| |\nabla \V_N||\W_{ N_1}|^p \dd x \dd t \label{eq:fra-energy-bound-wv-high-high-w-3d}\\
& + \sum_{N_1\ll N} N_1^{0+} \int_{0}^{T_0}\int_{\R^3} |\nabla \W| |\nabla \V_N||\V_{ N_1}|^p \dd x \dd t \label{eq:fra-energy-bound-wv-high-low-v-3d}\\
& + \sum_{N_1\ll N} N_1^{0+} \int_{0}^{T_0}\int_{\R^3} |\nabla \W| |\nabla \V_N||\W_{ N_1}|^p \dd x \dd t \label{eq:fra-energy-bound-wv-high-low-w-3d}\\
& + \sum_{N_1\loe N_2}\int_{0}^{T_0} \int_{\R^3} |\nabla\V_{N_1}||\nabla\V_{N_2}||\U_{\gsm N_2}|^p\dd x\dd t \label{eq:fra-energy-bound-vv-high-high-3d}\\
& + \sum_{N_1\loe N_2}\int_{0}^{T_0} \int_{\R^3} |\nabla\V_{N_1}||\nabla\V_{N_2}||\V_{\ll N_2}|^p\dd x\dd t \label{eq:fra-energy-bound-vv-high-low-v-3d} \\
& + \sum_{N_1\loe N_2}\int_{0}^{T_0} \int_{\R^3} |\nabla\V_{N_1}||\nabla\V_{N_2}||\W_{\ll N_2}|^p\dd x\dd t. \label{eq:fra-energy-bound-vv-high-low-w-3d}
}
Next, we estimate \eqref{eq:fra-energy-bound-wv-high-high-v-3d}-\eqref{eq:fra-energy-bound-vv-high-low-w-3d} one by one. 
Throughout the proof of Proposition \ref{prop:fra-energy-3d}, we always restrict the temporal integral region on $[0,T_0]$. Note that \eqref{eq:fra-energy-bound-wv-high-low-w-3d} is the main term, and requires more sophisticated analysis, thus we reduce the problem to the following two lemmas. 

\begin{lem}[Main term]\label{lem:fra-energy-3d-main}
Suppose that the assumptions in Theorem \ref{thm:frac-weighted-3d} and the bootstrap hypothesis \eqref{eq:fra-energy-bootstrap-hypothesis-3d} hold. Then,
\EQ{
\eqref{eq:fra-energy-bound-wv-high-low-w-3d} \lsm \de_0 N_0^{2(1-s)}.
}
\end{lem}
\begin{lem}[Remainders]\label{lem:fra-energy-3d-other}
Suppose that the assumptions in Theorem \ref{thm:frac-weighted-3d} and the bootstrap hypothesis \eqref{eq:fra-energy-bootstrap-hypothesis-3d} hold. Then,
\EQ{
\eqref{eq:fra-energy-bound-wv-high-high-v-3d}+ \eqref{eq:fra-energy-bound-wv-high-high-w-3d} + \eqref{eq:fra-energy-bound-wv-high-low-v-3d} +  \eqref{eq:fra-energy-bound-vv-high-high-3d} + \eqref{eq:fra-energy-bound-vv-high-low-v-3d} + \eqref{eq:fra-energy-bound-vv-high-low-w-3d} \lsm \de_0 N_0^{2(1-s)}.
}
\end{lem}

Now, we prove that Lemmas \ref{lem:fra-energy-3d-main} and \ref{lem:fra-energy-3d-other} imply Proposition \ref{prop:fra-energy-3d}. Recall that to close the bootstrap argument, we need to prove that \eqref{eq:fra-energy-bootstrap-hypothesis-3d} implies \eqref{eq:fra-energy-bootstrap-conclusion-3d}. Note that only $N_0$ depends on $\de_0$, and $A$ is constant depending on $p,\norm{\jb{x}^su_0}_{L_x^2}$, and $T_0$. This allows us to choose sufficiently small $\de_0=\de_0(A)$ such that 
\EQ{
\de_0 C(A)\le \frac12 A.
}
Then, applying Lemmas \ref{lem:fra-energy-3d-main} and  \ref{lem:fra-energy-3d-other} to the inequality \eqref{eq:fra-energy-3d-identity-1}, we have
\EQ{
	\sup_{t\in(0,T_0]}\E(t) \loe & \E(T_0) + \eqref{eq:fra-energy-bound-wv-high-high-v-3d} + ... +  \eqref{eq:fra-energy-bound-vv-high-low-w-3d} \\
	\loe & AN_0^{2(1-s)} + \de_0 C(A) N_0^{2(1-s)} \\
	\loe & \frac32A N_0^{2(1-s)}.
}
This gives \eqref{eq:fra-energy-bootstrap-conclusion-3d}, and then Proposition \ref{prop:fra-energy-3d} follows. 

Consequently, we have reduce the proof of Theorem \ref{thm:frac-weighted-3d} to Lemmas \ref{lem:fra-energy-3d-main} and \ref{lem:fra-energy-3d-other}.

\subsection{Proof of Lemma \ref{lem:fra-energy-3d-main}}
The main technique is applying the bilinear Strichartz estimates in Lemma \ref{lem:fra-energy-bound-bilinear-3d} to reduce the derivative on $\V$. The proof is divided into two cases: $\frac23<p<0.86$ and $0.86<p<\frac43$. As pointed out in Remark \ref{rem:bilinear-strichartz}, we will apply the first upper bound in Lemma \ref{lem:fra-energy-bound-bilinear-3d} when $\frac23<p<0.86$, and the second upper bound when $0.86<p<\frac43$.

$\bullet$ \textbf{Case I: $\frac23<p<0.86$.} First, we define that
\EQ{
q_1:=\frac{52(p+2)^2}{3(35p^2+60p-8)}\text{, and }r_1:=\frac{13(p+2)^2}{30-4p-11p^2},
}
which will be used in the following. We can check that $2<r_1<6$ when $\frac23<p<0.86$, and $\frac{2}{q_1} + \frac{3}{r_1}=\frac32$, thus $(q_1,r_1)$ is $L^2$-admissible. Now, we turn to the proof for \eqref{eq:fra-energy-bound-wv-high-low-w-3d}. By H\"older's inequality,
\EQ{
\eqref{eq:fra-energy-bound-wv-high-low-w-3d} \lsm & \sum_{N_1\ll N} \sum_{k\in\Z:2^k\loe T_0} N_1^{0+}  \int_{I_k}  \norm{\nabla\W}_{L_x^2} \norm{\nabla\V_{N}\W_{N_1}}_{L_x^2}^{\frac{2}{15}} \norm{\W_{N_1}}_{L_x^2}^{\frac{5p+4}{15(p+2)}++} \\
& \qquad\qquad \qquad \cdot\norm{\nabla\V_{N}}_{L_x^{r_1}}^{\frac{13}{15}} \norm{\W_{N_1}}_{L_x^{p+2}}^{p-\frac{2}{15}-\frac{5p+4}{15(p+2)}--} \dd t \\
\lsm & \sum_{N_1\ll N} \sum_{k\in\Z:2^k\loe T_0} N_1^{0+} 2^{(\frac{14}{15}-\frac{13}{15q_1})k}   \norm{\nabla\W}_{L_t^\I L_x^2} \norm{\nabla\V_{N}\W_{N_1}}_{L_{t,x}^2}^{\frac{2}{15}} \norm{\W_{N_1}}_{L_t^\I L_x^2}^{\frac{5p+4}{15(p+2)}++} \\
& \qquad\qquad \qquad \cdot\norm{\nabla\V_{N}}_{L_t^{q_1}L_x^{r_1}}^{\frac{13}{15}} \norm{\W_{N_1}}_{L_t^\I L_x^{p+2}}^{p-\frac{2}{15}-\frac{5p+4}{15(p+2)}--},
}
where $t$ is taken on $I_k$. Here we use the notation $a++:=a+\ep_1+\ep_2$ for some sufficiently small constant $\ep_1>0$ and $\ep_2>0$. 
By the first upper bound in Lemma \ref{lem:fra-energy-bound-bilinear-3d} and \eqref{eq:fra-energy-bound-w-3d}, for $N_1\ll N$,
\EQ{
&\norm{\nabla\V_{N}\W_{N_1}}_{L_{t,x}^2}^{\frac{2}{15}} \norm{\W_{N_1}}_{L_t^\I L_x^2}^{\frac{5p+4}{15(p+2)}++} \\
\lsm & \brkb{\de_0 N^{\frac12-s} 2^{(\frac32p-1)k} N_1^{\frac{5p+4}{2(p+2)}} N_0^{\frac{2(p+1)}{p+2}(1-s)} }^{\frac{2}{15}} \norm{\W_{N_1}}_{L_t^\I L_x^2}^{\frac{5p+4}{15(p+2)}++} \\
\lsm & \de_0^{\frac{2}{15}} N_1^{0--} N^{-\frac{1}{15}+\frac{2}{15}(1-s)} 2^{\frac{2}{15}(\frac32p-1)k} N_0^{\frac{4(p+1)}{15(p+2)}(1-s)}\norm{\nabla\W_{N_1}}_{L_t^\I L_x^2}^{\frac{5p+4}{15(p+2)}++} \\
\lsm & \de_0^{\frac{2}{15}} N_1^{0--} N^{-\frac{1}{15}+\frac{2}{15}(1-s)} 2^{\frac{2}{15}(\frac32p-1)k} N_0^{\frac{4(p+1)}{15(p+2)}(1-s)} \cdot 2^{(\frac{5p+4}{15(p+2)}(\frac34p-1)++)k} N_0^{\frac{5p+4}{15(p+2)}(1-s)++}.
}
Combining this inequality, \eqref{eq:fra-energy-bound-w-3d}, \eqref{eq:fra-energy-bound-bilinear-3d}, and Lemma \ref{lem:fra-linear},
\EQ{
\eqref{eq:fra-energy-bound-wv-high-low-w-3d} \lsm & \de_0 \sum_{N_1\ll N} \sum_{k\in\Z:2^k\loe T_0} N_1^{0-} 2^{(\frac{14}{15}-\frac{13}{15q_1})k} \cdot 2^{(\frac34p-1)k}N_0^{1-s}\\
& \qquad\qquad\qquad\quad \cdot N^{-\frac{1}{15} + \frac{2}{15}(1-s)} 2^{\frac{2}{15}(\frac32p-1)k} N_0^{\frac{4(p+1)}{15(p+2)}(1-s)} \\
& \qquad\qquad\qquad\quad \cdot 2^{(\frac{5p+4}{15(p+2)}(\frac34p-1)++)k} N_0^{\frac{5p+4}{15(p+2)}(1-s)++} \\
& \qquad\qquad\qquad\quad \cdot N^{\frac{13}{15}(1-s)} \cdot N_0^{(p-\frac{2}{15}-\frac{5p+4}{15(p+2)})\frac{2}{p+2}(1-s)--} \\
\lsm & \de_0 \sum_{N_1\ll N} \sum_{k\in\Z:2^k\loe T_0} N_1^{0-} N^{\frac{14}{15}-s} 2^{(\frac{72p^3+133p^2-32p-56}{60(p+2)^2}++)k} N_0^{\frac{2(9p^2+22p+10)}{5(p+2)^2}(1-s)++}.
}
We can check that for $\frac23<p<0.86$,
\EQ{
0.007<\frac{72p^3+133p^2-32p-56}{60(p+2)^2}<0.124\text{, and }1.61<\frac{2(9p^2+22p+10)}{5(p+2)^2}<1.74.
}
Therefore, we have that for $\frac23<p<0.86$,
\EQn{\label{esti:fra-energy-bound-wv-high-low-w-3d-1}
\eqref{eq:fra-energy-bound-wv-high-low-w-3d} \lsm & \de_0 \sum_{N_1\ll N} \sum_{k\in\Z:2^k\loe T_0} N_1^{0-} N^{\frac{14}{15}-s} 2^{0.007\times k} N_0^{1.74\times(1-s)} \\
\lsm & \de_0  N_0^{2(1-s)}.
}

$\bullet$ \textbf{Case II: $0.86\le p\le \frac{52}{45}$.}
Now, we define that
\EQ{
q_2:=2(p-\frac{4}{15})\text{, and }r_2:=\frac{9}{2}(p-\frac{4}{15}),
}
Then, we can verify that $4\le q_2<\I$ and $2<r_2\le 4$ for $0.86\le p\le \frac{52}{45}$. By H\"older's inequality,
\EQ{
\eqref{eq:fra-energy-bound-wv-high-low-w-3d} \lsm & \sum_{N_1\ll N} \sum_{k\in\Z:2^k\loe T_0} N_1^{0+}   \norm{\nabla\W}_{L_t^\I L_x^2} \norm{\nabla\V_{N}\W_{N_1}}_{L_{t,x}^2}^{\frac{2}{15}} \norm{\W_{N_1}}_{L_t^\I L_x^2}^{\frac{2}{15}} \\
& \qquad\qquad \qquad \cdot\norm{\nabla\V_{N}}_{L_t^2L_x^{6}}^{\frac{13}{15}} \norm{\W_{N_1}}_{L_t^{q_2}L_x^{r_2}}^{p-\frac{4}{15}},
}
where $t$ is taken on $I_k$. By \eqref{eq:fra-energy-bound-w-3d}, the second upper bound in Lemma \ref{lem:fra-energy-bound-bilinear-3d}, and Lemma \ref{lem:fra-linear},
\EQ{
& \norm{\nabla\W}_{L_t^\I L_x^2} \norm{\nabla\V_{N}\W_{N_1}}_{L_{t,x}^2}^{\frac{2}{15}} \norm{\W_{N_1}}_{L_t^\I L_x^2}^{\frac{2}{15}} \norm{\nabla\V_{N}}_{L_t^2L_x^{6}}^{\frac{13}{15}} \\
\lsm & 2^{(\frac34p-1)k} N_0^{1-s} \cdot \brkb{\de_0 N_1 N^{\frac12-s} 2^{\brkb{(\frac12p+\frac53)(\frac34p-1)-\frac12p+\frac56}k}  N_0^{\frac{3p-2+2\ep}{6}(1-s)}}^{\frac{2}{15}} \\
& \cdot N_1^{-\frac2{15}}\brkb{2^{(\frac34p-1)k} N_0^{1-s}}^{\frac2{15}} \cdot \de_0^{\frac{13}{15}} N^{\frac{13}{15}(1-s)} \\
\lsm & \de_0 N^{-\frac{1}{15} + (1-s)} 2^{\frac{9p^2+159p-224}{180}k} N_0^{\frac{3p+49+2\ep}{45}(1-s)}.
}
Using this inequality,
\EQ{
\eqref{eq:fra-energy-bound-wv-high-low-w-3d} \lsm & \de_0 \sum_{N_1\ll N} \sum_{k\in\Z:2^k\loe T_0}  N^{-\frac{1}{15} + (1-s) +} 2^{\frac{9p^2+159p-224}{180}k} N_0^{\frac{3p+49+2\ep}{45}(1-s)} \norm{\W_{N_2}}_{L_t^{q_2}L_x^{r_2}}^{p-\frac{4}{15}}.
}
Note that by $s>\frac{14}{15}$,
\EQ{
\sum_{N_1\ll N} N^{-\frac{1}{15} + (1-s) +} \lsm \sum_{N_1\in2^\N} N_1^{-\frac{1}{15} + (1-s) +} \lsm 1,
}
then we have
\EQn{\label{esti:fra-energy-bound-wv-high-low-w-3d-2-1}
\eqref{eq:fra-energy-bound-wv-high-low-w-3d} \lsm & \de_0  \sum_{k\in\Z:2^k\loe T_0} 2^{\frac{9p^2+159p-224}{180}k} N_0^{\frac{3p+49+2\ep}{45}(1-s)} \sup_{N_1\in2^\N}\norm{\W_{N_1}}_{L_t^{q_2}L_x^{r_2}}^{p-\frac{4}{15}}.
}

Next, we will give the bound for $\norm{\W_{N_2}}_{L_t^{q_2}L_x^{r_2}}$ and \eqref{eq:fra-energy-bound-wv-high-low-w-3d} in three subcases: $0.86\le p<\frac{32}{35}$,  $\frac{32}{35}\le p \le \frac{52}{45}$, and $\frac{52}{45}<p<\frac43$. 

Before starting the proof, we give a rough description of the difference among the above subcases. This is due to the value of $r_2$, and we need to use different norms to control $L_x^{r_2}$-norm by interpolation. 

In the first two cases when $0.86\le p \le \frac{52}{45}$, we observe that $2<r_2\le 4$, and that the exponent $p$ is comparably small, then the estimate for time integral is still the main issue in the energy increment. Therefore, for this purpose, we use the $L_x^{p+2}$ estimate from \eqref{eq:fra-energy-bound-w-3d} in the interpolation inequality, which does not create any temporal singularity.

If $\frac{52}{45}<p<\frac43$, the main difficulty is to control the increase of $N_0$, irrespective of the time singularity. Therefore, we employ the interaction Morawetz estimate in the interpolation inequality, which has better bound on $N_0$ than the energy estimate.

Now, we give the concrete estimates for $\norm{\W_{N_2}}_{L_t^{q_2}L_x^{r_2}}$ and \eqref{eq:fra-energy-bound-wv-high-low-w-3d} as follows:

$\bullet$ \textit{Case II A:} $0.86\le p<\frac{32}{35}$. In this case, we have that
\EQ{
2<r_2<p+2.
}
By H\"older's inequality in $t$ and interpolation,
\EQ{
\norm{\W_{N_2}}_{L_t^{q_2}L_x^{r_2}}^{p-\frac{4}{15}} \lsm & 2^{\frac12 k} \norm{\W_{N_2}}_{L_t^{\I}L_x^{r_2}}^{p-\frac{4}{15}} \\
\lsm & 2^{\frac12 k} \brkb{ \norm{\W_{N_2}}_{L_t^{\I}L_x^{2}}^{\frac{2(32-35p)}{3p(15p-4)}} \norm{\W_{N_2}}_{L_t^{\I}L_x^{p+2}}^{\frac{(p+2)(45p-32)}{3p(15p-4)}} }^{p-\frac{4}{15}},
}
where the parameter satisfies that for $0.86\le p<\frac{32}{35}$,
\EQ{
	0.83<\frac{(p+2)(45p-32)}{3p(15p-4)}<1.
}
Then, by \eqref{eq:fra-energy-bound-w-3d} and \eqref{eq:fra-mass-bound-w-3d},
\EQn{\label{esti:fra-energy-bound-wv-high-low-w-3d-2-q2r2-1}
\norm{\W_{N_2}}_{L_t^{q_2}L_x^{r_2}}^{p-\frac{4}{15}} \lsm 2^{\frac12 k} N_0^{\frac{2(45p-32)}{45p}(1-s)}.
}
Now, we combine \eqref{esti:fra-energy-bound-wv-high-low-w-3d-2-1} and \eqref{esti:fra-energy-bound-wv-high-low-w-3d-2-q2r2-1},
\EQ{
\eqref{eq:fra-energy-bound-wv-high-low-w-3d} \lsm & \de_0  \sum_{k\in\Z:2^k\loe T_0} 2^{\frac{9p^2+159p-224}{180}k} N_0^{\frac{3p+49+2\ep}{45}(1-s)} \cdot 2^{\frac12 k} N_0^{\frac{2(45p-32)}{45p}(1-s)} \\
\lsm & \de_0  \sum_{k\in\Z:2^k\loe T_0} 2^{\frac{9p^2+159p-134}{180}k} N_0^{\frac{3p^2+139p-64+2\ep}{45p}(1-s)}.
}
For $0.86\le p<\frac{32}{35}$, we can verify that
\EQ{
0.052<\frac{9p^2+159p-134}{180}<0.1\text{, and }1.49<\frac{3p^2+139p-64}{45p}<1.6.
}
Then, for $0.86\le p<\frac{32}{35}$
\EQn{\label{esti:fra-energy-bound-wv-high-low-w-3d-2-esti-1}\eqref{eq:fra-energy-bound-wv-high-low-w-3d} 
\lsm & \de_0  \sum_{k\in\Z:2^k\loe T_0} 2^{0.52\times k} N_0^{1.6\times(1-s)} \lsm \de_0 N_0^{2(1-s)}.
}

$\bullet$ \textit{Case II B:} $\frac{32}{35}\le p \le \frac{52}{45}$. In this case, we have that
\EQ{
p+2\le r_2\le 4.
}
By H\"older's inequality in $t$ and interpolation,
\EQ{
\norm{\W_{N_2}}_{L_t^{q_2}L_x^{r_2}}^{p-\frac{4}{15}} \lsm & 2^{\frac12 k} \norm{\W_{N_2}}_{L_t^{\I}L_x^{r_2}}^{p-\frac{4}{15}} \\
\lsm & 2^{\frac12 k} \brkb{ \norm{\W_{N_2}}_{L_t^{\I}L_x^{p+2}}^{\frac{3(p+2)(8-5p)}{(15p-4)(4-p)}} \norm{\W_{N_2}}_{L_t^{\I}L_x^{6}}^{\frac{2(35p-32)}{(15p-4)(4-p)}} }^{p-\frac{4}{15}},
}
where the parameter satisfies that for $\frac{32}{35}\le p \le \frac{52}{45}$,
\EQ{
0.55<\frac{3(p+2)(8-5p)}{(15p-4)(4-p)}\le1.
}
Then, by Bernstein's inequality and \eqref{eq:fra-energy-bound-w-3d},
\EQn{\label{esti:fra-energy-bound-wv-high-low-w-3d-2-q2r2-2}
\norm{\W_{N_2}}_{L_t^{q_2}L_x^{r_2}}^{p-\frac{4}{15}} \lsm & 2^{\frac12 k}  \norm{\W_{N_2}}_{L_t^{\I}L_x^{p+2}}^{\frac{(p+2)(8-5p)}{5(4-p)}} \norm{\nabla\W_{N_2}}_{L_t^{\I}L_x^{2}}^{\frac{2(35p-32)}{15(4-p)}} \\
\lsm & 2^{\frac12k} N_0^{\frac{2(8-5p)}{5(4-p)}(1-s)} \cdot  2^{\frac{2(35p-32)}{15(4-p)}(\frac34p-1)k} N_0^{\frac{2(35p-32)}{15(4-p)}(1-s)} \\
\lsm & 2^{\frac{105p^2-251p+188}{30(4-p)}k} N_0^{\frac{8(5p-2)}{15(4-p)}(1-s)}.
}
Now, we combine \eqref{esti:fra-energy-bound-wv-high-low-w-3d-2-1} and \eqref{esti:fra-energy-bound-wv-high-low-w-3d-2-q2r2-2},
\EQ{
\eqref{eq:fra-energy-bound-wv-high-low-w-3d} \lsm & \de_0  \sum_{k\in\Z:2^k\loe T_0} 2^{\frac{9p^2+159p-224}{180}k} N_0^{\frac{3p+49+2\ep}{45}(1-s)} \cdot 2^{\frac{105p^2-251p+188}{30(4-p)}k} N_0^{\frac{8(5p-2)}{15(4-p)}(1-s)} \\
	\lsm & \de_0  \sum_{k\in\Z:2^k\loe T_0} 2^{\frac{9p^3-507p^2+646p-232}{180(p-4)}k} N_0^{\brkb{\frac{3p^2-83p-148}{45(p-4)}+\frac{2}{45}\ep}(1-s)}.
}
For $\frac{32}{35}\le p \le \frac{52}{45}$, we can verify that
\EQ{
0.1<\frac{9p^3-507p^2+646p-232}{180(p-4)}<0.3\text{, and }1.59<\frac{3p^2-83p-148}{45(p-4)}<1.88.
}
Then, for $\frac{32}{35}\le p \le \frac{52}{45}$,
\EQn{\label{esti:fra-energy-bound-wv-high-low-w-3d-2-esti-2}
\eqref{eq:fra-energy-bound-wv-high-low-w-3d} 
\lsm & \de_0  \sum_{k\in\Z:2^k\loe T_0} 2^{0.1\times k} N_0^{1.88\times(1-s)} \lsm \de_0 N_0^{2(1-s)}.
}

$\bullet$ \textit{Case II C:} $\frac{52}{45}<p<\frac43$. In this case, we have that
\EQ{
4<r_2<\frac{24}{5}.
}
By interpolation, H\"older's inequality in $t$, and Bernstein's inequality,
\EQ{
\norm{\W_{N_2}}_{L_t^{q_2}L_x^{r_2}} \lsm & \normb{ \norm{\W_{N_2}}_{L_x^{4}}^{\frac{6(8-5p)}{15p-4}} \norm{\W_{N_2}}_{ L_x^6}^{\frac{45p-52}{15p-4}}}_{L_t^{q_2}} \\
\lsm & 2^{\frac{3(5p-3)}{2(15p-4)}k} \norm{\W_{N_2}}_{L_{t,x}^{4}}^{\frac{6(8-5p)}{15p-4}} \norm{\nabla\W_{N_2}}_{L_t^\I L_x^2}^{\frac{45p-52}{15p-4}},,
}
where the parameter satisfies that for $\frac{52}{45}<p<\frac43$,
\EQ{
\frac12<\frac{6(8-5p)}{15p-4}<1.
}
Then, by \eqref{eq:fra-energy-bound-w-3d} and Proposition \ref{prop:interaction-morawetz},
\EQn{\label{esti:fra-energy-bound-wv-high-low-w-3d-2-q2r2-3}
\norm{\W_{N_2}}_{L_t^{q_2}L_x^{r_2}}^{p-\frac{4}{15}} \lsm & 2^{\frac{5p-3}{10}k} \cdot \brkb{2^{\frac14(\frac34p-1)k}N_0^{\frac14(1-s)}}^{\frac{2(8-5p)}{5}} \cdot \brkb{2^{(\frac34p-1)k}N_0^{1-s}}^{\frac{45p-52}{15}} \\
\lsm & 2^{\frac{225p^2-480p+284}{120}k} N_0^{\frac{15p-16}{6}(1-s)}.
}
Now, we combine \eqref{esti:fra-energy-bound-wv-high-low-w-3d-2-1} and \eqref{esti:fra-energy-bound-wv-high-low-w-3d-2-q2r2-3},
\EQ{
\eqref{eq:fra-energy-bound-wv-high-low-w-3d} \lsm & \de_0  \sum_{k\in\Z:2^k\loe T_0} 2^{\frac{9p^2+159p-224}{180}k} N_0^{\frac{3p+49+2\ep}{45}(1-s)} \cdot 2^{\frac{225p^2-480p+284}{120}k} N_0^{\frac{15p-16}{6}(1-s)} \\
\lsm & \de_0  \sum_{k\in\Z:2^k\loe T_0} 2^{\frac{693p^2-1122p+404}{360}k} N_0^{\brkb{\frac{231p-142}{90}+\frac{2}{45}\ep}(1-s)}.
}
For $\frac{52}{45}<p<\frac43$, we can verify that
\EQ{
0.09<\frac{693p^2-1122p+404}{360}<0.39\text{, and }1.38<\frac{231p-142}{90}<1.85.
}
Then, for $\frac{52}{45}<p<\frac43$,
\EQn{\label{esti:fra-energy-bound-wv-high-low-w-3d-2-esti-3}
\eqref{eq:fra-energy-bound-wv-high-low-w-3d} \lsm & \de_0  \sum_{k\in\Z:2^k\loe T_0} 2^{0.09\times k} N_0^{1.85\times(1-s)} \lsm \de_0 N_0^{2(1-s)}.
}

Finally, Lemma \ref{lem:fra-energy-3d-main} follows from \eqref{esti:fra-energy-bound-wv-high-low-w-3d-1}, \eqref{esti:fra-energy-bound-wv-high-low-w-3d-2-esti-1}, \eqref{esti:fra-energy-bound-wv-high-low-w-3d-2-esti-2}, and \eqref{esti:fra-energy-bound-wv-high-low-w-3d-2-esti-3}.

\subsection{Proof of Lemma \ref{lem:fra-energy-3d-other}}

Now, we consider the remaining terms.

$\bullet$ \textbf{Estimate of \eqref{eq:fra-energy-bound-wv-high-high-v-3d}.} Note that $(2,6)$ and $(\frac{4p}{3p-2-\ep},\frac{6p}{2+\ep})$ are $L^2$-admissible.
By H\"older's inequality and Lemma \ref{lem:fra-linear},
\EQ{
\eqref{eq:fra-energy-bound-wv-high-high-v-3d} 
\lsm & \sum_{N\lsm N_1} N_1^{0+} \norm{\nabla \W}_{L_t^{\frac{4}{4-3p+\ep}}L_x^2} \norm{\nabla \V_N}_{L_t^2 L_x^6} \norm{\V_{N_1}}_{L_t^{\frac{4p}{3p-2-\ep}}L_x^{3p}}^p \\
\lsm &  \sum_{N\lsm N_1} \norm{\nabla \W}_{L_t^{\frac{4}{4-3p+\ep}}L_x^2} N^{1-s}N_1^{-sp+\frac{\ep}{2}+}\norm{\jb{\nabla}^s \V_N}_{L_t^2 L_x^6} \norm{\jb{\nabla}^s\V_{N_1}}_{L_t^{\frac{4p}{3p-2-\ep}}L_x^{\frac{6p}{2+\ep}}}^p \\
\lsm & \de_0^{p+1}\norm{\nabla \W}_{L_t^{\frac{4}{4-3p+\ep}}L_x^2}\sum_{N\lsm N_1}  N^{1-s}N_1^{-sp+\frac{\ep}{2}+}.
}
By \eqref{eq:fra-energy-bound-w-3d},
\EQ{
\norm{\nabla \W}_{L_t^{\frac{4}{4-3p+\ep}}L_x^2} \lsm & N_0^{1-s} \normb{t^{\frac{3p}{4}-1}}_{L_t^{\frac{4}{4-3p+\ep}}} \\
\lsm & N_0^{1-s} \brkb{\int_{0}^{T_0} t^{-\frac{4-3p}{4-3p+\ep}}\dd t}^{\frac{4-3p+\ep}{4}} \\
\lsm & N_0^{1-s} T_0^{\ep/4} \\
\lsm & N_0^{1-s},
}
and by $s>\frac34>\frac1{2p}$,
\EQ{
\sum_{N\lsm N_1}  N^{1-s}N_1^{-sp+\frac{\ep}{2}+} \lsm \sum_{N_1\in 2^\N}  N_1^{1-2sp+\frac{\ep}{2}+} \lsm 1.
}
Combining the above two inequalities,
\EQn{\label{esti:fra-energy-bound-wv-high-high-v-3d}
\eqref{eq:fra-energy-bound-wv-high-high-v-3d} 
\lsm \de_0^{p+1} N_0^{1-s}.
}

$\bullet$ \textbf{Estimate of \eqref{eq:fra-energy-bound-wv-high-high-w-3d}.} 
By dyadic decomposition in time integral, we have that 
\begin{align*}
\eqref{eq:fra-energy-bound-wv-high-high-w-3d}
\lsm  \sum_{N\lsm N_1} \sum_{k\in\Z:2^k\loe T_0} N_1^{0+} \int_{I_k} \int_{\R^3} |\nabla \W| |\nabla \V_N||\W_{ N_1}|^p \dd x \dd t,
\end{align*}
where $I_k$ denotes $[2^{k-1},2^k]$. We first deal with the case when $\frac23<p\loe1$. By H\"older's inequality,
\EQ{
\eqref{eq:fra-energy-bound-wv-high-high-w-3d} 
\lsm & \sum_{N\lsm N_1} \sum_{k\in\Z:2^k\loe T_0} N_1^{0+} \norm{\nabla\W}_{L_t^\I L_x^2} \norm{\nabla\V_N}_{L_t^{1}L_x^{\frac{2(p+2)}{2-p}}} \norm{\W_{N_1}}_{L_t^\I L_x^{p+2}}^p,
}
where $t$ is taken over $I_k$. By \eqref{eq:fra-energy-bound-w-3d}, for any $k\in\Z$ with $2^k\loe T_0$,
\EQ{
\norm{\nabla\W}_{L_{t}^\I L_x^2} \lsm 2^{(\frac{3p}{4}-1)k} N_0^{1-s}.
}
Noting that for $\frac23<p\loe 1$, $(\frac{2(p+2)}{3p},\frac{2(p+2)}{2-p})$ is $L^2$ admissible, by H\"older's inequality and Lemma \ref{lem:fra-linear},
\EQ{
\norm{\nabla\V_N}_{L_t^{1}L_x^{\frac{2(p+2)}{2-p}}} \lsm & 2^{\frac{4-p}{2(p+2)}k} \norm{\nabla\V_N}_{L_t^{\frac{2(p+2)}{3p}}L_x^{\frac{2(p+2)}{2-p}}} \\
\lsm & 2^{\frac{4-p}{2(p+2)}k} N^{1-s}.
}
By Bernstein's inequality and \eqref{eq:fra-energy-bound-w-3d},
\EQ{
\norm{\W_{N_1}}_{L_t^\I L_x^{p+2}} \lsm N_1^{-\frac{4-p}{2(p+2)}} 2^{(\frac{3p}{4}-1)k} N_0^{1-s}.
}
Interpolation this with \eqref{eq:fra-energy-bound-w-3d},
\EQ{
\norm{\W_{N_1}}_{L_t^\I L_x^{p+2}(I_k)} \lsm N_1^{-\al_1\frac{4-p}{2(p+2)}} 2^{\al_1(\frac{3p}{4}-1)k} N_0^{\frac{2+\al_1 p}{p+2}(1-s)},
}
where $\al_1$ is defined by 
\EQ{
\al_1:=\frac{3p}{(p+2)(4-3p)}-.
}
Note that for $\frac23<p\loe1$, we have that $\frac38<\al_1<1$. The choice of $\al_1$ also gives
\EQ{
(1+\al_1 p)(\frac{3p}{4}-1) + \frac{4-p}{2(p+2)} = 0+.
}
Therefore, combining the above inequalities,
\EQ{
\eqref{eq:fra-energy-bound-wv-high-high-w-3d} 
\lsm & \sum_{N\lsm N_1} \sum_{k\in\Z:2^k\loe T_0} N_1^{0+} 2^{(\frac{3p}{4}-1)k} N_0^{1-s} \cdot  2^{\frac{4-p}{2(p+2)}k} N^{1-s} \\
& \cdot N_1^{-\al_1\frac{p(4-p)}{2(p+2)}} 2^{\al_1 p(\frac{3p}{4}-1)k} N_0^{\frac{p(2+\al_1 p)}{p+2}(1-s)} \\
\lsm & \sum_{N\lsm N_1}\sum_{k\in\Z:2^k\loe T_0}  2^{\brkb{(1+\al_1 p)(\frac{3p}{4}-1) + \frac{4-p}{2(p+2)}}k} N_1^{-\al_1\frac{p(4-p)}{2(p+2)}+} N^{1-s} N_0^{\frac{\al_1 p^2 + 3p + 2}{p+2}(1-s)} \\
\lsm & \sum_{N_1\in2^\N}  N_1^{-\frac{3p^2(4-p)}{2(4-3p)(p+2)^2}+1-s+} N_0^{-\frac{2(3p^3+6p^2-10p-8)}{(p+2)^2(4-3p)}(1-s)}.
}
Now, we can verify that for $\frac23<p\loe 1$,
\EQ{
-\frac12\loe-\frac{3p^2(4-p)}{2(4-3p)(p+2)^2}<-\frac{5}{32},\quad \text{and }\frac{25}{16}< -\frac{2(3p^3+6p^2-10p-8)}{(p+2)^2(4-3p)} \le 2.
}
Then, for any $s>\frac{27}{32}$,
\EQ{
\eqref{eq:fra-energy-bound-wv-high-high-w-3d} 
\lsm & \sum_{N_1\in2^\N}  N_1^{-\frac{5}{32}+1-s+} N_0^{2(1-s)}.
}

We then handle the case when $1 < p<\frac43$. First, we define the exponents $q_1$, $r_1$, and $r_2$ as follows: let 
\EQn{\label{defn:q1}
\frac{1}{q_1} := \frac{5p+1}{20}.
}
Then, we can check that $\frac{60}{23}<q_1 \loe \frac{10}{3}$ for $1<p<\frac43$. Particularly, we have that $q_1>2$, so it can be used to define an admissible pair.
Now, we take $r_1$ such that $(q_1,r_1)$ is $L^2$-admissible, thus by the definition of $q_1$ in \eqref{defn:q1},
\EQn{\label{defn:r1}
\frac{1}{r_1} = \frac{14-5p}{30}.
}
Define $r_2$ by 
\EQn{\label{eq:r1r2}
\frac{1}{2} = \frac{1}{r_1} + \frac{p}{r_2}.
}
By \eqref{defn:r1} and \eqref{eq:r1r2},
\EQn{\label{defn:r2}
\frac{1}{r_2} =\frac{5p+1}{30p}.
}
We can check that $\frac{10}{3}<r_2(p) < \frac{120}{23}$ for any $1<p<\frac43$. We also note that by \eqref{defn:q1},
\EQn{\label{defn:q1'}
\frac{1}{q_1'} = \frac{19-5p}{20}.
}
Next, we give the estimate of \eqref{eq:fra-energy-bound-wv-high-high-w-3d} under the condition $1<p<\frac43$. By the relation \eqref{eq:r1r2} and H\"older's inequality,
\EQ{
\eqref{eq:fra-energy-bound-wv-high-high-w-3d} \lsm & \sum_{N\lsm N_1} N_1^{0+} \int_{0}^{T_0}\int_{\R^3} |\nabla \W| |\nabla \V_N||\W_{ N_1}|^p \dd x \dd t  \\
\lsm & \sum_{N\lsm N_1} N_1^{0+} \frac{N^{\frac{1}{10}}}{N_1^{\frac{1}{10p}}} \normb{\jb{\nabla}^{\frac{9}{10}}\V_N}_{L_t^{q_1}L_x^{r_1}} \norm{\norm{\nabla \W}_{L_x^2} \normb{\jb{\nabla}^{\frac{1}{10p}}\W_{N_1}}_{L_x^{r_2}}^p}_{L_t^{q_1'}} \\
\lsm & \sum_{N:N\gsm N_0} N^{0+} \normb{\jb{\nabla}^{\frac{9}{10}}\V_N}_{L_t^{q_1}L_x^{r_1}} \norm{\norm{\nabla \W}_{L_x^2} \normb{\jb{\nabla}^{\frac{1}{10p}}\W}_{L_x^{r_2}}^p}_{L_t^{q_1'}}.
}
Noting that by  \eqref{defn:r2}, we have that
\EQ{
\frac{1}{10p}-\frac{3}{r_2} = -\frac12,
}
then we have the Sobolev inequality:
\EQ{
\normb{\jb{\nabla}^{\frac{1}{10p}}\W}_{L_x^{r_2}} \lsm \norm{\jb{\nabla} \W}_{L_x^2}.
}
Then, by Lemma \ref{lem:fra-linear},
\EQ{
\eqref{eq:fra-energy-bound-wv-high-high-w-3d}
\lsm & \sum_{N:N\gsm N_0} N^{0+} \normb{\jb{\nabla}^{\frac{9}{10}}\V_N}_{L_t^{q_1}L_x^{r_1}} \normb{\norm{\nabla \W}_{L_x^2}^{1+p}}_{L_t^{q_1'}} \\
\lsm & \sum_{N:N\gsm N_0} N^{0+} N^{\frac{9}{10}-s} \normb{\jb{\nabla}^{s}\V_N}_{L_t^{q_1}L_x^{r_1}} \normb{\brkb{t^{\frac{3p}{4}-1}}^{1+p}}_{L_t^{q_1'}} N_0^{(1+p)(1-s)} \\
\lsm & \de_0 \sum_{N:N\gsm N_0} N^{0+} N^{\frac{9}{10}-s} \normb{\brkb{t^{\frac{3p}{4}-1}}^{1+p}}_{L_t^{q_1'}} N_0^{(1+p)(1-s)}.
}
For $1<p<\frac43$, we can verify that
\EQ{
(\frac{3p}{4}-1)(1+p)q_1'=\frac{5(3p-4)(p+1)}{19-5p}>-\frac57.
}
Then,
\EQ{
\normb{\brkb{t^{\frac{3p}{4}-1}}^{1+p}}_{L_t^{q_1'}} \loe C(T_0)\lsm 1.
}
Noting that  $s>\frac{37}{40}>\frac{10p-1}{10p}$,
\EQ{
\sum_{N:N\gsm N_0} N^{0+} N^{\frac{9}{10}-s}  N_0^{(1+p)(1-s)} \lsm & N_0^{0+} N_0^{\frac{9}{10}-s}  N_0^{(1+p)(1-s)} \\
\lsm & N_0^{2(1-s)} N_0^{-\frac{1}{10} + p(1-s)+} \\
\lsm & N_0^{2(1-s)}.
}
Combining the above two inequalities, we have that
\EQn{\label{esti:fra-energy-bound-wv-high-high-w-3d}
\eqref{eq:fra-energy-bound-wv-high-high-w-3d} \lsm \de_0 N_0^{2(1-s)}.
}

$\bullet$ \textbf{Estimate of \eqref{eq:fra-energy-bound-wv-high-low-v-3d}.}
First, by H\"older's inequality,
\EQ{
\eqref{eq:fra-energy-bound-wv-high-low-v-3d} \lsm & \sum_{N_1\ll N} N_1^{0+} \int_{0}^{T_0}\int_{\R^3} |\nabla \W| |\nabla \V_N||\V_{N_1}|^p \dd x \dd t \\
\lsm & \sum_{N_1\ll N} N_1^{0+} \norm{\nabla \V_N \V_{N_1}}_{L_{t,x}^2}^{\frac15} \norm{\nabla \V_N}_{L_t^\I L_x^2}^{\frac45} \norm{\nabla \W}_{L_t^{\frac{4}{4-3p+\ep}}L_x^2} \norm{\V_{N_1}}_{L_t^{\frac{4(p-\frac15)}{3p-\frac25-\ep}}L_x^\I}^{p-\frac15}.
}
For $N,N_1\in2^\N$ with $N_1\ll N$, by the bilinear Strichartz estimate in Lemma \ref{lem:bilinearstrichartz},
\EQn{\label{esti:fra-energy-bound-wv-high-low-v-3d-bi}
\norm{\nabla \V_N \V_{N_1}}_{L_{t,x}^2} \lsm & \frac{N_1}{N^{1/2}} \norm{\nabla P_N\V_+}_{L_x^2} \norm{P_{N_1}\V_+}_{L_x^2} \\
\lsm & \frac{N_1^{1-s}}{N^{-1/2+s}} \norm{\jb{\nabla}^s P_N\V_+}_{L_x^2} \norm{\jb{\nabla}^sP_{N_1}\V_+}_{L_x^2} \\
\lsm & \de_0^2 N_1^{1-s} N^{1-s} N^{-1/2}.
}
Note that for any $\frac23<p<\frac43$,
\EQ{
\frac{4(p-\frac15)}{3p-\frac25-\ep}< 2,
}
then by H\"older's and Sobolev's inequalities,
\EQ{
\norm{\V_{N_1}}_{L_t^{\frac{4(p-\frac15)}{3p-\frac25-\ep}}L_x^\I} \lsm \norm{\V_{N_1}}_{L_t^{2}L_x^\I} \lsm \normb{\jb{\nabla}^{\frac12+\ep}\V_{N_1}}_{L_t^{2}L_x^6}.
}
Moreover, by \eqref{eq:fra-energy-bound-w-3d},
\EQ{
\norm{\nabla \W}_{L_t^{\frac{4}{4-3p+\ep}}L_x^2} \lsm & \norm{t^{\frac{3p}{4}-1}N_0^{1-s}}_{L_t^{\frac{4}{4-3p+\ep}}} \\
\lsm & \brkb{\int_{0}^{T_0} t^{-\frac{4-3p}{4-3p+\ep}}\dd t}^{\frac{4-3p+\ep}{4}} N_0^{1-s} \\
\lsm & N_0^{1-s}.
}
Combining the above inequalities,
\EQ{
\eqref{eq:fra-energy-bound-wv-high-low-v-3d} 
\lsm & \sum_{N_1\ll N} N_1^{0+} \brkb{\de_0^2 N_1^{1-s} N^{1-s} N^{-1/2}}^{\frac15} \de_0^{\frac45}N^{\frac45(1-s)} N_0^{1-s} \normb{\jb{\nabla}^{\frac12+\ep}\V_{N_1}}_{L_t^{2}L_x^6}^{p-\frac15} \\
\lsm & \de_0^{\frac65}\sum_{N_1\ll N}  N^{1-s}N^{-\frac{1}{10}+} N_0^{1-s} \normb{\jb{\nabla}^{\frac{1-s}{5p-1} + \frac12+\ep}\V_{N_1}}_{L_t^{2}L_x^6}^{p-\frac15} \\
\lsm & \de_0^{\frac65}\sum_{N:N\gsm N_0}  N^{1-s}N^{-\frac{1}{10}+} N_0^{1-s} \normb{\jb{\nabla}^{\frac{1-s}{5p-1} + \frac12+2\ep}\V}_{L_t^{2}L_x^6}^{p-\frac15} \\
\lsm & \de_0^{\frac65}  N_0^{1-s}N_0^{-\frac{1}{10}+} N_0^{1-s} \normb{\jb{\nabla}^{\frac{1-s}{5p-1} + \frac12+2\ep}\V}_{L_t^{2}L_x^6}^{p-\frac15}.
}
Note that by $s>\frac{13}{20}$,
\EQ{
	\frac{1-s}{5p-1} + \frac12+2\ep < \frac37(1-s) + \frac12+2\ep <s.
}
Therefore,
\EQn{\label{esti:fra-energy-bound-wv-high-low-v-3d} 
\eqref{eq:fra-energy-bound-wv-high-low-v-3d} 
\lsm & \de_0^{\frac65}  N_0^{-\frac{1}{10}+} N_0^{2(1-s)} \normb{\jb{\nabla}^{s}\V}_{L_t^{2}L_x^6}^{p-\frac15} \\
\lsm & \de_0^{1+p} N_0^{-\frac{1}{10}+} N_0^{2(1-s)}.
}

$\bullet$ \textbf{Estimate of \eqref{eq:fra-energy-bound-vv-high-high-3d}.} By H\"older's inequality,
\EQ{
\eqref{eq:fra-energy-bound-vv-high-high-3d} \lsm & \sum_{N_1\loe N_2}\int_{0}^{T_0} \int_{\R^3} |\nabla\V_{N_1}||\nabla\V_{N_2}||\U_{\gsm N_2}|^p\dd x\dd t \\
\lsm & \sum_{N_1\loe N_2\lsm N} \norm{\nabla\V_{N_1}}_{L_t^2 L_x^6} \norm{\nabla\V_{N_2}}_{L_t^{\frac{4}{3p-2}} L_x^{\frac{6}{5-3p}}} \norm{\U_{N}}_{L_t^{\frac{4p}{4-3p}}L_x^2}^p.
}
Noting that $(\frac{4}{3p-2},\frac{6}{5-3p})$ is $L^2$-admissible, by Lemma \ref{lem:fra-linear},
\EQ{
\eqref{eq:fra-energy-bound-vv-high-high-3d} \lsm & \sum_{N_1\loe N_2\lsm N} N_1^{1-s} N_2^{1-s} \norm{\jb{\nabla}^s\V_{N_1}}_{L_t^2 L_x^6} \norm{\jb{\nabla}^s\V_{N_2}}_{L_t^{\frac{4}{3p-2}} L_x^{\frac{6}{5-3p}}} \norm{\U_{N}}_{L_t^{\frac{4p}{4-3p}}L_x^2}^p \\
\lsm & \de_0^2\sum_{N\in2^\N} N^{2(1-s)}\normb{\U_{N}}_{L_t^{\frac{4p}{4-3p}}L_x^2}^p \\
\lsm & \de_0^2\normb{\jb{\nabla}^{\frac2p(1-s)+}\U}_{L_t^{\frac{4p}{4-3p}}L_x^2}^p.
}
Since $s>\frac{14}{15}>\frac{2}{2+p}$, we have that $\frac2p(1-s)<s$, then
\EQ{
\normb{\jb{\nabla}^{\frac2p(1-s)+}\U}_{L_t^{\frac{4p}{4-3p}}L_x^2} \lsm \normb{\jb{\nabla}^{s}\V}_{L_t^{\frac{4p}{4-3p}}L_x^2} + \normb{\jb{\nabla}^{\frac2p(1-s)+}\W}_{L_t^{\frac{4p}{4-3p}}L_x^2}.
}
First, by H\"older's inequality in $t$ and Lemma \ref{lem:fra-linear},
\EQ{
\normb{\jb{\nabla}^{s}\V}_{L_t^{\frac{4p}{4-3p}}L_x^2} \lsm \normb{\jb{\nabla}^{s}\V}_{L_t^{\I}L_x^2} \lsm \de_0.
}
Second, by \eqref{eq:fra-energy-bound-w-3d}, interpolation, and $\frac2p(1-s)<s$,
\EQ{
\normb{\jb{\nabla}^{\frac2p(1-s)+}\W}_{L_x^2} \lsm & \normb{\jb{\nabla}\W}_{L_x^2}^{\frac2p(1-s)+} \norm{\W}_{L_x^2}^{1-\frac2p(1-s)-} \\
\lsm & t^{\frac2p(\frac{3p}{4}-1)(1-s)-}N_0^{\frac2p(1-s)^2+} \\
\lsm & t^{\frac2p(\frac{3p}{4}-1)(1-s)-}N_0^{s(1-s)}.
}
Then, by $-\frac2{15}<-2(1-s)<0$,
\EQ{
\normb{\jb{\nabla}^{\frac2p(1-s)+}\W}_{L_t^{\frac{4p}{4-3p}}L_x^2} \lsm & \normb{t^{\frac2p(\frac{3p}{4}-1)(1-s)-}}_{L_t^{\frac{4p}{4-3p}}} N_0^{s(1-s)} \\
\lsm & \brkb{\int_{0}^{T_0} t^{-2(1-s)-}\dd t}^{\frac{4-3p}{4p}} N_0^{s(1-s)} \\
\lsm & N_0^{s(1-s)}.
}
Therefore, we have
\EQn{\label{esti:fra-energy-bound-vv-high-high-3d}
\eqref{eq:fra-energy-bound-vv-high-high-3d} \lsm & \de_0^2\brkb{\normb{\jb{\nabla}^{s}\V}_{L_t^{\frac{4p}{4-3p}}L_x^2} + \normb{\jb{\nabla}^{\frac2p(1-s)+}\W}_{L_t^{\frac{4p}{4-3p}}L_x^2}}^p \\
\lsm & \de_0^2\brkb{\de_0+N_0^{s(1-s)}}^p \\
\lsm & \de_0^2 N_0^{2(1-s)}.
}

$\bullet$ \textbf{Estimate of \eqref{eq:fra-energy-bound-vv-high-low-v-3d}.}
By H\"older's inequality,
\EQ{
\eqref{eq:fra-energy-bound-vv-high-low-v-3d} \lsm & \sum_{N_1\loe N_2,N\ll N_2}\int_{t_0}^{T_0} \int_{\R^3} |\nabla\V_{N_1}||\nabla\V_{N_2}||\V_{N}|^p\dd x\dd t \\
\lsm & \sum_{N_1\loe N_2,N\ll N_2} \norm{\nabla \V_{N_2} \V_{N}}_{L_{t,x}^2}^{\frac{4}{15}} \norm{\nabla \V_{N_2}}_{L_t^\I L_x^2}^{\frac{11}{15}} \norm{\nabla \V_{N_1}}_{L_t^{\I}L_x^2} \norm{\V_{N}}_{L_t^{\frac{15p-4}{13}}L_x^\I}^{p-\frac{4}{15}}.
}
By Lemma \ref{lem:fra-linear}, \eqref{esti:fra-energy-bound-wv-high-low-v-3d-bi}, $\frac{15p-4}{13}<\frac{16}{13}$, H\"older's, and Sobolev's inequality,
\EQn{\label{esti:fra-energy-bound-vv-high-low-v-3d}
\eqref{eq:fra-energy-bound-vv-high-low-v-3d} \lsm & \sum_{N_1\loe N_2,N\ll N_2} \norm{\nabla \V_{N_2} \V_{ N}}_{L_{t,x}^2}^{\frac{4}{15}} \norm{\nabla \V_{N_2}}_{L_t^\I L_x^2}^{\frac{11}{15}} N_1^{1-s}\norm{\jb{\nabla}^s \V_{N_1}}_{L_t^{\I}L_x^2} \norm{\V_{ N}}_{L_t^{2}L_x^\I}^{p-\frac{4}{15}} \\
\lsm & \de_0\sum_{N\ll N_2} \brkb{\de_0^2 N^{1-s} N_2^{1-s} N_2^{-1/2}}^{\frac{4}{15}} \de_0^{\frac{11}{15}} N_2^{\frac{11}{15}(1-s)} N_2^{1-s} \normb{\jb{\nabla}^{\frac12+\ep}\V_{ N}}_{L_t^{2}L_x^6}^{p-\frac{4}{15}} \\
\lsm & \de_0^{\frac{34}{15}}\sum_{N_2:N_2\gsm N_0}  N_2^{-\frac{2}{15} + 2(1-s)} \normb{\jb{\nabla}^{\frac{4(1-s)}{15(p-\frac{4}{15})} + \frac12+2\ep}\V}_{L_t^{2}L_x^6}^{p-\frac{4}{15}}.
}
Since $p>\frac23$ and $s>\frac{14}{15}$, we have that
\EQ{
\frac{4(1-s)}{15(p-\frac{4}{15})} + \frac12 < \frac23(1-s) + \frac12 <\frac{49}{90} <s,
}
then
\EQ{
\eqref{eq:fra-energy-bound-vv-high-low-v-3d} \lsm & \de_0^{\frac{34}{15}} N_0^{-\frac{2}{15} + 2(1-s)} \normb{\jb{\nabla}^{s}\V}_{L_t^{2}L_x^6}^{p-\frac{4}{15}} \\
\lsm & \de_0^{p+2} N_0^{-\frac{2}{15}} N_0^{2(1-s)}.
}

$\bullet$ \textbf{Estimate of \eqref{eq:fra-energy-bound-vv-high-low-w-3d}.} 
We first make the dyadic decomposition for the time integral:
\EQ{
\eqref{eq:fra-energy-bound-vv-high-low-w-3d} \lsm \sum_{k\in\Z:2^k\loe T_0} \sum_{N_1\loe N_2,N_3\ll N_2} N_3^{0+} \int_{I_k} \int_{\R^3} |\nabla\V_{N_1}||\nabla\V_{N_2}||\W_{N_3}|^p\dd x\dd t.
}
We can check that for $2/3<p<4/3$, $(\frac{104}{3(15p-4)},\frac{52}{15(2-p)})$ is $L_x^2$-admissible. Then, by H\"older's inequality, Lemmas \ref{lem:fra-linear} and \ref{lem:fra-energy-bound-bilinear-3d},
\EQ{
\eqref{eq:fra-energy-bound-vv-high-low-w-3d} \lsm & \sum_{k\in\Z:2^k\loe T_0} \sum_{N_1\loe N_2,N_3\ll N_2}  N_3^{0+} 2^{\frac{64-45p}{60}k}\norm{\nabla \V_{N_1}}_{L_t^{\frac{104}{3(15p-4)}}L_x^{\frac{52}{15(2-p)}}} \norm{\nabla \V_{N_2}\W_{N_3}}_{L_{t,x}^2}^{\frac{4}{15}} \\
&\qquad\qquad\qquad\cdot \norm{\nabla \V_{N_2}}_{L_t^{\frac{104}{3(15p-4)}}L_x^{\frac{52}{15(2-p)}}}^{\frac{11}{15}} \norm{\W_{N_3}}_{L_t^\I L_x^2}^{p-\frac4{15}} \\
\lsm & \de_0^2 \sum_{k\in\Z:2^k\loe T_0} \sum_{N_1\loe N_2,N_3\ll N_2} N_3^{0+} 2^{\frac{64-45p}{60}k} N_1^{1-s} (N_2^{\frac12-s} 2^{(\frac32p-1)k} N_3^{\frac{5p+4}{2(p+2)}} N_0^{\frac{2(p+1)}{p+2}(1-s)})^{\frac4{15}}   \\
& \qquad\qquad\qquad \cdot N_2^{\frac{11}{15}(1-s)}  \norm{\W_{N_3}}_{L_t^\I L_x^2}^{p-\frac4{15}} \\
\lsm & \de_0^2 \sum_{k\in\Z:2^k\loe T_0} \sum_{N_1\loe N_2,N_3\ll N_2}  2^{\frac{16-7p}{20}k} N_1^{1-s} N_2^{-\frac{2}{15}+1-s}   N_3^{\frac{2(5p+4)}{15(p+2)}+}\norm{\W_{N_3}}_{L_t^\I L_x^2}^{p-\frac4{15}} N_0^{\frac{8(p+1)}{15(p+2)}(1-s)},
}
where $t$ is taken over $I_k$. For any $\frac23<p<\frac43$, we can check that $\frac{2(5p+4)}{15(p+2)}<p-\frac{4}{15}$, then
\EQ{
N_3^{\frac{2(5p+4)}{15(p+2)}+}\norm{\W_{N_3}}_{L_t^\I L_x^2}^{p-\frac4{15}} \lsm N_3^{0-}    \norm{\nabla\W_{N_3}}_{L_t^\I L_x^2}^{p-\frac4{15}}.
}
Moreover, we also have that
\EQ{
\sum_{N_1\loe N_2}N_1^{1-s} N_2^{-\frac{2}{15}+1-s} \lsm \sum_{N_1}  N_1^{-\frac{2}{15}+2(1-s)} \lsm 1.
}
Consequently, combining the above two inequalities and \eqref{eq:fra-energy-bound-w-3d},
\EQ{
\eqref{eq:fra-energy-bound-vv-high-low-w-3d} \lsm & \de_0^2 \sum_{k\in\Z:2^k\loe T_0}2^{\frac{16-7p}{20}k} \sum_{N_1\loe N_2,N_3\ll N_2} N_1^{1-s} N_2^{-\frac{2}{15}+1-s} N_3^{0-}    \norm{\nabla\W_{N_3}}_{L_t^\I L_x^2}^{p-\frac4{15}} N_0^{\frac{8(p+1)}{15(p+2)}(1-s)} \\
\lsm & \de_0^2 \sum_{k\in\Z:2^k\loe T_0}2^{\frac{16-7p}{20}k}    \norm{\nabla\W}_{L_t^\I L_x^2}^{p-\frac4{15}} N_0^{\frac{8(p+1)}{15(p+2)}(1-s)} \\
\lsm & \de_0^2 \sum_{k\in\Z:2^k\loe T_0} 2^{\frac{45p^2-93p+64}{60}k}  N_0^{\frac{p(15p+34)}{15(p+2)}(1-s)}.
}
Note that for $2/3<p<4/3$,
\EQ{
\frac{45p^2-93p+64}{60}>0.26\text{, and }\frac{p(15p+34)}{15(p+2)}<1.45.
}
Then, we have
\EQn{\label{esti:fra-energy-bound-vv-high-low-w-3d}
\eqref{eq:fra-energy-bound-vv-high-low-w-3d} 
\lsm & \de_0^2 \sum_{k\in\Z:2^k\loe T_0} 2^{0.26\times k}  N_0^{1.45\times(1-s)} \\
\lsm & \de_0^2 N_0^{2(1-s)}.
}

Finally, Lemma \ref{lem:fra-energy-3d-other} follows from \eqref{esti:fra-energy-bound-wv-high-high-v-3d}, \eqref{esti:fra-energy-bound-wv-high-high-w-3d}, \eqref{esti:fra-energy-bound-wv-high-low-v-3d}, \eqref{esti:fra-energy-bound-vv-high-high-3d}, \eqref{esti:fra-energy-bound-vv-high-low-v-3d}, and  \eqref{esti:fra-energy-bound-vv-high-low-w-3d}.

\vskip 1.5cm
\section{Scattering in fractional weighted space:  general dimensions}\label{sec:fractional-dd}
\vskip .5cm

See Section \ref{sec:fra-initial} and \ref{sec:fra-local} for the preliminary settings. Specifically, we refer the readers to Section \ref{sec:fra-initial} for the linear-nonlinear decomposition and the estimates for the linear solution $\V$ (Lemma \ref{lem:fra-linear}). Moreover, please see Corollary \ref{cor:fra-local-u} for the local $H^s$-estimate of $\U$. 

Now, in Section \ref{sec:fractional-dd}, we will give the proof of the fractional weighted scattering in general dimensions case, namely Theorem \ref{thm:frac-weighted-dd}.  To this end, by the scattering criterion in Proposition \ref{prop:scattering-criterion}, it suffices to prove that
\EQ{
	\sup_{0<t\le T_0}\E(t)<C,
}
where $\E$ is defined in \eqref{defn:pseudo-conformal-energy}, namely
\EQ{
\E(t) = \frac{1}{4}t^{2-\frac{dp}{2}}\int_{\R^d}|\nabla\W(t,x)|^2 \dx + \frac{1}{p+2} \int_{\R^d}|\U(t,x)|^{p+2} \dx.
}
This will be finished in  Proposition \ref{prop:fra-energy-dd} below.

To begin with, we give the definition of $s_0=s_0(d,p)$. Throughout Section \ref{sec:fractional-dd}, we assume that $s_0<s<1$ .
\begin{defn}[Regularity condition]\label{defn:fra-scattering-s0}
\begin{enumerate}
\item (1D case)
When $d=1$ and $2<p<4$, then $s_0=\frac{18}{19}$.
\item ($p<\frac{2}{d-2}$ case)
When $d\goe 2$ and $\frac2d<p<\min\fbrk{\frac{2}{d-2},\frac{4}{d}}$, then
\EQ{
s_0:=1-\frac{1}{100d^2}\cdot\min\fbrk{dp-2,4-dp,2-(d-2)p}.
}
\item ($p\goe\frac{2}{d-2}$ case)
If it holds that
\EQ{
\frac{2}{d-2}\loe p<\min\fbrk{\frac4d,\frac{2}{d-4}}\text{, and }5\loe d\loe11,
}
then we take $s_0$ such that $1-s_0>0$ is a sufficiently small parameter that depends on $d$ and $p$, and $\th_1:=\frac12(1-s_0)$.
\end{enumerate}
\end{defn}
\begin{remark}
When $d\goe2$, it necessitates that $1-s_0$ should be sufficiently small to guarantee the various required summation in the proof. In the $p<\frac{2}{d-2}$ case, it is not difficult to write the exact definition of $s_0$, while in the $p\goe\frac{2}{d-2}$ case, the exact expression for $s_0$ is too complex, even though it is still possible to write it down. For example, in the proof of \eqref{eq:fra-energy-bound-wv-high-high-5d-main-1} below, we need to prove that 
\EQ{
-\frac{dp^2}{4}+\frac34(d+2)p-\frac52-1000\th_1>0
}
under the condition
\EQ{
\frac{2}{d-2}<p<\frac4d\text{, when }d=5,6;\quad \frac12<p<\frac47\text{, when }d=7.
}
We can check that under such condition, the polynomial $-\frac{dp^2}{4}+\frac34(d+2)p-\frac52>0$, then it suffices to take $\th_1\loe \frac{1}{2000}\brkb{-\frac{dp^2}{4}+\frac34(d+2)p-\frac52}$, and thus $s_0\goe1-\frac{1}{4000}\brkb{-\frac{dp^2}{4}+\frac34(d+2)p-\frac52}$.

\end{remark}

\subsection{Energy estimates near the infinity time}
In this section, we give an improved estimate for $w=u-v$, that is $Jw(t)\in L_x^2(\R^3)$ locally.
\begin{prop}\label{prop:fra-local-w-dd}
Suppose that $s_0$ is defined in Definition \ref{defn:fra-scattering-s0} and the assumptions in Theorem \ref{thm:frac-weighted-dd} hold. Then, there exists some constant $T_0=T_0(d,p,\norm{\jb{x}^su_0}_{L_x^2(\R^d)})>0$ such that
\EQ{
\sup_{T_0\loe t <+\I}\norm{\W(t)}_{H_x^1(\R^d)} \loe C N_0^{2(1-s)}.
}
\end{prop}
Similar to the proof of 3D case in Proposition \ref{prop:fra-local-w-3d}, it is reduced to the following lemma.
\begin{lem}\label{lem:fra-local-nonlinear-dd}
Suppose that $s_0$ is defined in Definition \ref{defn:fra-scattering-s0} and the assumptions in Theorem \ref{thm:frac-weighted-dd} hold. Fix any $T_0>T$, where $T$ is defined in Corollary \ref{cor:fra-local-u}. Then, there exists some constant $\th=\th(d,p)>0$ such that
\EQ{
\normb{\int_t^{+\I} S(t-\ta)(\ta^{\frac{dp}{2}-2}|\U|^p\U) \dd \ta}_{L_t^\I \dot H_x^1([T_0,+\I)\times\R^d)} \lsm T_0^{-\th} (1+\norm{\W}_{L_t^\I \dot H_x^1([T_0,+\I)\times\R^d)}) .
}
\end{lem}
\begin{proof}[Proof of Lemma \ref{lem:fra-local-nonlinear-dd}]
In this lemma, we will restrict the variable $(t,x)$ on $[T_0,+\I)\times\R^d$. Now, we split the proof into two subcases: $d=1$ and $d\goe2$.

$\bullet$ \textbf{Proof of 1D case:} By the Strichartz estimate,
\EQ{
	\normb{\int_t^{+\I} S(t-\ta)(\ta^{\frac{p}{2}-2}|\U|^p\U) \dd \ta}_{L_t^\I\dot H_x^1} \lsm & \normb{t^{\frac{p}{2}-2}|\U|^p\nabla \U }_{L_t^{\frac43}L_x^1 +L_t^1 L_x^2}.
}
Then, we decompose 
\EQ{
	|\U|^p\nabla \U = & |\U|^p\nabla\W + \sum_{N\in2^\N} |\U|^{p-2}O(\U^2\nabla \V_N) \\
	= & |\U|^p\nabla\W +  \sum_{N\in2^\N} |\U|^{p-2}O(\U_{\gsm N}(\U_{\gsm N} + \U_{\ll N})\nabla \V_N)\\
	& + \sum_{N\in2^\N} |\U|^{p-2}O(\U_{\ll N}^2\nabla \V_N).
}
It suffices to consider
\begin{subequations}
	\EQnn{
		\normb{t^{\frac{p}{2}-2}|\U|^p\nabla \U }_{L_t^{\frac43}L_x^1} \lsm & \normb{t^{\frac{p}{2}-2}|\U|^p\nabla\W }_{L_t^{\frac43}L_x^1} \label{esti:frac-local-nonlinear-term-1}\\
		& + \normb{t^{\frac{p}{2}-2}\sum_{N\in2^\N} |\U|^{p-2}O(\U_{\gsm N}(\U_{\gsm N} + \U_{\ll N})\nabla \V_N) }_{L_t^{\frac43}L_x^1}\label{esti:frac-local-nonlinear-term-2}\\
		& + \normb{t^{\frac{p}{2}-2}\sum_{N\in2^\N} |\U|^{p-2}O(\U_{\ll N}^2\nabla \V_N) }_{L_t^{1}L_x^2}.\label{esti:frac-local-nonlinear-term-3}
	}
\end{subequations}
We first estimate \eqref{esti:frac-local-nonlinear-term-1}. By H\"older's inequality,  Corollary \ref{cor:fra-local-u}, and noting that $(\frac{4p}{p-1},2p)$ is $L_x^2$-admissible,
\EQn{\label{esti:frac-local-nonlinear-1}
\eqref{esti:frac-local-nonlinear-term-1} \lsm & \normb{t^{\frac{1}{2}-\frac2p}\U}_{L_t^{\frac43p}L_x^{2p}}^p \norm{\nabla\W}_{L_t^\I L_x^2} \\
\lsm & T_0^{-\frac{4-p}{4}}\norm{\U}_{L_t^{\frac{4p}{p-1}}L_x^{2p}}^p \norm{\nabla\W}_{L_t^\I L_x^2} \\
\lsm & T_0^{-\frac{4-p}{4}} \norm{\W}_{L_t^\I \dot H_x^1}.
}
For \eqref{esti:frac-local-nonlinear-term-2}, by H\"older's inequality, Corollary \ref{cor:fra-local-u}, and Lemma \ref{lem:fra-linear},
\EQn{\label{esti:frac-local-nonlinear-2}
& \normb{t^{\frac{p}{2}-2}\sum_{N\in2^\N} |\U|^{p-2}O(\U_{\gsm N}(\U_{\gsm N} + \U_{\ll N})\nabla \V_N)}_{L_t^{\frac43}L_x^1} \\
\lsm & \sum_{N\in2^\N} \normb{t^{\frac{p}{2}-2}|\U|^{p-2}(\U_{\gsm N} + \U_{\ll N})}_{L_{t}^{\frac43} L_x^\I} \normb{\U_{\gsm N}}_{L_t^\I L_x^2} \normb{\nabla\V_N}_{L_t^\I L_x^2} \\
\lsm & \sum_{N\in2^\N} \normb{t^{\frac{p}{2}-2}|\U|^{p-2}(\U_{\gsm N} + \U_{\ll N})}_{L_{t}^{\frac43} L_x^\I} \normb{N^{\frac12}\U_{\gsm N}}_{L_t^\I L_x^2} \normb{N^{\frac12}\V_N}_{L_t^\I L_x^2} \\
\lsm & \normb{t^{\frac{p-4}{2(p-1)}}\U}_{L_t^{\frac43(p-1)}L_x^\I}^{p-1} \norm{\U}_{L_t^\I H_x^s}\norm{\V}_{L_t^\I H_x^s} \\
\lsm & T_0^{-\frac{4-p}{4}} \norm{\U}_{L_t^{4}L_x^\I}^{p-1} \norm{\U}_{L_t^\I H_x^s}\norm{\V}_{L_t^\I H_x^s} \\
\lsm & T_0^{-\frac{4-p}{4}}.
}
Finally, we consider \eqref{esti:frac-local-nonlinear-term-3}. To this end, we first use bilinear Strichartz estimate, the fractional chain rule in Lemma \ref{lem:frac-chain}, Lemma \ref{lem:fra-initial-data}, and Corollary \ref{cor:fra-local-u}, for any $N_1\ll N$,
\EQn{\label{esti:frac-local-nonlinear-3-bi}
	\norm{\U_{N_1}\nabla \V_N}_{L_{t,x}^2} \lsm & N_1^{-s} \normb{N^{\frac12}\V_{+,N}}_{L_x^2} \brkb{\normb{N_1^{s}\U_{+,N_1}}_{L_x^2} + \normb{t^{\frac p2-2}N_1^{s}P_{N_1}(|\U|^p\U)}_{L_t^{\frac{4}{3}}L_x^1}} \\
	\lsm & N^{\frac12-s} N_1^{-s} \norm{\V_+}_{H_x^s}\brkb{\norm{\U_+}_{H_x^s} + \normb{t^{\frac p2-2}|\nabla|^{s}(|\U|^p\U)}_{L_t^{\frac43}L_x^1}} \\
	\lsm & N^{\frac12-s} N_1^{-s} \norm{\V_+}_{H_x^s}\brkb{\norm{\U_+}_{H_x^s} + \normb{t^{\frac 12-\frac 2p}\U}_{L_t^{\frac43p}L_x^{2p}}^p\norm{\U}_{L_t^\I \dot H_x^s}} \\
	\lsm & N^{\frac12-s} N_1^{-s} \norm{\V_+}_{H_x^s}\brkb{\norm{\U_+}_{H_x^s} + T_0^{-\frac{4-p}{4}} \norm{\U}_{L_t^{\frac{4p}{p-1}}L_x^{2p}}^p\norm{\U}_{L_t^\I \dot H_x^s}} \\
	\lsm & N^{\frac12-s} N_1^{-s}.
}
Then, by H\"older's, Bernstein's inequalities, \eqref{esti:frac-local-nonlinear-3-bi}, and Corollary \ref{cor:fra-local-u},
\EQn{\label{esti:frac-local-nonlinear-3-1}
\eqref{esti:frac-local-nonlinear-term-3}
\lsm & \sum_{N_2\loe N_1\ll N} \normb{t^{\frac{p}{2}-2}|\U|^{p-2}O(\U_{N_2}\U_{N_1}\nabla \V_N)}_{L_t^{1}L_x^2} \\
\lsm & \sum_{N_2\loe N_1\ll N} \normb{t^{\frac p2-2}|\U|^{p-2}}_{L_t^2 L_x^\I} \norm{\U_{N_2}}_{L_t^\I L_x^\I} \norm{\U_{N_1}\nabla \V_N}_{L_{t,x}^2} \\
\lsm & \sum_{N_2\loe N_1\ll N} \normb{t^{\frac{p-4}{2(p-2)}}\U}_{L_t^{2(p-2)} L_x^\I}^{p-2}  N_2^{\frac12-s}\norm{N_2^s\U_{N_2}}_{L_t^\I L_x^2}  N^{\frac12-s} N_1^{-s} \\
\lsm & \normb{t^{\frac{p-4}{2(p-2)}}\U}_{L_t^{2(p-2)}L_x^\I}^{p-2} \sum_{N_2\loe N_1\ll N} N_2^{\frac12-s} N^{\frac12-s} N_1^{-s}\\
\lsm & \normb{t^{\frac{p-4}{2(p-2)}}\U}_{L_t^{2(p-2)}L_x^\I}^{p-2}.
}
By the change of variable (see Remark \ref{rem:spacetime-exponent-transform}) and H\"older's inequality, we have
\EQn{\label{esti:frac-local-nonlinear-3-2}
	\normb{t^{\frac{p-4}{2(p-2)}}\U}_{L_t^{2(p-2)}L_x^\I([T_0,+\I)\times\R)}^{p-2} = & \norm{u}_{L_t^{2(p-2)}L_x^\I([0,1/T_0]\times\R)}^{p-2} \\
	\lsm & T_0^{-\frac{4-p}{4}} \norm{u}_{L_t^4L_x^\I([0,1/T_0]\times\R)}^{p-2} \\
	\lsm & T_0^{-\frac{4-p}{4}}.
}
Therefore, by \eqref{esti:frac-local-nonlinear-3-1} and \eqref{esti:frac-local-nonlinear-3-2},
\EQn{\label{esti:frac-local-nonlinear-3}
	\eqref{esti:frac-local-nonlinear-term-3}
	\lsm T_0^{-\frac{4-p}{4}}.
}
The lemma follows by \eqref{esti:frac-local-nonlinear-1}, \eqref{esti:frac-local-nonlinear-2}, and \eqref{esti:frac-local-nonlinear-3}, with $\th= \frac{4-p}{4}$.

$\bullet$ \textbf{Proof of 2D and higher dimensional cases:} by the Strichartz estimate,
\EQ{
\normb{\int_t^{+\I} S(t-\ta)(\ta^{\frac{dp}{2}-2}|\U|^p\U) \dd \ta}_{L_t^\I \dot H_x^1} \lsm & \normb{t^{\frac{dp}{2}-2}|\U|^p\nabla \U }_{L_t^{2}L_x^{\frac{2d}{d+2}}}.
}
Then, we decompose 
\EQ{
|\U|^p|\nabla \U| \lsm & |\U|^p|\nabla\W| + \sum_{N\in2^\N} |\U_{\gsm N} + \U_{\ll N}|^{p}|\nabla \V_N| \\
\lsm & |\U|^p|\nabla\W| + \sum_{N\in2^\N} |\U_{\gsm N}|^{p}|\nabla \V_N|
 + \sum_{N\in2^\N} | \U_{\ll N}|^{p}|\nabla \V_N|.
}
It suffices to consider
\EQnnsub{
\normb{t^{\frac{dp}{2}-2}|\U|^p\nabla \U }_{L_t^{2}L_x^{\frac{2d}{d+2}}} \lsm & \normb{t^{\frac{dp}{2}-2}|\U|^p|\nabla\W| }_{L_t^{2}L_x^{\frac{2d}{d+2}}} \label{esti:frac-local-nonlinear-term-dd-1}\\
& + \normb{t^{\frac{dp}{2}-2}\sum_{N\in2^\N} |\U_{\gsm N}|^{p}|\nabla \V_N| }_{L_t^{2}L_x^{\frac{2d}{d+2}}}\label{esti:frac-local-nonlinear-term-dd-2}\\
& + \normb{t^{\frac{dp}{2}-2}\sum_{N\in2^\N} | \U_{\ll N}|^{p}|\nabla \V_N| }_{L_t^2L_x^{\frac{2d}{d+2}}}.\label{esti:frac-local-nonlinear-term-dd-3}
}

For \eqref{esti:frac-local-nonlinear-term-dd-1}, by H\"older's inequality,
\EQ{
\eqref{esti:frac-local-nonlinear-term-dd-1} 
\lsm &  \normb{t^{\frac{dp}{2}-2}}_{L_t^{\frac{4}{4-dp}}} \norm{\U}_{L_t^{\I}L_x^{2}}^p \norm{\nabla\W}_{L_t^{\frac{4}{dp-2}} L_x^{\frac{2d}{d+2-dp}}}.
}
Note that
\EQ{
\normb{t^{\frac{dp}{2}-2}}_{L_t^{\frac{4}{4-dp}}} \lsm  \brkb{\int_{T_0}^{+\I} t^{-2} \dd t}^{\frac{4-dp}{4}} \lsm T_0^{-\frac{4-dp}{4}},
}
and that $(\frac{4}{dp-2},\frac{2d}{d+2-dp})$ is $L^2$ admissible for $\frac2d<p<\frac4d$, then combining Corollary \ref{cor:fra-local-u},
\EQn{\label{eq:frac-local-nonlinear-term-dd-1}
\eqref{esti:frac-local-nonlinear-term-dd-1} 
\lsm T_0^{-\frac{4-dp}{4}} \norm{\W}_{L_t^\I \dot H_x^1}.
}

For \eqref{esti:frac-local-nonlinear-term-dd-2}, by $s>\frac{1}{1+p}$, H\"older's inequality, and Corollary \ref{cor:fra-local-u},
\EQn{\label{eq:frac-local-nonlinear-term-dd-2}
\eqref{esti:frac-local-nonlinear-term-dd-2} 
\lsm & \sum_{N\in2^\N} \normb{t^{\frac{dp}{2}-2}}_{L_t^{\frac{4}{4-dp}}} \norm{\U_{\gsm N}}_{L_t^{\I}L_x^2}^p \norm{\nabla\V_N}_{L_t^{\frac{4}{dp-2}} L_x^{\frac{2d}{d+2-dp}}} \\
\lsm & T_0^{-\frac{4-dp}{4}} \sum_{N\in2^\N} N^{1-(1+p)s} \normb{N^s\U_{\gsm N}}_{L_t^{\I}L_x^2}^p \norm{\jb{\nabla}^s\V_N}_{L_t^{\frac{4}{dp-2}} L_x^{\frac{2d}{d+2-dp}}} \\
\lsm & T_0^{-\frac{4-dp}{4}}  \normb{\jb{\nabla}^{s}\U}_{L_t^{\I}L_x^{2}}^p \norm{\jb{\nabla}^s\V}_{L_t^{\frac{4}{dp-2}} L_x^{\frac{2d}{d+2-dp}}} \\
\lsm & T_0^{-\frac{4-dp}{4}}.
}

Finally, we estimate \eqref{esti:frac-local-nonlinear-term-dd-3}. We first take
\EQ{
\al:=\frac{4-dp}{10d}.
}
Note that $(\frac{4(1-\al)}{dp-2-d\al},\frac{2d(1-\al)}{d+2-dp})$ is $L^2$ admissible. By \eqref{eq:bound-for-up}, H\"older's inequality,  Lemma \ref{lem:fra-linear}, and Corollary \ref{cor:fra-local-u},
\EQn{\label{eq:frac-local-nonlinear-term-dd-3-1}
\eqref{esti:frac-local-nonlinear-term-dd-3} 
\lsm & \sum_{N_1\ll N} N_1^{0+}\norm{\U_{N_1}\nabla \V_N}^{\al}_{L_{t,x}^{2}}\normb{t^{\frac{dp}{2}-2}}_{L_t^{\frac{4}{4-dp+(d-2)\al}}} \norm{\U_{N_1}}_{L_t^\I L_x^2}^{p-\al} \\
& \quad\quad\quad\cdot  \normb{\nabla \V_N}_{L_t^{\frac{4(1-\al)}{dp-2-d\al}}L_x^{\frac{2d(1-\al)}{d+2-dp}}}^{1-\al} \\
\lsm & T_0^{-\frac{4-dp-(d-2)\al}{4}} \sum_{N_1\ll N} N_1^{0+} N^{(1-\al)(1-s)}\norm{\U_{N_1}\nabla \V_N}^{\al}_{L_{t,x}^{2}} \norm{\U_{N_1}}_{L_t^\I L_x^2}^{p-\al}.
}
By the bilinear Strichartz estimate in Lemma \ref{lem:bilinearstrichartz}, fractional chain rule in Lemma \ref{lem:frac-chain}, Lemma \ref{lem:fra-initial-data}, and Corollary \ref{cor:fra-local-w-dd},
\EQn{\label{eq:frac-local-nonlinear-term-dd-3-bi}
\norm{\U_{N_1}\nabla \V_N}_{L_{t,x}^{2}} \lsm & \frac{N_1^{\frac{d-1}{2}}}{N^{1/2}} \brkb{\norm{P_{N_1}\U_+}_{L_x^2} + \normb{t^{\frac{dp}{2}-2}P_{N_1}\brkb{|\U|^p\U}}_{L_t^{\frac{2}{1+\ep}}L_x^{\frac{2d}{d+2-2\ep}}}}\norm{P_N\V_+}_{L_x^2} \\
\lsm & N_1^{\frac{d-1}{2}-s} N^{\frac12-s} \brkb{\norm{\jb{\nabla}^sP_{N_1}\U_+}_{L_x^2} + \normb{t^{\frac{dp}{2}-2}\jb{\nabla}^s\brkb{|\U|^p\U}}_{L_t^{\frac{2}{1+\ep}}L_x^{\frac{2d}{d+2-2\ep}}}}\\
& \cdot\norm{\jb{\nabla}^sP_N\V_+}_{L_x^2} \\
\lsm & N_1^{\frac{d-1}{2}-s} N^{\frac12-s} \brkb{1 + \normb{t^{\frac{dp}{2}-2}}_{L_t^{\frac{4}{4-dp}}} \norm{\U}_{L_t^{\frac{2p}{\ep}} L_x^{\frac{2dp}{dp-2\ep}}}^p \normb{\jb{\nabla}^s\U}_{L_t^{\frac{4}{dp-2}}L_x^{\frac{2d}{d+2-dp}}}} \\
\lsm & N_1^{\frac{d-1}{2}-s} N^{\frac12-s} \brkb{1 + T_0^{-\frac{4-dp}{4}}  \normb{\jb{\nabla}^s\U}_{S^0}^{p+1}} \\
\lsm & N_1^{\frac{d-1}{2}-s} N^{\frac12-s}.
}
Then, by \eqref{eq:frac-local-nonlinear-term-dd-3-1},  \eqref{eq:frac-local-nonlinear-term-dd-3-bi}, $4-dp-(d-2)\al > \frac12(4-dp)$, $1-s<1-s_0<\frac{1}{20d}(4-dp)$, and Corollary \ref{cor:fra-local-w-dd},
\EQn{\label{eq:frac-local-nonlinear-term-dd-3}
\eqref{esti:frac-local-nonlinear-term-dd-3} 
\lsm & T_0^{-\frac{4-dp}{8}} \sum_{N_1\ll N} N_1^{0+} N^{(1-\al)(1-s)} \brkb{N_1^{\frac{d-1}{2}-s} N^{\frac12-s}}^\al  \norm{\U_{N_1}}_{L_t^\I L_x^2}^{p-\al} \\
\lsm & T_0^{-\frac{4-dp}{8}} \sum_{N_1\ll N} N^{-\frac12\al +(1-s)} \normb{\jb{\nabla}^{\frac{\al}{p-\al}(\frac{d-1}{2}-s)+}\U_{N_1}}_{L_t^\I L_x^2}^{p-\al} \\
\lsm & T_0^{-\frac{4-dp}{8}} \sum_{N\in2^\N} N^{-\frac12\al +(1-s)} \normb{\jb{\nabla}^{s}\U_{N_1}}_{L_t^\I L_x^2}^{p-\al} \\
\lsm & T_0^{-\frac{4-dp}{8}}.
}
By \eqref{eq:frac-local-nonlinear-term-dd-1}, \eqref{eq:frac-local-nonlinear-term-dd-2}, and \eqref{eq:frac-local-nonlinear-term-dd-3}, we have
\EQ{
\normb{t^{\frac{dp}{2}-2}|\U|^p\nabla \U }_{L_t^{2}L_x^{\frac{2d}{d+2}}} \lsm T_0^{-\th}\brkb{1+\norm{\W}_{L_t^\I \dot H_x^1}}.
}
This finishes the proof.
\end{proof}
Using the same method of the proof for Corollary \ref{cor:fra-local-w-3d}, we derive that:
\begin{cor}\label{cor:fra-local-w-dd}
Suppose that $s_0$ is defined in Definition \ref{defn:fra-scattering-s0} and the assumptions in Theorem \ref{thm:frac-weighted-dd} hold. Then, there exists a constant $A=A(d,p,\norm{\jb{x}^su_0}_{L_x^2(\R^d)},T_0)>0$ such that
\EQ{
\E(T_0) \loe AN_0^{2(1-s)}.
}
\end{cor}

\subsection{Energy estimate towards the origin}
\begin{prop}\label{prop:fra-energy-dd}
Suppose that $s_0$ is defined in Definition \ref{defn:fra-scattering-s0} and the assumptions in Theorem \ref{thm:frac-weighted-dd} hold. Let $A$ be the constant in Corollary \ref{cor:fra-local-w-dd}. Then, we have
	\EQ{
		\sup_{t\in(0,T_0]} \E(t) \loe 2A N_0^{2(1-s)}. 
	}
\end{prop}
To this end, we establish the bootstrap framework. Denote that $I:=(0,T_0]$. Then, it suffices to prove that if we assume the bootstrap hypothesis
\EQn{\label{eq:fra-energy-bootstrap-hypothesis-dd}
	\sup_{t\in I} \E(t) \loe 2A N_0^{2(1-s)},
}
then
\EQn{\label{eq:fra-energy-bootstrap-conclusion-dd}
	\sup_{t\in I} \E(t) \loe \frac32A N_0^{2(1-s)}.
}
Now, we reduce the proof of \eqref{eq:fra-energy-bootstrap-conclusion-dd} to the following two lemmas:
\begin{lem}[$p<\frac{2}{d-2}$ case]\label{lem:fra-energy-dd}
Suppose that $s_0$ is defined in Definition \ref{defn:fra-scattering-s0} ,and the assumptions in Theorem \ref{thm:frac-weighted-dd} and the bootstrap hypothesis \eqref{eq:fra-energy-bootstrap-hypothesis-dd} hold. Suppose also that
\EQn{\label{eq:fra-energy-dd-assumption-dp}
d\goe 1\text{, and }\frac2d<p<\min\fbrk{\frac4d,\frac{2}{d-2}}.
}
Then, we have
\EQ{
\absb{\int_0^{T_0}\int_{\R}|\U|^p\U\wb{\V_t}\dd x\dd t} \loe C(A)\de_0  N_0^{2(1-s)}. 
}
\end{lem}

\begin{lem}[$p\goe\frac{2}{d-2}$ case]\label{lem:fra-energy-5d}
Suppose that $s_0$ is defined in Definition \ref{defn:fra-scattering-s0} ,and the assumptions in Theorem \ref{thm:frac-weighted-dd} and the bootstrap hypothesis \eqref{eq:fra-energy-bootstrap-hypothesis-dd} hold. Suppose also that
\EQn{\label{eq:fra-energy-5d-assumption-dp}
5\loe d\loe 11\text{, and }\frac{2}{d-2}\loe p<\min\fbrk{\frac{4}{d},\frac{2}{d-4}}.
}
Then, we have
\EQ{
\absb{\int_0^{T_0}\int_{\R}|\U|^p\U\wb{\V_t}\dd x\dd t} \loe C(A)\de_0  N_0^{2(1-s)}.  
}
\end{lem}
Now, we give the proof of Proposition \ref{prop:fra-energy-dd} assuming that Lemmas \ref{lem:fra-energy-dd} and \ref{lem:fra-energy-5d} hold:
\begin{proof}[Proof of Proposition \ref{prop:fra-energy-dd}]
Note that \eqref{eq:fra-energy-dd-assumption-dp} and \eqref{eq:fra-energy-5d-assumption-dp} together imply the assumption on $(d,p)$ in Theorem \ref{thm:frac-weighted-dd}, namely
\begin{enumerate}
\item $\frac2d< p < \frac4d$, when $1\loe d\loe 8$;
\item $\frac2d< p < \frac{2}{d-4}$, when $9\loe d\loe 11$;
\item $\frac2d< p < \frac{2}{d-2}$, when $d\goe 12$.
\end{enumerate}
Recall the proof in 3D case,
\EQ{
\sup_{t\in(0,T_0]}\E(t)\loe \E(T_0) + |\int_{0}^{T_0}\int_{\R^d}|\U|^p\U\wb{\V_t}\dd x\dd t|.
}
Then, noting that $\de_0=\de_0(A)$ is sufficiently small, by Corollary \ref{cor:fra-local-w-dd}, Lemmas \ref{lem:fra-energy-dd} and \ref{lem:fra-energy-5d}, 
\EQ{
\sup_{t\in(0,T_0]}\E(t)\loe & A N_0^{2(1-s)} + C(A)\de_0 N_0^{2(1-s)} \\
\loe &  A N_0^{2(1-s)} + \frac{1}{2} A N_0^{2(1-s)} \\
= & \frac{3}{2} A N_0^{2(1-s)}.
}
Now, we have proven \eqref{eq:fra-energy-bootstrap-conclusion-dd}, thereby completing the proof of Proposition \ref{prop:fra-energy-dd}.
\end{proof}

Next, we start to prove Lemmas \ref{lem:fra-energy-dd} and \ref{lem:fra-energy-5d}. In the proofs, we regard $A$ and $T_0$ as constants, and omit their dependence, namely we write $C$ for $C(A,T_0)$. 
\subsection{Useful estimates of the nonlinear part}
We first gather some useful bounds for $\W$ for general dimensions under the bootstrap hypothesis \eqref{eq:fra-energy-bootstrap-hypothesis-dd}. This lemma will be applied to the proof of both Lemmas \ref{lem:fra-energy-dd} and \ref{lem:fra-energy-5d}.
\begin{lem}
Suppose that $s_0$ is defined in Definition \ref{defn:fra-scattering-s0} ,and the assumptions in Theorem \ref{thm:frac-weighted-dd} and the bootstrap hypothesis \eqref{eq:fra-energy-bootstrap-hypothesis-dd} hold. Recall that $I:=(0,T_0]$. Then, the following estimates hold.
\begin{enumerate}
\item (Energy bound) For any $t\in I$,
\EQn{\label{eq:fra-energy-bound-w} 
&\norm{\nabla \W(t)}_{L_x^2(\R^d)} \lsm t^{\frac{dp}{4}-1} N_0^{1-s}\text{, and }\norm{\W(t)}_{L_x^{p+2}(\R^d)} \lsm N_0^{\frac{2}{p+2}(1-s)} .
}
\item (Mass bound) For any $t\in I$,
\EQn{\label{eq:fra-mass-bound-w}
\norm{\W(t)}_{L_x^2(\R^d)} \lsm 1. 
}
\item (Bilinear Strichartz estimate) For any $t\in I$ and $N_1,N\in2^\N$ such that $N_1\ll N$,
\EQn{\label{eq:fra-energy-bound-bilinear}
\norm{\nabla \V_{N} \W_{N_1}}_{L_{t,x}^2(I\times\R^d)} + \norm{\nabla \V_{N} \U_{N_1}}_{L_{t,x}^2(I\times\R^d)} \lsm \de_0 N^{\frac12-s}  N_1^{\frac{2dp-p+2d-2}{2(p+2)}} N_0^{\frac{2(p+1)}{p+2}(1-s)}.
}
\end{enumerate}
\end{lem}
\begin{proof}
\eqref{eq:fra-energy-bound-w} follows directly from \eqref{eq:fra-energy-bootstrap-hypothesis-dd}.
	
	By the conservation of mass, \eqref{eq:fra-mass-bound-w} holds.

Finally, we prove \eqref{eq:fra-energy-bound-bilinear}. We first note that for $N_1\in2^\N$, by Sobolev's inequality and \eqref{eq:fra-energy-bound-w}, 
\EQn{\label{esti:fra-energy-bound-bilinear-global-l2}
\normb{t^{\frac{dp}{2} -2}P_{N_1}(|\U|^p\U)}_{L_t^1 L_x^2} \lsm & N_1^{\frac{dp}{2(p+2)}}\normb{t^{\frac {dp}{2} -2}P_{N_1}(|\U|^p\U)}_{L_t^1 L_x^{\frac{p+2}{p+1}}} \\
\lsm &  N_1^{\frac{dp}{2(p+2)}}\norm{|\U|^{p+1}}_{L_t^\I L_x^{\frac{p+2}{p+1}}} \\
\lsm & N_1^{\frac{dp}{2(p+2)}}\brkb{\norm{\V}_{L_t^\I L_x^{p+2}} + \norm{\W}_{L_t^\I L_x^{p+2}}}^{p+1} \\
\lsm & N_1^{\frac{dp}{2(p+2)}} N_0^{\frac{2(p+1)}{p+2}(1-s)}.
}
Then, by bilinear Strichartz estimate in Lemma \ref{lem:bilinearstrichartz} and \eqref{esti:fra-energy-bound-bilinear-global-l2}, for $N_1\ll N$,
\EQ{
\norm{\nabla \V_{N} \U_{N_1}}_{L_{t,x}^2} \lsm &  N^{-\frac12}N_1^{\frac{d-1}{2}}\norm{N \V_{N}(T)}_{L_x^2}\brkb{\norm{P_{N_1}\U(T)}_{L_x^2} + \normb{t^{\frac p2 -2}P_{N_1}(|\U|^p\U)}_{L_t^1 L_x^2}} \\
\lsm & N^{-\frac12}N_1^{\frac{d-1}{2}} \norm{N \V_{N}(T)}_{L_x^2} N_1^{\frac{dp}{2(p+2)}} N_0^{\frac{2(p+1)}{p+2}(1-s)} \\
\lsm & N^{\frac12-s}N_1^{\frac{d-1}{2}} \norm{N^s \V_{N}(T)}_{L_x^2} N_1^{\frac{dp}{2(p+2)}} N_0^{\frac{2(p+1)}{p+2}(1-s)} \\
\lsm & \de_0N^{\frac12-s}  N_1^{\frac{2dp-p+2d-2}{2(p+2)}} N_0^{\frac{2(p+1)}{p+2}(1-s)}.
}
Similarly, we also have 
\EQ{
\norm{\nabla \V_{N} \W_{N_1}}_{L_{t,x}^2} \lsm &  N^{-\frac12}N_1^{\frac{d-1}{2}}\norm{N \V_{N}(T)}_{L_x^2}\brkb{\norm{P_{N_1}\W(T)}_{L_x^2} + \normb{t^{\frac p2 -2}P_{N_1}(|\U|^p\U)}_{L_t^1 L_x^2}} \\
\lsm & \de_0N^{\frac12-s}  N_1^{\frac{2dp-p+2d-2}{2(p+2)}} N_0^{\frac{2(p+1)}{p+2}(1-s)}.
}
This completes the proof of \eqref{eq:fra-energy-bound-bilinear}.
\end{proof}

\subsection{Energy estimate when $p<\frac{2}{d-2}$ and $d=1$}\label{sec:fra-energy-dd-1}
\begin{proof}[Proof of Lemma \ref{lem:fra-energy-dd} in 1D case]
We make the following decomposition:
\begin{subequations}
	\EQnn{
		|\int_0^{T_0}\int_{\R}|\U|^p\U\wb{\V_t}\dd x\dd t| \lsm & \sum_{N\lsm N_1} \int_0^{T_0}\int_\R |\nabla \W| |\nabla \V_N||\V_{ N_1}||\U|^{p-1} \dd x \dd t \label{eq:fra-energy-bound-wv-high-high-v}\\
		& + \sum_{N\lsm N_1} \int_0^{T_0}\int_\R |\nabla \W| |\nabla \V_N||\W_{ N_1}||\U|^{p-1} \dd x \dd t \label{eq:fra-energy-bound-wv-high-high-w}\\
		& + \sum_{N_1\ll N} \int_0^{T_0}\int_\R |\nabla \W| |\nabla \V_N||\U_{ N_1}||\U|^{p-1} \dd x \dd t \label{eq:fra-energy-bound-wv-high-low}\\
		& + \sum_{N_1\loe N_2}\int_0^{T_0} \int_\R |\nabla\V_{N_1}||\nabla\V_{N_2}||\U_{\gsm N_2}||\U|^{p-1}\dd x\dd t \label{eq:fra-energy-bound-vv-high-high}\\
		& + \sum_{N_1\loe N_2}\int_0^{T_0} \int_\R |\nabla\V_{N_1}||\nabla\V_{N_2}||\U_{\ll N_2}||\U|^{p-1}\dd x\dd t. \label{eq:fra-energy-bound-vv-high-low}
	}
\end{subequations}
Next, we estimate the terms \eqref{eq:fra-energy-bound-wv-high-high-v}-\eqref{eq:fra-energy-bound-vv-high-low} one by one. In the proof of the 1D case, the spacetime norms are taken over $(t,x) \in (0,T_0]\times\R$.

$\bullet$ \textbf{Estimate of \eqref{eq:fra-energy-bound-wv-high-high-v}.} For this term, we can directly transfer the first-order derivative to $\V_{N_1}$. First, by H\"older's inequality and \eqref{eq:fra-energy-bound-w},
\EQ{
	\eqref{eq:fra-energy-bound-wv-high-high-v}\lsm & \sum_{N\lsm N_1} \int_0^{T_0} \norm{\nabla \W}_{L_x^2} \norm{N \V_N}_{L_x^2}\norm{\V_{N_1}}_{L_x^\I}\brkb{\norm{\V}_{L_x^\I}^{p-1} + \norm{\W}_{L_x^\I}^{p-1}}  \dd t \\
	\lsm & \sum_{N\lsm N_1} \int_0^{T_0} t^{\frac p4-1} N_0^{1-s} \frac{N^{1-s}}{N_1^{1-s}}\norm{\V_N}_{H_x^s}\norm{N_1^{1-s}\V_{N_1}}_{L_x^\I}\brkb{\norm{\V}_{L_x^\I}^{p-1} + \norm{\W}_{L_x^\I}^{p-1}}  \dd t \\
	\lsm & \de_0 N_0^{1-s} \sum_{N_1} \int_0^{T_0} t^{\frac p4-1} \norm{N_1^{1-s}\V_{N_1}}_{L_x^\I}\brkb{\norm{\V}_{L_x^\I}^{p-1} + \norm{\W}_{L_x^\I}^{p-1}}  \dd t.
}
Then, we consider the estimate of
\EQnnsub{
& \sum_{N_1} \int_0^{T_0} t^{\frac p4-1} \norm{N_1^{1-s}\V_{N_1}}_{L_x^\I}\brkb{\norm{\V}_{L_x^\I}^{p-1} + \norm{\W}_{L_x^\I}^{p-1}}  \dd t \nonumber\\
\lsm &\sum_{N_1} \int_0^{T_0} t^{\frac p4-1} \norm{N_1^{1-s}\V_{N_1}}_{L_x^\I}\norm{\V}_{L_x^\I}^{p-1} \dd t \label{eq:fra-energy-bound-wv-high-high-v-1}\\
+& \sum_{N_1} \int_0^{T_0} t^{\frac p4-1} \norm{N_1^{1-s}\V_{N_1}}_{L_x^\I} \norm{\W}_{L_x^\I}^{p-1} \dd t.\label{eq:fra-energy-bound-wv-high-high-v-2}
}
For \eqref{eq:fra-energy-bound-wv-high-high-v-1}, noting that
\EQ{
	\normb{t^{\frac p4-1}}_{L_t^{\frac43}}^{\frac43}=\int_0^{T_0} t^{\frac43(\frac p4-1)} \dd t \lsm \int_0^{T_0} t^{-\frac23} \dd t \lsm 1,
}
then by Sobolev's inequality, H\"older's inequality in $t$, and Lemma \ref{lem:fra-linear},
\EQn{\label{esti:fra-energy-bound-wv-high-high-v-1}
	\eqref{eq:fra-energy-bound-wv-high-high-v-1}
	\lsm & \int_0^{T_0} t^{\frac p4-1} \norm{\jb{\nabla}^s\V}_{L_x^\I}\norm{\V}_{L_x^\I}^{p-1} \dd t \\
	\lsm & \normb{t^{\frac p4-1}\jb{\nabla}^s\V}_{L_t^1 L_x^\I} \norm{\V}_{L_{t,x}^\I}^{p-1} \\
	\lsm & \normb{\jb{\nabla}^s\V}_{L_t^4 L_x^\I} \norm{\V}_{L_{t}^\I H_x^s}^{p-1} \\
	\lsm & \de_0^p.
}
Next, we consider \eqref{eq:fra-energy-bound-wv-high-high-v-2}. We first claim an $L^\I$-estimate for $\W$: for any $t\in I$, \EQn{\label{eq:fra-energy-bound-w-linfty}
\norm{\W(t)}_{L_x^\I(\R)} \lsm t^{\frac{p-4}{2(p+4)}-} N_0^{\frac{4}{p+4}(1-s)+}.
}
Let $0<\ep<1$ be any sufficiently small parameter. By Sobolev's inequality, interpolation, \eqref{eq:fra-energy-bound-w}, and  \eqref{eq:fra-mass-bound-w}, then
\EQ{
\norm{\W(t)}_{L_x^\I} 
\lsm & \normb{\jb{\nabla}^{\frac{1}{p+2}+\ep}\W(t)}_{L_x^{p+2}} \\
\lsm & \norm{\W(t)}_{L_x^{p+2}}^{\frac{p+2}{p+4}-2\ep}  \normb{\jb{\nabla}^{\frac{p+4}{2(p+2)}}\W(t)}_{L_x^{p+2}}^{\frac{2}{p+4}+2\ep} \\
\lsm & \norm{\W(t)}_{L_x^{p+2}}^{\frac{p+2}{p+4}-2\ep}  \normb{\jb{\nabla}\W(t)}_{L_x^{2}}^{\frac{2}{p+4}+2\ep} \\
\lsm & t^{\frac{2}{p+4}\brko{\frac {p}{4}-1}-\frac{4-p}{2}\ep} N_0^{(\frac{4}{p+4}+\frac{2p}{p+2}\ep)(1-s)}.
}
This finishes the proof of \eqref{eq:fra-energy-bound-w-linfty}. Then, by \eqref{eq:fra-energy-bound-w-linfty}, H\"older's inequality in $t$, and Lemma \ref{lem:fra-linear},
\EQ{
\eqref{eq:fra-energy-bound-wv-high-high-v-2}
\lsm & \sum_{N_1\gsm N_0}N_1^{1-2s} \int_0^{T_0} t^{\frac p4-1} \norm{N_1^s\V_{N_1}}_{L_x^\I} \brkb{t^{\frac{p-4}{2(p+4)}-}N_0^{\frac{4}{p+4}(1-s)+}}^{p-1} \dd t \\
\lsm & N_0^{1-2s} \int_0^{T_0} t^{\frac p4-1} \norm{\jb{\nabla}^s\V}_{L_x^\I} t^{\frac{(p-4)(p-1)}{2(p+4)}-}N_0^{\frac{4(p-1)}{p+4}(1-s)+} \dd t \\
\lsm & N_0^{1-s}N_0^{-s + \frac{4(p-1)}{p+4}(1-s)+} \norm{\jb{\nabla}^s\V}_{L_t^4L_x^\I} \brkb{\int_0^{T_0}  t^{-\frac{(4-p)(3p+2)}{3(p+4)}-}\dd t}^{\frac43} \\
	\lsm & \de_0 N_0^{1-s}N_0^{-s + \frac{4(p-1)}{p+4}(1-s)} \brkb{\int_0^{T_0}  t^{-\frac{(4-p)(3p+2)}{3(p+4)}}\dd t}^{\frac43}.
}
Since for $2<p<4$, we can verify that
\EQ{
	0<\frac{(4-p)(3p+2)}{3(p+4)} < \frac89,
}
then by choosing suitable implicit small parameter depending on $p$,
\EQ{
	\int_0^{T_0}  t^{-\frac{(4-p)(3p+2)}{3(p+4)}-}\dd t \lsm 1.
}
Note also that $s>\frac35>\frac{4(p-1)}{5p}$, then
\EQ{
N_0^{-s}N_0^{\frac{4(p-1)}{p+4}(1-s)+}\lsm1.
}
Combining the above two inequalities,
\EQn{\label{esti:fra-energy-bound-wv-high-high-v-2}
	\eqref{eq:fra-energy-bound-wv-high-high-v-2}
	\lsm \de_0 N_0^{1-s}.
}
By \eqref{esti:fra-energy-bound-wv-high-high-v-1} and \eqref{esti:fra-energy-bound-wv-high-high-v-2},
\EQn{\label{esti:fra-energy-bound-wv-high-high-v}
	\eqref{eq:fra-energy-bound-wv-high-high-v} \lsm \de_0 N_0^{1-s}\brkb{\de_0^p + \de_0 N_0^{1-s}} \lsm \de_0 N_0^{2(1-s)}.
}

$\bullet$ \textbf{Estimate of \eqref{eq:fra-energy-bound-wv-high-high-w}.} For this term, we can directly transfer the first-order derivative to $\W_{N_1}$. By H\"older's inequality,
\begin{subequations}
	\EQnn{
		\eqref{eq:fra-energy-bound-wv-high-high-w} \lsm & \sum_{N\lsm N_1} \int_0^{T_0}\norm{\nabla \W}_{L_x^2} \norm{N \V_N}_{L_x^\I} \norm{\W_{N_1}}_{L_x^2} \norm{\V}_{L_x^\I}^{p-1} \dd t \label{eq:fra-energy-bound-wv-high-high-w-1}\\
		& + \sum_{N\lsm N_1} \int_0^{T_0}\norm{\nabla \W}_{L_x^2} \norm{N \V_N}_{L_x^\I} \norm{\W_{N_1}}_{L_x^2}  \norm{\W}_{L_x^\I}^{p-1} \dd t \label{eq:fra-energy-bound-wv-high-high-w-2}
	}
\end{subequations}
To this end, we first give an estimate that can transfer derivative. By Schur's test in Lemma \ref{lem:schurtest}, and interpolating with \eqref{eq:fra-energy-bound-w} and \eqref{eq:fra-mass-bound-w}, we have that for any $0\loe l\loe 1$,
\EQn{\label{esti:fra-energy-bound-wv-high-high-w-transfer-derivative}
	\sum_{N\lsm N_1}\norm{N \V_N}_{L_x^\I} \norm{\W_{N_1}}_{L_x^2} \lsm & \sum_{N\lsm N_1} \brkb{\frac{N}{N_1}}^{1-l}\norm{N^l \V_N}_{L_x^\I} \norm{N_1^{1-l}\W_{N_1}}_{L_x^2} \\
	\lsm & \norm{N^l\V_N}_{l_N^2L_x^\I} \norm{N_1^{1-l}\W_{N_1}}_{l_{N_1}^2L_x^2} \\
	\lsm & \norm{N^l\V_N}_{l_N^2L_x^\I} \norm{\jb{\nabla}^{1-l}\W}_{L_x^2} \\
	\lsm & \brkb{t^{\frac p4-1} N_0^{1-s}}^{1-l} \norm{N^l\V_N}_{l_N^2L_x^\I}.
}
For \eqref{eq:fra-energy-bound-wv-high-high-w-1}, we can transfer all the derivative to $\W_{N_1}$. Applying \eqref{esti:fra-energy-bound-wv-high-high-w-transfer-derivative} with $l=0$, Sobolev's inequality, and \eqref{eq:fra-energy-bound-w},
\EQn{\label{esti:fra-energy-bound-wv-high-high-w-1}
\eqref{eq:fra-energy-bound-wv-high-high-w-1} \lsm & \int_0^{T_0} t^{\frac p4-1} N_0^{1-s} \cdot t^{\frac p4-1} N_0^{1-s} \norm{\V_N}_{l_N^2L_x^\I}\norm{\V}_{L_x^\I}^{p-1} \dd t\\
	\lsm & N_0^{2(1-s)}\int_0^{T_0} t^{\frac p2-2} \dd t \norm{\V}_{L_t^\I H_x^s}^{p} \\ \lsm & \de_0^p N_0^{2(1-s)} .
}
For \eqref{eq:fra-energy-bound-wv-high-high-w-2}, we can only transfer at most $\frac16$-order derivative due to the singularity near the time origin. Applying \eqref{esti:fra-energy-bound-wv-high-high-w-transfer-derivative} with $l=\frac56$, H\"older's inequality,  \eqref{eq:fra-energy-bound-w-linfty}, and $\normb{N^{\frac56}\V_N}_{l_N^2L_t^4L_x^\I}\lsm N_0^{\frac56-s+} \normb{\jb{\nabla}^s\V}_{L_t^4L_x^\I}$, we have that
\EQn{\label{esti:fra-energy-bound-wv-high-high-w-2-1}
\eqref{eq:fra-energy-bound-wv-high-high-w-2} \lsm & \int_0^{T_0} t^{\frac p4-1}N_0^{1-s} \cdot \brkb{t^{\frac p4-1} N_0^{1-s}}^{\frac16} \normb{N^{\frac56}\V_N}_{l_N^2L_x^\I} \norm{\W}_{L_x^\I}^{p-1} \dd t \\
\lsm & N_0^{1-s} \normb{N^{\frac56}\V_N}_{l_N^2L_t^4L_x^\I} \normb{ t^{\frac76(\frac p4-1)} N_0^{\frac16(1-s)} \norm{\W}_{L_x^\I}^{p-1} }_{L_t^{\frac43}} \\
\lsm &  N_0^{2(1-s)}N_0^{-\frac16s+}\normb{\jb{\nabla}^{s}\V}_{L_t^4L_x^\I} \normb{t^{\frac76(\frac p4-1)}   t^{\frac{(p-4)(p-1)}{2(p+4)}-}N_0^{\frac{4(p-1)}{p+4}(1-s)+}}_{L_t^{\frac43}} \\
\lsm & \de_0 N_0^{2(1-s)} N_0^{-\frac16s + \frac{4(p-1)}{p+4}(1-s)+} \normb{ t^{-\frac{(4-p)(19p+16)}{24(p+4)}-}}_{L_t^{\frac43}}.
}
For $2<p<4$, we have that
\EQ{
-1<-\frac43\cdot\frac{(4-p)(19p+16)}{24(p+4)} = -\frac{(4-p)(19p+16)}{18(p+4)}<0,
}
then by choosing suitable implicit small parameter depending on $p$,
\EQn{\label{esti:fra-energy-bound-wv-high-high-w-2-time}
\normb{ t^{-\frac{(4-p)(19p+16)}{24(p+4)}-}}_{L_t^{\frac43}} \lsm 1.
}
Furthermore, when $s>\frac{9}{10}$,
\EQn{\label{esti:fra-energy-bound-wv-high-high-w-2-energy}
-\frac16s + \frac{4(p-1)}{p+4}(1-s) + < -\frac16s + \frac32(1-s) <0.
}
By \eqref{esti:fra-energy-bound-wv-high-high-w-2-1}, \eqref{esti:fra-energy-bound-wv-high-high-w-2-time}, and \eqref{esti:fra-energy-bound-wv-high-high-w-2-energy},
\EQn{\label{esti:fra-energy-bound-wv-high-high-w-2-2}
	\eqref{eq:fra-energy-bound-wv-high-high-w-2} \lsm \de_0 N_0^{2(1-s)}.
}
Then, by \eqref{esti:fra-energy-bound-wv-high-high-w-1} and \eqref{esti:fra-energy-bound-wv-high-high-w-2-2},
\EQn{\label{esti:fra-energy-bound-wv-high-high-w}
	\eqref{eq:fra-energy-bound-wv-high-high-w} \lsm \eqref{eq:fra-energy-bound-wv-high-high-w-1} + \eqref{eq:fra-energy-bound-wv-high-high-w-2} \lsm \de_0 N_0^{2(1-s)}.
}

$\bullet$ \textbf{Estimate of \eqref{eq:fra-energy-bound-wv-high-low}.} This is the main term that decides the lower bound of $s$, namely $\frac{18}{19}$. For this term, we need to use bilinear Strichartz estimate to transfer derivative. 
By H\"older's inequality, \eqref{eq:fra-energy-bound-bilinear}, and Lemma \ref{lem:fra-linear},
\EQn{\label{esti:fra-energy-bound-wv-high-low-1}
	\eqref{eq:fra-energy-bound-wv-high-low} \lsm & \sum_{N_1\ll N} \norm{\nabla \V_N \U_{N_1}}_{L_{t,x}^2}^{\frac{2}{9}} \normb{|\nabla \W| |\nabla\V_N|^{\frac{7}{9}} |\U_{N_1}|^{\frac{7}{9}}|\U|^{p-1}}_{L_{t,x}^{\frac{9}{8}}} \\
	\lsm & \sum_{N_1\ll N} \norm{\nabla \V_N \U_{N_1}}_{L_{t,x}^2}^{\frac{2}{9}} \norm{\nabla \V_N}_{L_t^4 L_x^\I}^{\frac79} \normb{ \norm{\nabla \U}_{ L_x^2} \norm{\U_{N_1}}_{L_{x}^2}^{\frac79} \norm{\U}_{L_{x}^\I}^{p-1} }_{L_t^{\frac{36}{25}}} \\
	\lsm & \de_0 \sum_{N_1\ll N} \norm{\nabla \V_N \U_{N_1}}_{L_{t,x}^2}^{\frac{2}{9}} N^{\frac79(1-s)} \norm{\jb{\nabla}^{s} \V_N}_{L_t^4 L_x^\I}^{\frac79} \normb{\norm{\nabla \W}_{ L_x^2}\norm{\U_{N_1}}_{L_{x}^2}^{\frac79} \norm{\U}_{L_{x}^\I}^{p-1}}_{L_t^{\frac{36}{25}}} \\
	\lsm & \de_0 \sum_{N_1\ll N} N^{\frac29(\frac12-s)}  N_1^{\frac{p}{9(p+2)}} N_0^{\frac49(1-s)} N^{\frac79(1-s)} \normb{t^{\frac p4-1}N_0^{1-s}\norm{\U_{N_1}}_{L_{x}^2}^{\frac79} \norm{\U}_{L_{x}^\I}^{p-1}}_{L_t^{\frac{36}{25}}} \\
	\lsm & \de_0 N_0^{\frac{13}{9}(1-s)} \sum_{N_1\ll N}   N^{-\frac1{9}} N^{1-s} \normb{t^{\frac p4-1}\normb{\jb{\nabla}^{\frac{p}{7(p+2)}}\U_{N_1}}_{L_{x}^2}^{\frac79} \norm{\U}_{L_{x}^\I}^{p-1}}_{L_t^{\frac{36}{25}}}.
}
By interpolation and \eqref{eq:fra-energy-bound-w},
\EQ{
	\normb{\jb{\nabla}^{\frac{p}{7(p+2)}}\U_{N_1}}_{L_{x}^2} \lsm & \normb{\jb{\nabla}^{\frac{p}{7(p+2)}}\V_{N_1}}_{L_{x}^2} + \normb{\jb{\nabla}^{\frac{p}{7(p+2)}}\W_{N_1}}_{L_{x}^2} \\
	\lsm & \norm{\V_{N_1}}_{H_x^s} + \normb{\jb{\nabla}^{\frac{p}{7(p+2)}}\W_{N_1}}_{L_{x}^2} \\
	\lsm & \norm{\V_{N_1}}_{H_x^s} + \norm{\W_{N_1}}_{L_{x}^{2}}^{\frac{6p+14}{7(p+2)}}\normb{\jb{\nabla}\W_{N_1}}_{L_{x}^2}^{\frac{p}{7(p+2)}} \\
	\lsm & 1 + \brkb{t^{\frac p4-1}N_0^{1-s}}^{\frac{p}{7(p+2)}} \\
	\lsm & t^{\frac{p}{7(p+2)}(\frac p4-1)}N_0^{\frac{p}{7(p+2)}(1-s)}.
}
By Sobolev's inequality, \eqref{eq:fra-energy-bound-w-linfty}, and Lemma \ref{lem:fra-linear},
\EQn{\label{eq:fra-energy-bound-u-linfty}
\norm{\U}_{L_{x}^\I} \lsm & \norm{\V}_{L_{x}^\I} + \norm{\W}_{L_{x}^\I} \\
	\lsm & \norm{\V}_{H_x^s} + t^{\frac{p-4}{2(p+4)}-} N_0^{\frac{4}{p+4}(1-s)+} \\
	\lsm & t^{\frac{p-4}{2(p+4)}-} N_0^{\frac{4}{p+4}(1-s)+}.
}
Thus, combining the above two inequalities,
\EQn{\label{esti:fra-energy-bound-wv-high-low-2}
&\normb{t^{\frac p4-1}\normb{\jb{\nabla}^{\frac{p}{7(p+2)}}\U_{N_1}}_{L_{x}^2}^{\frac79} \norm{\U}_{L_{x}^\I}^{p-1}}_{L_t^{\frac{36}{25}}} \\
\lsm & \normb{t^{\frac p4-1} t^{\frac{p}{9(p+2)}(\frac p4-1)}N_0^{\frac{p}{9(p+2)}(1-s)} t^{\frac{(p-4)(p-1)}{2(p+4)}-} N_0^{\frac{4(p-1)}{p+4}(1-s)+}}_{L_t^{\frac{36}{25}}} \\
\lsm & \normb{t^{\frac{(p-4)(7p^2+19p+9)}{9(p+2)(p+4)}-}}_{L_t^{\frac{36}{25}}} N_0^{\frac{37p^2+40p-72}{9(p+2)(p+4)}(1-s)+}.
}
For $2<p<4$, we can check that
\EQ{
-1<\frac{4(p-4)(7p^2+19p+9)}{25(p+2)(p+4)}<0\text{, and }\frac{37p^2+40p-72}{9(p+2)(p+4)}<\frac{85}{54}.
}
Then, by \eqref{esti:fra-energy-bound-wv-high-low-1} and \eqref{esti:fra-energy-bound-wv-high-low-2},
\EQn{\label{esti:fra-energy-bound-wv-high-low}
	\eqref{eq:fra-energy-bound-wv-high-low} \lsm & \de_0 N_0^{\frac{13}{9}(1-s)} \sum_{N_1,N:N_0\loe N_1\ll N}   N^{-\frac1{9}} N^{1-s}N_0^{\frac{85}{54}(1-s)} \\
	\lsm & \de_0 N_0^{\frac{13}{9}(1-s)}  N_0^{-\frac1{9} +1-s} N_0^{\frac{85}{54}(1-s)} \\
	\lsm & \de_0 N_0^{2(1-s)} N_0^{-\frac1{9} + \frac{109}{54}(1-s)} \\
	\lsm & \de_0 N_0^{2(1-s)},
}
where in the last inequality we use $s>\frac{18}{19}>\frac{103}{109}$.

$\bullet$ \textbf{Estimate of \eqref{eq:fra-energy-bound-vv-high-high}.} For this term, we can transfer the derivative to $\U_{\gsm N_2}$. First, by H\"older's inequality,
\EQ{
	\eqref{eq:fra-energy-bound-vv-high-high} \lsm & \sum_{N_1\loe N_2}\int_0^{T_0} \norm{\nabla\V_{N_1}}_{L_x^2} \norm{\nabla \V_{N_2}}_{L_x^2} \norm{\U_{\gsm N_2}}_{L_x^\I} \norm{\U}_{L_x^\I}^{p-1}\dd t \\
	\lsm & \sum_{N_1\loe N_2}\int_0^{T_0} \normb{N_1^{\frac{4}{5}} \V_{N_1}}_{L_x^2} \normb{N_2^{\frac{4}{5}} \V_{N_2}}_{L_x^2} \normb{\jb{\nabla}^{\frac{2}{5}}\U_{\gsm N_2}}_{L_x^\I} \norm{\U}_{L_x^\I}^{p-1}\dd t.
}
Note that by Lemma \ref{lem:fra-linear},
\EQ{
	\sum_{N_1\loe N_2}  \normb{N_1^{\frac{4}{5}} \V_{N_1}}_{L_x^2} \normb{N_2^{\frac{4}{5}} \V_{N_2}}_{L_x^2} \lsm N_0^{2({\frac{4}{5}}-s)} \norm{\V}_{H_x^s}^2 \lsm \de_0^2N_0^{2(1-s)}N_0^{-\frac25}.
}
Moreover, by $s>\frac{9}{10}$, Sobolev's inequality, \eqref{eq:fra-energy-bound-w}, and Lemma \ref{lem:fra-linear},
\EQ{
	\normb{\jb{\nabla}^{\frac25}\U_{\gsm N_2}}_{L_x^\I} \lsm  \normb{\jb{\nabla}^{\frac{9}{10}}\U}_{L_x^2} \lsm \normb{\jb{\nabla}^{s}\U}_{L_x^2} \lsm \norm{\W}_{H^1} + \norm{\V}_{H^s}
	\lsm  t^{\frac p4-1} N_0^{1-s}.
}
Combining the above three inequalities and \eqref{eq:fra-energy-bound-u-linfty},
\EQ{
\eqref{eq:fra-energy-bound-vv-high-high} \lsm & \de_0^2 N_0^{2(1-s)}N_0^{-\frac25} \int_0^{T_0} t^{\frac p4-1} N_0^{1-s} t^{\frac{(p-4)(p-1)}{2(p+4)}-} N_0^{\frac{4(p-1)}{p+4}(1-s)+}  \dd t \\
\lsm & \de_0^2N_0^{2(1-s)} N_0^{-\frac25 + \frac{5p}{p+4}(1-s)} \int_0^{T_0} t^{\frac{(3p+2)(p-4)}{4(p+4)}-} \dd t. 
}
For $2<p<4$, we have that
\EQ{
-\frac23<\frac{(3p+2)(p-4)}{4(p+4)}<0,
}
then by choosing suitable implicit small parameter depending on $p$,
\EQ{
\int_0^{T_0} t^{\frac{(3p+2)(p-4)}{4(p+4)}-} \dd t \lsm 1.
}
Therefore, by $2<p<4$ and $s>\frac{21}{25}$,
\EQn{\label{esti:fra-energy-bound-vv-high-high}
\eqref{eq:fra-energy-bound-vv-high-high} \lsm & \de_0^2N_0^{2(1-s)} N_0^{-\frac25 + \frac{5}{2}(1-s)}  \\
\lsm & \de_0^2N_0^{2(1-s)}.
}

$\bullet$ \textbf{Estimate of \eqref{eq:fra-energy-bound-vv-high-low}.} For this term, we also need to invoke bilinear Strichartz estimate to transfer derivative. We split this term into two sub-cases:
\EQnnsub{
	\eqref{eq:fra-energy-bound-vv-high-low} \lsm & \sum_{N_1\loe N_2} \sum_{N:N\ll N_2} \int_0^{T_0} \int_\R |\nabla\V_{N_1}||\nabla\V_{N_2}||\V_{N}||\U|^{p-1}\dd x\dd t \label{eq:fra-energy-bound-vv-high-low-1}\\
	& + \sum_{N_1\loe N_2} \sum_{N:N\ll N_2} \int_0^{T_0} \int_\R |\nabla\V_{N_1}||\nabla\V_{N_2}||\W_{N}||\U|^{p-1}\dd x\dd t. \label{eq:fra-energy-bound-vv-high-low-2}
}
For \eqref{eq:fra-energy-bound-vv-high-low-1}, note that by Lemma \ref{lem:bilinearstrichartz} and Lemma \ref{lem:fra-linear},
\EQ{
	\norm{\nabla\V_{N_2}\V_{N}}_{L_{t,x}^2} \lsm N_2^{\frac12}\norm{P_{N_2}\V_+}_{L_x^2}\norm{P_{N}\V_+}_{L_x^2} \lsm N_2^{\frac12-s}N^{-s}\de_0^2.
}
Applying this inequality, and combining H\"older's inequality, \eqref{eq:fra-energy-bound-u-linfty}, and Lemma \ref{lem:fra-linear},
\EQ{
	\eqref{eq:fra-energy-bound-vv-high-low-1} \lsm & \sum_{N_1\loe N_2} \sum_{N:N\ll N_2} \norm{\nabla\V_{N_2}\V_{N}}_{L_{t,x}^2} \norm{\nabla\V_{N_1}\U^{p-1}}_{L_{t,x}^2} \\
	\lsm & \de_0^2 \sum_{N_1\loe N_2} \sum_{N:N\ll N_2} N_2^{\frac12-s}N^{-s} \normb{\norm{\nabla\V_{N_1}}_{L_x^2} \norm{\U}_{L_x^\I}^{p-1}}_{L_{t}^2} \\
	\lsm & \de_0^3 \sum_{N_1\loe N_2} \sum_{N:N\ll N_2} N_2^{\frac12-s}N^{-s}N_1^{1-s} \brkb{\int_0^{T_0} t^{-\frac{(4-p)(p-1)}{(p+4)}-}\dd t}^{1/2} N_0^{\frac{4(p-1)}{p+4}(1-s)+}.
}
Then, note that by $s>\frac12$,
\EQ{
	\sum_{N_1\loe N_2} \sum_{N:N\ll N_2} N_2^{\frac12-s}N^{-s}N_1^{1-s} \lsm \sum_{N_1\loe N_2}  N_2^{\frac12-2s}N_1^{1-s} \lsm \sum_{N_2}  N_2^{\frac32-3s} \lsm 1,
}
and for $2<p<4$,
\EQ{
	0<\frac{(4-p)(p-1)}{(p+4)}<\frac13\text{, and }\frac{4(p-1)}{p+4}<2.
}
Therefore, we can obtain
\EQ{
	\eqref{eq:fra-energy-bound-vv-high-low-1} \lsm \de_0^3 N_0^{2(1-s)}.
}
Next, we consider the term \eqref{eq:fra-energy-bound-vv-high-low-2}. First, by H\"older's inequality, Lemma \ref{lem:fra-linear}, and \eqref{eq:fra-energy-bound-bilinear}
\EQ{
\eqref{eq:fra-energy-bound-vv-high-low-2} \lsm  \sum_{N_1\loe N_2}\sum_{N:N\ll N_2} & \norm{\nabla\V_{N_1}}_{L_t^\I L_x^2}\norm{\nabla \V_{N_2}\W_N}_{L_{t,x}^2}^{\frac34}\norm{\nabla\V_{N_2}}_{L_t^4 L_x^\I}^{\frac14} \\
& \cdot  \normb{\norm{\W_{N}}_{L_x^2}^{\frac14}\norm{\U}_{L_x^\I}^{p-1}}_{L_t^{\frac{16}{9}}} \\
\lsm \sum_{N_1\loe N_2}\sum_{N:N\ll N_2} & \de_0 N_1^{1-s}\brkb{N_2^{\frac12-s}N^{\frac{p}{2(p+2)}}N_0^{\frac{2(p+1)}{p+2}(1-s)}}^{\frac34}N_2^{\frac14(1-s)} \\
& \cdot  \normb{\norm{\W_{N}}_{L_x^2}^{\frac14}\norm{\U}_{L_x^\I}^{p-1}}_{L_t^{\frac{16}{9}}} \\
\lsm \sum_{N_1\loe N_2}\sum_{N:N\ll N_2} & \de_0 N_1^{1-s}N_2^{-\frac38+(1-s)}N_0^{\frac{3(p+1)}{2(p+2)}(1-s)} \normb{\normo{N^{\frac{3p}{2(p+2)}}\W_{N}}_{L_x^2}^{\frac14}\norm{\U}_{L_x^\I}^{p-1}}_{L_t^{\frac{16}{9}}}.
}
Note that by $\frac{3p}{2(p+2)}<1$,
\EQ{
\normb{N^{\frac{3p}{2(p+2)}}\W_{N}}_{L_x^2}^{\frac14} \lsm N^{0-} \norm{\jb{\nabla}\W}_{L_x^2}^{\frac14} \lsm N^{0-} t^{\frac14(\frac p4-1)} N_0^{\frac14(1-s)},
}
then by \eqref{eq:fra-energy-bound-u-linfty}, we further get
\EQ{
\eqref{eq:fra-energy-bound-vv-high-low} \lsm & \sum_{N_1\loe N_2} \de_0 N_1^{1-s}N_2^{-\frac38+(1-s)}N_0^{\frac{3(p+1)}{2(p+2)}(1-s)} \normb{t^{\frac14(\frac p4-1)} N_0^{\frac14(1-s)} \norm{\U}_{L_x^\I}^{p-1}}_{L_t^{\frac{16}{9}}} \\
\lsm & \sum_{N_1\loe N_2} \de_0 N_1^{1-s} N_2^{-\frac38+(1-s)}N_0^{\frac{p(23p+52)}{4(p+2)(p+4)}(1-s)} \normb{t^{\frac{(p-4)(9p-4)}{16(p+4)}-} }_{L_t^{\frac{16}{9}}}.
}
For $2<p<4$, we have
\EQ{
-\frac{14}{27}<\frac{(p-4)(9p-4)}{16(p+4)}<0,
}
then by choosing suitable implicit small parameter depending on $p$,
\EQ{
\normb{t^{\frac{(p-4)(9p-4)}{16(p+4)}-} }_{L_t^{\frac{16}{9}}} \lsm 1.
}
Moreover, by $s>\frac{13}{16}$,
\EQ{
\sum_{N_1\loe N_2}  N_1^{1-s} N_2^{-\frac38+(1-s)} \lsm \sum_{N_2}  N_2^{-\frac38+2(1-s)} \lsm N_0^{-\frac38+2(1-s)}.
}
Then, combining the above two inequalities,
\EQ{
\eqref{eq:fra-energy-bound-vv-high-low} 
\lsm & \de_0  N_0^{2(1-s)}N_0^{-\frac38+\frac{31p^2+100p+64}{4(p+2)(p+4)}(1-s)}. 
}
For $2<p<4$, we have
\EQ{
\frac{10p^2+17p-12}{4(p+2)(p+4)}<5.
}
Therefore, by $s>\frac{37}{40}$,
\EQn{\label{esti:fra-energy-bound-vv-high-low}
\eqref{eq:fra-energy-bound-vv-high-low} \lsm \de_0  N_0^{2(1-s)} N_0^{-\frac38+5(1-s)} \lsm \de_0  N_0^{2(1-s)}.
}

Finally, Lemma \ref{lem:fra-energy-dd} in 1D case follows from \eqref{esti:fra-energy-bound-wv-high-high-v}, 
\eqref{esti:fra-energy-bound-wv-high-high-w}, 
\eqref{esti:fra-energy-bound-wv-high-low}, 
\eqref{esti:fra-energy-bound-vv-high-high}, and
\eqref{esti:fra-energy-bound-vv-high-low}.
\end{proof}
\subsection{Energy estimate when $p<\frac{2}{d-2}$ and $d\goe2$}\label{sec:fra-energy-dd-2}
\begin{proof}[Proof of Lemma \ref{lem:fra-energy-dd} in 2D and higher dimensional cases]
Throughout Section \ref{sec:fra-energy-dd-2}, we use two parameters
\EQn{\label{defn:fra-energy-dd-2-th}
\th_1 := \frac{2-(d-2)p}{10}\text{, and }\th_2 := \frac{(dp-2)^2}{2(d+2)}.
}
We make the following decomposition:
\EQnnsub{
|\int_{t_0}^{T_0}\int_{\R^d}|\U|^p\U\wb{\V_t}\dd x\dd t| \lsm & \sum_{N\lsm N_1} N_1^{0+} \int_{0}^{T_0}\int_{\R^d} |\nabla \W| |\nabla \V_N||\V_{ N_1}|^p \dd x \dd t \label{eq:fra-energy-bound-wv-high-high-v-dd}\\
& + \sum_{N\lsm N_1} N_1^{0+} \int_{0}^{T_0}\int_{\R^d} |\nabla \W| |\nabla \V_N||\W_{ N_1}|^p \dd x \dd t \label{eq:fra-energy-bound-wv-high-high-w-dd}\\
& + \sum_{N_1\ll N} N_1^{0+} \int_{0}^{T_0}\int_{\R^d} |\nabla \W| |\nabla \V_N||\V_{ N_1}|^p \dd x \dd t \label{eq:fra-energy-bound-wv-high-low-v-dd}\\
& + \sum_{N_1\ll N} N_1^{0+} \int_{0}^{T_0}\int_{\R^d} |\nabla \W| |\nabla \V_N||\W_{ N_1}|^p \dd x \dd t \label{eq:fra-energy-bound-wv-high-low-w-dd}\\
& + \sum_{N_1\loe N_2}\int_{0}^{T_0} \int_{\R^d} |\nabla\V_{N_1}||\nabla\V_{N_2}||\U_{\gsm N_2}|^p\dd x\dd t \label{eq:fra-energy-bound-vv-high-high-dd}\\
& + \sum_{N_1\loe N_2}\int_{0}^{T_0} \int_{\R^d} |\nabla\V_{N_1}||\nabla\V_{N_2}||\V_{\ll N_2}|^p\dd x\dd t \label{eq:fra-energy-bound-vv-high-low-v-dd} \\
& + \sum_{N_1\loe N_2}\int_{0}^{T_0} \int_{\R^d} |\nabla\V_{N_1}||\nabla\V_{N_2}||\W_{\ll N_2}|^p\dd x\dd t. \label{eq:fra-energy-bound-vv-high-low-w-dd}
}
Next, we estimate \eqref{eq:fra-energy-bound-wv-high-high-v-dd}-\eqref{eq:fra-energy-bound-vv-high-low-w-dd} one by one.
In the proof of the $d\goe 2$ case, the spacetime norms are taken over $(t,x) \in (0,T_0]\times\R^d$.

$\bullet$ \textbf{Estimate of \eqref{eq:fra-energy-bound-wv-high-high-v-dd}.} By H\"older's,  Bernstein's inequality, and Lemma \ref{lem:fra-linear},
\EQ{
\eqref{eq:fra-energy-bound-wv-high-high-v-dd} 
\lsm & \sum_{N\lsm N_1} N_1^{0+} \norm{\nabla \W}_{L_{t,x}^2} \norm{\nabla \V_N}_{L_t^2 L_x^{\frac{2d}{d-2}}}\norm{\V_{N_1}}_{L_t^{\I}L_x^{dp}}^p \\
\lsm & \de_0 \sum_{N\lsm N_1} N^{1-s} \norm{\nabla \W}_{L_{t,x}^2} \normb{\jb{\nabla}^{\frac{dp-2}{2p}+}\V_{N_1}}_{L_t^{\I}L_x^{2}}^p \\
\lsm & \de_0 \sum_{N\lsm N_1} N^{1-s} N_1^{-ps+\frac{dp-2}{2}+} \norm{\nabla \W}_{L_{t,x}^2} \normb{\jb{\nabla}^{s}\V_{N_1}}_{L_t^{\I}L_x^{2}}^p \\
\lsm & \de_0^{p+1} \norm{\nabla \W}_{L_{t,x}^2}\sum_{N\lsm N_1}  N^{1-s}N_1^{-ps+\frac{dp-2}{2}+}.
}
By \eqref{eq:fra-energy-bound-w},
\EQn{\label{eq:fra-energy-bound-wv-high-high-v-dd-1}
\norm{\nabla \W}_{L_{t,x}^2} \lsm & N_0^{1-s} \normb{t^{\frac{dp}{4}-1}}_{L_t^{2}} \\
\lsm & N_0^{1-s} \brkb{\int_{0}^{T_0} t^{\frac{dp}{2}-2}\dd t}^{\frac12} \\
\lsm & N_0^{1-s} T_0^{\frac{dp}{4}-\frac12} \\
\lsm & N_0^{1-s}.
}
Moreover, by the choice of $s_0$ in Definition \ref{defn:fra-scattering-s0}, we have
\EQ{
1-s<\frac{2-(d-2)p}{2(1+p)}<\frac{2-(d-2)p}{2p},
}
which is well defined since we impose that $p<\frac{2}{d-2}$. Thus,
\EQ{
-ps+\frac{dp-2}{2}<0\text{, and }\frac{dp}{2}-(1+p)s<0.
}
Consequently,
\EQn{\label{eq:fra-energy-bound-wv-high-high-v-dd-2}
\sum_{N\lsm N_1}  N^{1-s}N_1^{-ps+\frac{dp-2}{2}+} \lsm \sum_{N\in2^\N}  N^{\frac{dp}{2}-(1+p)s+} \lsm 1.
}
By \eqref{eq:fra-energy-bound-wv-high-high-v-dd-1} and \eqref{eq:fra-energy-bound-wv-high-high-v-dd-2}, we have
\EQn{\label{esti:fra-energy-bound-wv-high-high-v-dd}
\eqref{eq:fra-energy-bound-wv-high-high-v-dd} 
\lsm \de_0^{p+1} N_0^{1-s}.
}

$\bullet$ \textbf{Estimate of \eqref{eq:fra-energy-bound-wv-high-high-w-dd}.} By H\"older's, Sobolev's inequalities,
\EQ{
\eqref{eq:fra-energy-bound-wv-high-high-w-dd} \lsm & \sum_{N\lsm N_1} N_1^{0+} \norm{\nabla\V_{N}}_{L_t^2 L_x^{\frac{2d}{d-2}}} \normb{\norm{\nabla\W}_{L_x^2}\norm{\W_{N_1}}_{L_x^{dp}}^p}_{L_t^2} \\
\lsm & \sum_{N\lsm N_1} N_1^{0+} N^{1-s} \norm{\jb{\nabla}^s\V_{N}}_{L_t^2 L_x^{\frac{2d}{d-2}}} \normb{\norm{\nabla\W}_{L_x^2}\normb{\jb{\nabla}^{\frac{dp-2}{2p}}\W_{N_1}}_{L_x^{2}}^p}_{L_t^2} \\
\lsm & \de_0\sum_{N\lsm N_1} N_1^{0-} \normb{\norm{\nabla\W}_{L_x^2}\normb{\jb{\nabla}^{\frac{dp-2}{2p}+\frac1p(1-s)++}\W_{N_1}}_{L_x^{2}}^p}_{L_t^2} \\
\lsm & \de_0 \normb{\norm{\nabla\W}_{L_x^2}\normb{\jb{\nabla}^{\frac{dp-2}{2p}+\frac1p(1-s)+}\W}_{L_x^{2}}^p}_{L_t^2}.
}
Note that by the choice of $s_0$ in Definition \ref{defn:fra-scattering-s0}, we have
\EQ{
1-s<\frac{2-(d-2)p}{2},
}
which gives
\EQ{
\frac{dp-2}{2p}+\frac1p(1-s)<1.
}
Therefore, we can apply the interpolation,  \eqref{eq:fra-energy-bound-w}, and \eqref{eq:fra-mass-bound-w} to obtain
\EQ{
\normb{\jb{\nabla}^{\frac{dp-2}{2p}+\frac1p(1-s)+}\W}_{L_x^{2}}\lsm & \normb{\W}_{L_x^{2}}^{\frac{2-(d-2)p}{2p}-\frac1p(1-s)-} \normb{\jb{\nabla}\W}_{L_x^{2}}^{\frac{dp-2}{2p}+\frac1p(1-s)+} \\
\lsm & t^{(\frac d4p-1)\brkb{\frac{dp-2}{2p}+\frac1p(1-s)}-} N_0^{\brkb{\frac{dp-2}{2p}+\frac1p(1-s)}(1-s)+}.
}
Thus, using this inequality and \eqref{eq:fra-energy-bound-w},
\EQ{
\eqref{eq:fra-energy-bound-wv-high-high-w-dd} \lsm & \de_0 N_0^{(\frac{dp}{2}+1-s+)(1-s)} \normb{t^{(\frac{dp}{4}-1)(\frac{dp}{2}+1-s+)}}_{L_t^2}.
}
By the choice of $s_0$ in Definition \ref{defn:fra-scattering-s0},
\EQ{
1-s<\min\fbrk{\frac{4-dp}{2},\frac{(dp-2)^2}{2(4-dp)}},
}
then
\EQ{
\frac{dp}{2}+1-s<2\text{, and }-\frac12<(\frac{dp}{4}-1)(\frac{dp}{2}+1-s)<0,
}
Therefore, 
\EQn{\label{esti:fra-energy-bound-wv-high-high-w-dd}
\eqref{eq:fra-energy-bound-wv-high-high-w-dd} \lsm \de_0 N_0^{2(1-s)}.
}

$\bullet$ \textbf{Estimate of \eqref{eq:fra-energy-bound-wv-high-low-v-dd}.}
Recall that $\th_1= \frac{2-(d-2)p}{10}$ in  \eqref{defn:fra-energy-dd-2-th}. By H\"older's inequality,
\EQ{
\eqref{eq:fra-energy-bound-wv-high-low-v-dd} 
\lsm & \sum_{N_1\ll N} N_1^{0+} \norm{\nabla \V_N \V_{N_1}}_{L_{t,x}^2}^{\th_1} \norm{\nabla \V_N}_{L_t^2 L_x^{\frac{2d}{d-2}}}^{1-\th_1} \normb{\norm{\nabla\W}_{L_x^2} \norm{\V_{N_1}}_{L_x^{\frac{d(p-\th_1)}{1-\th_1}}}^{p-\th_1}}_{L_t^2}.
}
By the bilinear Strichartz estimate in Lemma \ref{lem:bilinearstrichartz} and Lemma \ref{lem:fra-initial-data}, for $N_1\ll N$,
\EQn{\label{eq:fra-energy-bound-wv-high-low-v-dd-bi}
\norm{\nabla \V_N \V_{N_1}}_{L_{t,x}^2} \lsm & N^{-\frac12} N_1^{\frac{d-1}{2}} \norm{\nabla P_N\V_+}_{L_x^2} \norm{P_{N_1}\V_+}_{L_x^2} \\
\lsm & N^{\frac12-s} N_1^{\frac{d-1}{2}-s} \norm{\jb{\nabla}^s P_N\V_+}_{L_x^2} \norm{\jb{\nabla}^sP_{N_1}\V_+}_{L_x^2} \\
\lsm & N^{\frac12-s} N_1^{\frac{d-1}{2}-s}\de_0^2.
}
Applying this inequality, and by Sobolev's inequality, Lemma \ref{lem:fra-linear}, and \eqref{eq:fra-energy-bound-wv-high-high-v-dd-1},
\EQ{
\eqref{eq:fra-energy-bound-wv-high-low-v-dd} 
\lsm & \de_0^{2\th_1} \sum_{N_1\ll N} N^{\th_1(\frac12-s)} N_1^{\th_1(\frac{d-1}{2}-s)+} \norm{\nabla \V_N}_{L_t^2 L_x^{\frac{2d}{d-2}}}^{1-\th_1} \normb{\norm{\nabla\W}_{L_x^2} \norm{\V_{N_1}}_{L_x^{\frac{d(p-\th_1)}{1-\th_1}}}^{p-\th_1}}_{L_t^2} \\
\lsm & \de_0^{2} \sum_{N_1\ll N} N^{-\frac12\th_1 +1 -s} N_1^{\th_1(\frac{d-1}{2}-s)+} \normb{\norm{\nabla\W}_{L_x^2} \normb{\jb{\nabla}^{\frac{dp-2-(d-2)\th_1}{2(p-\th_1)}}\V_{N_1}}_{L_x^{2}}^{p-\th_1}}_{L_t^2} \\
\lsm & \de_0^{2} \sum_{N_1\ll N} N^{-\frac12\th_1 +1 -s} N_1^{\th_1(\frac{d-1}{2}-s) + \frac{dp-2-(d-2)\th_1}{2}-(p-\th_1)s+} \norm{\nabla\W}_{L_{t,x}^2} \\
\lsm & \de_0^{2} \sum_{N_1\ll N} N^{-\frac12\th_1 +1 -s} N_1^{\frac{dp-2}{2}-ps + \frac12\th_1+} N_0^{1-s}.
}
Recall that $\th_1=\frac{2-(d-2)p}{10}$ and the choice of $s_0$ in Definition \ref{defn:fra-scattering-s0}, we have
\EQ{
1-s<\frac{2-(d-2)p}{20},
}
then we have
\EQ{
-\frac12\th_1 +1 -s < 0,
}
and 
\EQ{
\frac{dp-2}{2}-ps + \frac12\th_1 = & -\frac{2-(d-2)p}{2} + p(1-s)+ \frac12\th_1 \\
\loe & -\frac{2-(d-2)p}{2} + \frac15(2-(d-2)p) + \frac{2-(d-2)p}{20} \\
= & -\frac{2-(d-2)p}{4}<0.
}
Therefore, we can obtain
\EQn{\label{esti:fra-energy-bound-wv-high-low-v-dd} 
\eqref{eq:fra-energy-bound-wv-high-low-v-dd} 
\lsm \de_0^{2}  N_0^{1-s}.
}

$\bullet$ \textbf{Estimate of \eqref{eq:fra-energy-bound-wv-high-low-w-dd}.}
Recall that $\th_2 = \frac{(dp-2)^2}{2(d+2)}$. By H\"older's, Sobolev's inequalities, and \eqref{eq:fra-energy-bound-bilinear},
\EQ{
\eqref{eq:fra-energy-bound-wv-high-low-w-dd} \lsm & \sum_{N_1\ll N} N_1^{0+} \norm{\nabla\V_N\W_{N_1}}_{L_{t,x}^2}^{\th_2} \norm{\V_N}_{L_t^2 L_x^{\frac{2d}{d-2}}}^{1-\th_2} \normb{\norm{\nabla\W}_{L_x^2} \norm{\W_{N_1}}_{L_x^{\frac{d(p-\th_2)}{1-\th_2}}}^{p-\th_2}}_{L_t^2} \\
\lsm & \de_0 \sum_{N_1\ll N} N^{-\frac12\th_2+1-s} N_0^{\frac{2(p+1)\th_2}{p+2}(1-s)} \normb{\norm{\nabla\W}_{L_x^2} \normb{\jb{\nabla}^{\frac{(2dp-p+2d-2)\th_2}{2(p+2)(p-\th_2)}+}\W_{N_1}}_{L_x^{\frac{d(p-\th_2)}{1-\th_2}}}^{p-\th_2}}_{L_t^2}.
}
By Bernstein's inequality,  interpolation, and \eqref{eq:fra-mass-bound-w},
\EQ{
\normb{\jb{\nabla}^{\frac{(2dp-p+2d-2)\th_2}{2(p+2)(p-\th_2)}+}\W_{N_1}}_{L_x^{\frac{d(p-\th_2)}{1-\th_2}}} \lsm & N_1^{0-}\normb{\jb{\nabla}^{\frac{d\th_2}{p-\th_2}}\W_{N_1}}_{L_x^{\frac{d(p-\th_2)}{1-\th_2}}} \\
\lsm & N_1^{0-} \normb{\jb{\nabla}^{\frac{d\th_2}{p-\th_2}+\frac{dp-2-(d-2)\th_2}{2(p-\th_2)}}\W_{N_1}}_{L_x^{2}} \\
\lsm & N_1^{0-} \norm{\nabla\W}_{L_x^2}^{\frac{dp-2}{2(p-\th_2)} + \frac{d+2}{2(p-\th_2)}\th_2}.
}
Applying this inequality and \eqref{eq:fra-energy-bound-w},
\EQ{
\eqref{eq:fra-energy-bound-wv-high-low-w-dd} \lsm & \de_0 \sum_{N\in2^\N} N^{-\frac12\th_2+1-s} N_0^{\frac{2(p+1)\th_2}{p+2}(1-s)} \normb{\norm{\nabla\W}_{L_x^2}^{\frac{dp}{2} + \frac{d+2}{2}\th_2}}_{L_t^2} \\
\lsm & \de_0 \sum_{N\in2^\N} N^{-\frac12\th_2+1-s} N_0^{(\frac{2(p+1)\th_2}{p+2} + \frac{dp}{2} + \frac{d+2}{2}\th_2)(1-s)} \normb{t^{(\frac{dp}{4}-1)(\frac{dp}{2} + \frac{d+2}{2}\th_2)}}_{L_t^2} \\
\lsm & \de_0 N_0^{2(1-s) -\frac12\th_2+ ( \frac{dp-2}{2} + (\frac{2(p+1)}{p+2} + \frac{d+2}{2})\th_2)(1-s)} \normb{t^{(\frac{dp}{4}-1)(\frac{dp}{2} + \frac{d+2}{2}\th_2)}}_{L_t^2}.
}
Recall that $\th_2 = \frac{(dp-2)^2}{2(d+2)}$,
then
\EQ{
-\frac12<(\frac{dp}{4}-1)(\frac{dp}{2} + \frac{d+2}{2}\th_2)<0,
}
Consequently,
\EQ{
\normb{t^{(\frac{dp}{4}-1)(\frac{dp}{2} + \frac{d+2}{2}\th_2)}}_{L_t^2} \lsm 1.
}
Moreover, by the choice of $s_0$ in Definition \ref{defn:fra-scattering-s0}, we have
\EQ{
1-s<\frac{dp-2}{2(d+2)(dp-1)},
}
Note also that $\frac{2(p+1)}{p+2}<\frac{d+2}{2}$, then we have
\EQ{
&-\frac12\th_2+ ( \frac{dp-2}{2} + (\frac{2(p+1)}{p+2} + \frac{d+2}{2})\th_2)(1-s) \\
< & -\frac{(dp-2)^2}{4(d+2)}+ ( \frac{dp-2}{2} + (d+2)\frac{(dp-2)^2}{2(d+2)})(1-s) \\
< & -\frac{(dp-2)^2}{4(d+2)}+ \frac12(dp-2)(dp-1)(1-s) \\
<&0.
}
This gives
\EQ{
N_0^{2(1-s) -\frac12\th_2+ ( \frac{dp-2}{2} + (\frac{2(p+1)}{p+2} + \frac{d+2}{2})\th_2)(1-s)} \lsm N_0^{2(1-s)}.
}
Therefore, we have
\EQn{\label{esti:fra-energy-bound-wv-high-low-w-dd}
\eqref{eq:fra-energy-bound-wv-high-low-w-dd} \lsm & \de_0 N_0^{2(1-s)}.
}

$\bullet$ \textbf{Estimate of \eqref{eq:fra-energy-bound-vv-high-high-dd}.} By H\"older's inequality,
\EQ{
\eqref{eq:fra-energy-bound-vv-high-high-dd} 
\lsm & \sum_{N_1\loe N_2\lsm N} N^{0+} \norm{\nabla\V_{N_1}}_{L_t^2 L_x^{\frac{2d}{d-2}}} \norm{\nabla\V_{N_2}}_{L_t^{\frac{4}{dp-2}} L_x^{\frac{2d}{d+2-dp}}} \norm{\U_{N}}_{L_t^{\frac{4p}{4-dp}}L_x^2}^p \\
\lsm & \de_0^2 \sum_{N\in2^\N} N^{0+} N^{2(1-s)} \norm{\U_{N}}_{L_t^{\frac{4p}{4-dp}}L_x^2}^p \\
\lsm & \de_0^2 \normb{\jb{\nabla}^{\frac2p(1-s)+}\U}_{L_t^{\frac{4p}{4-dp}}L_x^2}^p.
}
By Lemma \ref{lem:fra-linear}, interpolation, and \eqref{eq:fra-energy-bound-w},
\EQ{
\normb{\jb{\nabla}^{\frac2p(1-s)+}\U}_{L_x^2} \lsm & \norm{\jb{\nabla}^s\V}_{L_x^2} + \normb{\jb{\nabla}^{\frac2p(1-s)+}\W}_{L_x^2} \\
\lsm & \de_0 + \norm{\nabla\W}_{L_x^2}^{\frac2p(1-s)+}.
}
By this inequality and \eqref{eq:fra-mass-bound-w},
\EQ{
\eqref{eq:fra-energy-bound-vv-high-high-dd} 
\lsm & \de_0^2 \normb{\jb{\nabla}^{\frac2p(1-s)+}\U}_{L_t^{\frac{4p}{4-dp}}L_x^2}^p \\
\lsm & \de_0^2 \normb{t^{(\frac{dp}{4}-1)\frac2p(1-s)}}_{L_t^{\frac{4p}{4-dp}}}^p N_0^{2(1-s)^2}.
}
Since $s>\frac12$,
\EQ{
(\frac{dp}{4}-1)\frac2p(1-s)\cdot\frac{4p}{4-dp}=-2(1-s)>-1,
}
which yields
\EQ{
\normb{t^{(\frac{dp}{4}-1)\frac2p(1-s)}}_{L_t^{\frac{4p}{4-dp}}} \lsm 1.
}
Consequently,
\EQn{\label{esti:fra-energy-bound-vv-high-high-dd}
\eqref{eq:fra-energy-bound-vv-high-high-dd} 
\lsm \de_0^2 N_0^{2(1-s)}.
}

$\bullet$ \textbf{Estimate of \eqref{eq:fra-energy-bound-vv-high-low-v-dd}.} Recall that $\th_1= \frac{2-(d-2)p}{10}$, then by H\"older's inequality, \eqref{eq:fra-energy-bound-wv-high-low-v-dd-bi}, Lemma \ref{lem:fra-linear}, and Sobolev's inequality,
\EQ{
\eqref{eq:fra-energy-bound-vv-high-low-v-dd} 
\lsm & \sum_{N_1\loe N_2}\sum_{N:N\ll N_2} N^{0+} \norm{\nabla \V_{N_2} \V_{N}}_{L_{t,x}^2}^{\th_1} \norm{\nabla \V_{N_2}}_{L_t^2 L_x^{\frac{2d}{d-2}}}^{1-\th_1} \normb{\norm{\nabla\V_{N_1}}_{L_x^2} \norm{\V_{N}}_{L_x^{\frac{d(p-\th_1)}{1-\th}}}^{p-\th_1}}_{L_t^2} \\
\lsm & \de_0^{2\th_1} \sum_{N_1\loe N_2}\sum_{N:N\ll N_2} N_2^{\th_1(\frac12-s)} N^{\th_1(\frac{d-1}{2}-s)+} \norm{\nabla \V_{N_2}}_{L_t^2 L_x^{\frac{2d}{d-2}}}^{1-\th_1} \normb{\norm{\nabla\V_{N_1}}_{L_x^2} \norm{\V_{N}}_{L_x^{\frac{d(p-\th_1)}{1-\th_1}}}^{p-\th_1}}_{L_t^2} \\
\lsm & \de_0^{2} \sum_{N_1\loe N_2}\sum_{N:N\ll N_2} N_2^{-\frac12\th_1 +1 -s} N^{\th_1(\frac{d-1}{2}-s)+}\norm{\nabla\V_{N_1}}_{L_t^\I L_x^2} \normb{\jb{\nabla}^{\frac{dp-2-(d-2)\th_1}{2(p-\th_1)}}\V_{N}}_{L_t^\I L_x^{2}}^{p-\th_1} \\
\lsm & \de_0^{2} \sum_{N_1\loe N_2}\sum_{N:N\ll N_2} N_2^{-\frac12\th_1 +1 -s} N^{\th_1(\frac{d-1}{2}-s) + \frac{dp-2-(d-2)\th_1}{2}-(p-\th_1)s+}N_1^{1-s} \\
\lsm & \de_0^{2} \sum_{N\ll N_2} N_2^{-\frac12\th_1 +2(1 -s)} N^{\frac{dp-2}{2}-ps + \frac12\th_1+}.
}
Recall the choice of $s_0$ in Definition \ref{defn:fra-scattering-s0}, we have
\EQ{
1-s<\frac{2-(d-2)p}{40},
}
then we have
\EQ{
	-\frac12\th_1 +2(1 -s) < 0,
}
and 
\EQ{
	\frac{dp-2}{2}-ps + \frac12\th_1 = & -\frac{2-(d-2)p}{2} + p(1-s)+ \frac12\th_1 \\
	\loe & -\frac{2-(d-2)p}{2} + \frac1{10}(2-(d-2)p) + \frac{2-(d-2)p}{20} \\
	= & -\frac{7(2-(d-2)p)}{20}<0.
}
Therefore, we can obtain
\EQn{\label{esti:fra-energy-bound-vv-high-low-v-dd}  
\eqref{eq:fra-energy-bound-vv-high-low-v-dd}  
	\lsm \de_0^{2}.
}

$\bullet$ \textbf{Estimate of \eqref{eq:fra-energy-bound-vv-high-low-w-dd}.} 
Recall that $\th_2 = \frac{(dp-2)^2}{2(d+2)}$, then by H\"older's inequality,
\EQ{
\eqref{eq:fra-energy-bound-vv-high-low-w-dd} \lsm & \sum_{N_1\loe N_2}\sum_{N:N\ll N_2} N^{0+} \norm{\nabla\V_{N_2}\W_{N}}_{L_{t,x}^2}^{\th_2} \norm{\V_{N_2}}_{L_t^2 L_x^{\frac{2d}{d-2}}}^{1-\th_2} \normb{\norm{\nabla\V_{N_1}}_{L_x^2} \norm{\W_{N}}_{L_x^{\frac{d(p-\th_2)}{1-\th_2}}}^{p-\th_2}}_{L_t^2} \\
\lsm & \de_0 \sum_{N \ll N_2} N_2^{-\frac12\th_2+2(1-s)} N_0^{\frac{2(p+1)\th_2}{p+2}(1-s)}  \normb{\jb{\nabla}^{\frac{(2dp-p+2d-2)\th_2}{2(p+2)(p-\th_2)}+}\W_{N_1}}_{L_t^{2(p-\th_2)}L_x^{\frac{d(p-\th_2)}{1-\th_2}}}^{p-\th_2}.
}
By Bernstein's inequality,  interpolation, and \eqref{eq:fra-mass-bound-w},
\EQ{
\normb{\jb{\nabla}^{\frac{(2dp-p+2d-2)\th_2}{2(p+2)(p-\th_2)}+}\W_{N}}_{L_x^{\frac{d(p-\th_2)}{1-\th_2}}} 
\lsm N^{0-} \norm{\nabla\W}_{L_x^2}^{\frac{dp-2}{2(p-\th_2)} + \frac{d+2}{2(p-\th_2)}\th_2}.
}
Therefore, by Lemma \ref{lem:fra-linear} and \eqref{eq:fra-energy-bound-w},
\EQ{
\eqref{eq:fra-energy-bound-vv-high-low-w-dd} \lsm & \de_0 \sum_{N_2\in2^\N} N_2^{-\frac12\th_2+2(1-s)} N_0^{\frac{2(p+1)\th_2}{p+2}(1-s)} \normb{\norm{\nabla\W}_{L_x^2}^{\frac{dp-2}{2} + \frac{d+2}{2}\th_2}}_{L_t^2} \\
\lsm & \de_0 \sum_{N_2\in2^\N} N_2^{-\frac12\th_2+2(1-s)} N_0^{(\frac{2(p+1)\th_2}{p+2} + \frac{dp-2}{2} + \frac{d+2}{2}\th_2)(1-s)} \normb{t^{(\frac{dp}{4}-1)(\frac{dp-2}{2} + \frac{d+2}{2}\th_2)}}_{L_t^2} \\
\lsm & \de_0 N_0^{2(1-s) -\frac12\th_2+ ( \frac{dp-2}{2} + (\frac{2(p+1)}{p+2} + \frac{d+2}{2})\th_2)(1-s)} \normb{t^{(\frac{dp}{4}-1)(\frac{dp-2}{2} + \frac{d+2}{2}\th_2)}}_{L_t^2}.
}
Recall the definition of $\th_2= \frac{(dp-2)^2}{2(d+2)}$, then
\EQ{
-\frac12<(\frac{dp}{4}-1)(\frac{dp-2}{2} + \frac{d+2}{2}\th_2)<0,
}
which gives
\EQ{
\normb{t^{(\frac{dp}{4}-1)(\frac{dp-2}{2} + \frac{d+2}{2}\th_2)}}_{L_t^2} \lsm 1.
}
Moreover, by the choice of $s_0$ in Definition \ref{defn:fra-scattering-s0}, we have
\EQ{
	1-s<\frac{dp-2}{2(d+2)(dp-1)},
}
and by $\frac{2(p+1)}{p+2}<\frac{d+2}{2}$ and $\th_2= \frac{(dp-2)^2}{2(d+2)}$,
\EQ{
& -\frac12\th_2+ ( \frac{dp-2}{2} + (\frac{2(p+1)}{p+2} + \frac{d+2}{2})\th_2)(1-s) \\
< & -\frac12\th_2+ ( \frac{dp-2}{2} + (d+2)\th_2)(1-s)\\
< & -\frac{(dp-2)^2}{4(d+2)}+ \frac{(dp-2)(dp-1)}{2}(1-s) \\
<& 0.
}
Therefore, we have
\EQn{\label{esti:fra-energy-bound-vv-high-low-w-dd}
\eqref{eq:fra-energy-bound-vv-high-low-w-dd} \lsm & \de_0 N_0^{2(1-s)}.
}

Finally, Lemma \ref{lem:fra-energy-dd} in 2D and higher dimensional cases follows from \eqref{esti:fra-energy-bound-wv-high-high-v-dd}, \eqref{esti:fra-energy-bound-wv-high-high-w-dd}, \eqref{esti:fra-energy-bound-wv-high-low-v-dd}, \eqref{esti:fra-energy-bound-wv-high-low-w-dd}, \eqref{esti:fra-energy-bound-vv-high-high-dd}, \eqref{esti:fra-energy-bound-vv-high-low-v-dd}, and \eqref{esti:fra-energy-bound-vv-high-low-w-dd}.
\end{proof}

\subsection{Global spacetime estimate}
Define $I_k:=[2^{k-1},2^k]$ with $k\in\Z$ such that $2^k\loe T_0$. In order to consider the case when $p\goe\frac{2}{d-2}$, we will exploit some spacetime estimate $\norm{\U}_{L_t^q L_x^r}$ for $r>\frac{2d}{d-2}$. 
\begin{lem}[Sobolev inequality implies spacetime estimate]\label{lem:fra-energy-global-spacetime-estimate}
Let $5\loe d\loe 11$, $\frac{2}{d-2}\loe p <\min\fbrk{\frac4d,\frac{2}{d-4}}$, $T_0>0$, and $k\in\Z$ such that $2^k\loe T_0$. Suppose that $s_0$ is defined in Definition \ref{defn:fra-scattering-s0} ,and the assumptions in Theorem \ref{thm:frac-weighted-dd} and the bootstrap hypothesis \eqref{eq:fra-energy-bootstrap-hypothesis-dd} hold. 
If there exists some $0\loe l<s$ and $2<r_0<\frac{2d}{d-2}$ such that the Sobolev's inequality holds: for any $f\in W_x^{l,r_0}$,
\EQn{\label{eq:fra-energy-global-spacetime-estimate-sobolev}
\norm{f}_{L_x^{dp}(\R^d)} \lsm \norm{\jb{\nabla}^lf}_{L_x^{r_0}(\R^d)},
}
then, for any $L^2$-admissible pair $(q,r)$,
\EQ{
\normb{\jb{\nabla}^{l}\U}_{L_t^qL_x^r(I_k\times\R^d)} \lsm & 2^{(\frac{dp}{4}-1)lk} N_0^{l(1-s)} +  2^{(\frac{2+l}{1-p}\cdot(\frac{dp}{4}-1) + (\frac1{2p}-\frac{1}{q_0})\cdot \frac{p}{1-p})k} N_0^{\frac{l}{1-p}(1-s)},
}
where $q_0$ denotes the exponent such that $(q_0,r_0)$ is $L^2$-admissible.
\end{lem}
\begin{proof}
In this lemma, we restrict the spacetime variable $(t,x)$ on $I_k\times\R^d$. By the equation and Strichartz estimate,
\EQ{
\normb{\jb{\nabla}^l\U}_{L_t^qL_x^r} \lsm \norm{\U(2^{k-1},\cdot)}_{H_x^l} + \normb{t^{\frac{dp}{2}-2}\jb{\nabla}^l(|\U|^p\U)}_{L_t^2 L_x^{\frac{2d}{d+2}}}.
}
By the fractional chain rule in Lemma \ref{lem:frac-chain},
\EQ{
\normb{\jb{\nabla}^l(|\U|^p\U)(t)}_{L_x^{\frac{2d}{d+2}}} \lsm \normb{\jb{\nabla}^l\U(t)}_{L_x^2} \norm{\U(t)}_{L_x^{dp}}^p.
}
By Lemma \ref{lem:fra-linear}, interpolation, \eqref{eq:fra-energy-bound-w}, and \eqref{eq:fra-mass-bound-w}, for any $t\in[2^{k-1},2^k]$,
\EQ{
\normb{\jb{\nabla}^l\U(t)}_{L_x^2} \lsm & \normb{\jb{\nabla}^l\V(t)}_{L_x^2} + \normb{\jb{\nabla}^l\W(t)}_{L_x^2}\\
\lsm & \de_0 + 2^{(\frac{dp}{4}-1)lk}N_0^{l(1-s)} \\
\lsm & 2^{(\frac{dp}{4}-1)lk}N_0^{l(1-s)}.
}
Particularly, we have that $\norm{\U(2^{k-1},\cdot)}_{H_x^l} \lsm 2^{(\frac{dp}{4}-1)lk}N_0^{l(1-s)}$. Therefore, applying H\"older's inequality,
\EQ{
\normb{\jb{\nabla}^l\U}_{L_t^qL_x^r} 
\lsm & 2^{(\frac{dp}{4}-1)lk} N_0^{l(1-s)} + 2^{(\frac{dp}{4}-1)(2+l)k} N_0^{l(1-s)} \cdot \norm{\jb{\nabla}^l\U}_{L_t^{2p} L_x^{r_0}}^p  \\
\lsm & 2^{(\frac{dp}{4}-1)lk} N_0^{l(1-s)} + 2^{\brkb{(\frac{dp}{4}-1)(2+l) + (\frac{1}{2p}-\frac{1}{q_0})p}k} N_0^{l(1-s)} \cdot \norm{\jb{\nabla}^l\U}_{L_t^{q_0} L_x^{r_0}}^p.
}
By Lemma \ref{lem:fra-linear},
\EQ{
\norm{\jb{\nabla}^l\U}_{L_x^{q_0}L_x^{r_0}} \lsm &  \norm{\jb{\nabla}^l\V}_{L_x^{q_0}L_x^{r_0}} + \norm{\jb{\nabla}^l\W}_{L_x^{q_0}L_x^{r_0}} \\
\lsm & \de_0 + \norm{\jb{\nabla}^l\W}_{L_x^{q_0}L_x^{r_0}},
}
Then, we get that 
\EQ{
\normb{\jb{\nabla}^l\U}_{L_t^qL_x^r} 
\lsm & 2^{(\frac{dp}{4}-1)lk} N_0^{l(1-s)}  \\
& + 2^{\brkb{(\frac{dp}{4}-1)(2+l) + (\frac{1}{2p}-\frac{1}{q_0})p}k} N_0^{l(1-s)} \cdot \brkb{\de_0^p + \norm{\jb{\nabla}^s\W}_{L_x^{q_0}L_x^{r_0}}^p} .
}
By Young's inequality and noting that $p<\frac4d<1$,
\EQ{
\normb{\jb{\nabla}^l\U}_{L_t^qL_x^r} 
\loe &  \frac12\norm{\jb{\nabla}^s\W}_{L_t^{q_0}L_x^{r_0}} + C\de_0 + C 2^{(\frac{dp}{4}-1)lk} N_0^{l(1-s)} \\
& + C 2^{\brkb{\frac{2+l}{1-p}\cdot(\frac{dp}{4}-1) + (\frac{1}{2p}-\frac{1}{q_0})\cdot \frac{p}{1-p}}k} N_0^{\frac{l}{1-p}(1-s)}.
}
Taking $(q,r)=(q_0,r_0)$ first, we obtain the $L_x^{q_0}L_x^{r_0}$-estimate, and then inserting this back again, we can  finish the proof of the lemma.
\end{proof}

\begin{lem}\label{lem:fra-energy-global-spacetime-estimate-bilinear}
Let the assumptions in Lemma \ref{lem:fra-energy-global-spacetime-estimate} hold, and $l,q_0,r_0$ be defined in Lemma \ref{lem:fra-energy-global-spacetime-estimate}. Then, for $N,N_1\in2^\N$ with $N_1\ll N$,
\EQ{
\norm{\nabla \V_N \U_{N_1}}_{L_{t,x}^2(I_k\times\R^d)} \lsm & \de_0 N_1N^{-1/2+(1-s)} \cdot\Big(2^{\brkb{(\frac{dp}{4}-1)(1+lp) +\frac{1}{2}-\frac{p}{q_0}}k} N_0^{l(1-s)} \\
& + 2^{\brkb{(1+\frac{(2+l)p}{1-p})\cdot(\frac{dp}{4}-1) + (\frac{1}{2p}-\frac{1}{q_0})\cdot \frac{p}{1-p}}k} N_0^{\frac{l}{1-p}(1-s)}\Big).
}
\end{lem}
\begin{proof}
In this lemma, we restrict the spacetime variable $(t,x)$ on $I_k\times\R^d$. By the bilinear Strichartz estimate in Lemma \ref{lem:bilinearstrichartz},
\EQ{
\norm{\nabla \V_N \U_{N_1}}_{L_{t,x}^2} \lsm & \frac{N_1}{N^{1/2}} \norm{NP_N\V(2^{k-1})}_{L_x^2}\\
& \cdot \brkb{\norm{P_N\U(2^{k-1})}_{L_x^2}+\normb{t^{\frac{dp}{4}-1}P_N(|\U|^{p}\U)}_{L_t^2 L_x^{\frac{2d}{d+2}}}} \\
\lsm & N_1N^{-1/2+(1-s)} \norm{N^sP_N\V(2^{k-1})}_{L_x^2}\\
& \cdot\brkb{\norm{P_N\U(2^{k-1})}_{L_x^2}+2^{(\frac{dp}{4}-1)k}\norm{\U}_{L_t^\I L_x^2} \norm{\U}_{L_t^{2p} L_x^{dp}}^p}.
}
Note that by Lemma \ref{lem:fra-linear},
\EQ{
\norm{N^sP_N\V(2^{k-1})}_{L_x^2} \lsm \de_0,
}
and by Lemma \ref{lem:fra-initial-data} and \eqref{eq:fra-mass-bound-w},
\EQ{
\norm{P_N\U(2^{k-1})}_{L_x^2} + \norm{\U}_{L_t^\I L_x^2} \lsm 1.
}
By the Sobolev's and H\"older's inequalities,
\EQ{
\norm{\U}_{L_t^{2p} L_x^{dp}} \lsm & 2^{(\frac{1}{2p}-\frac{1}{q_0})k} \norm{\jb{\nabla}^l\U}_{L_t^{q_0} L_x^{r_0}} \\
\lsm & 2^{\brkb{(\frac{dp}{4}-1)l +\frac{1}{2p}-\frac{1}{q_0}}k} N_0^{l(1-s)} + 2^{\brkb{\frac{2+l}{1-p}\cdot(\frac{dp}{4}-1) + (\frac{1}{2p}-\frac{1}{q_0})\cdot \frac{1}{1-p}}k} N_0^{\frac{l}{1-p}(1-s)}.
}
This finishes the proof.
\end{proof}

We will apply Lemmas \ref{lem:fra-energy-global-spacetime-estimate} and \ref{lem:fra-energy-global-spacetime-estimate-bilinear} in two different settings for $l,q_0,r_0$. Note that our assumptions on $(d,p)$ on Lemma \ref{lem:fra-energy-global-spacetime-estimate} can be rephrased as
\EQn{\label{eq:fra-energy-bound-wv-high-high-5d-case}
\frac{2}{d-2}\loe p<\frac{4}{d}\text{, when } 5\loe d\loe 8;\quad \frac{2}{d-2}\loe p<\frac{2}{d-4}\text{, when }9\loe d\loe 11.
}
Then, we split the condition \eqref{eq:fra-energy-bound-wv-high-high-5d-case} into two subcases:
\begin{gather}
\label{eq:fra-energy-bound-wv-high-high-5d-case-1}\tag{Case A}
\frac{2}{d-2}<p<\frac4d\text{, when }d=5,6;\quad \frac12<p<\frac47\text{, when }d=7.
\end{gather}
and
\begin{gather}
\label{eq:fra-energy-bound-wv-high-high-5d-case-2} \tag{Case B}
\aligned
p=\frac{2}{d-2}\text{, when }d=5,6;\quad \frac{2}{d-2}\loe p\loe \frac12\text{, when }d=7;\\
\frac{2}{d-2}\loe p <\frac{2}{d-4}\text{, when }8\loe d\loe 11.
\endaligned
\end{gather}
We gather two inequalities that will be used frequently below:
\begin{lem}
[Two Sobolev's inequalities]
\label{lem:fra-energy-global-spacetime-estimate-sobolev}
\begin{enumerate}
\item 
Let the condition \eqref{eq:fra-energy-bound-wv-high-high-5d-case-1} holds. Denote that $s_1,q_1,r_1$ by
\EQ{
s_1:=\frac{1-p}{p}\text{, }\frac{1}{q_1}:=\frac d4-\frac{s_1}2-\frac{1}{2p} +50d\th_1\text{,  }\frac{1}{r_1}:=\frac{s_1}{d}+\frac{1}{dp}-100\th_1,
}
then $(q_1,r_1)$ is $L_x^2$-admissible, and for any $f\in W_x^{s_1-100\th_1,r_1}$,
\EQ{
\norm{f}_{L_x^{dp}(\R^d)} \lsm \norm{\jb{\nabla}^{s_1-100\th_1}f}_{L_x^{r_1}(\R^d)}.
}
\item Let the condition \eqref{eq:fra-energy-bound-wv-high-high-5d-case-2} holds. Denote $q_2,r_2$ by
\EQ{
\frac{1}{q_2}:=\frac{d}{4}-\frac{s}{2}-\frac{1}{2p}\text{, and }\frac{1}{r_2}:= \frac{s}{d} + \frac{1}{dp},
}
then $(q_2,r_2)$ is $L_x^2$-admissible, and for any $f\in W^{s,r_2}$,
\EQ{
\norm{f}_{L_x^{dp}(\R^d)} \lsm \norm{\jb{\nabla}^{s}f}_{L_x^{r_2}(\R^d)}.
}
\end{enumerate}
\end{lem}
These two inequalities follow directly from the Sobolev inequality. In fact, if \eqref{eq:fra-energy-bound-wv-high-high-5d-case-1} holds, then by the choice of $s_0$ and $\th_1$ in Definition \ref{defn:fra-scattering-s0}, we have that
\EQ{
s_1-100\th_1\le s_0\text{, }-\frac1p\le s_1-100\th_1-\frac{d}{r_1}\text{, and }r_1\le \frac{2d}{d-2}
} 
Moreover, if \eqref{eq:fra-energy-bound-wv-high-high-5d-case-2} holds, we also have that 
\EQ{
-\frac1p\le s-\frac{d}{r_2}\text{ and }r_2\le \frac{2d}{d-2}.
}

Combining Lemmas \ref{lem:fra-energy-global-spacetime-estimate}, \ref{lem:fra-energy-global-spacetime-estimate-bilinear}, and \ref{lem:fra-energy-global-spacetime-estimate-sobolev}, we get the spacetime estimates and the bilinear Strichartz estimates that will be used in the energy estimate.
\begin{cor}
\label{cor:fra-energy-global-spacetime-estimate}
Let $5\loe d\loe 11$, $\frac{2}{d-2}\loe p <\min\fbrk{\frac4d,\frac{2}{d-4}}$, $T_0>0$, and $k\in\Z$ such that $2^k\loe T_0$. Suppose that $s_0$ is defined in Definition \ref{defn:fra-scattering-s0} ,and the assumptions in Theorem \ref{thm:frac-weighted-dd} and the bootstrap hypothesis \eqref{eq:fra-energy-bootstrap-hypothesis-dd} hold. Denote that $s_1,q_1,r_1,q_2,r_2$ as in Lemma \ref{lem:fra-energy-global-spacetime-estimate-sobolev}. Let also $(q,r)$ be any $L_x^2$-admissible pair.
\begin{enumerate}
\item 
If \eqref{eq:fra-energy-bound-wv-high-high-5d-case-1} holds, then
\EQ{
\normb{\jb{\nabla}^{s_1-100\th_1}\U}_{L_t^qL_x^r(I_k\times\R^d)} \lsm & 2^{(\frac{dp}{4}-1)(s_1-100\th_1)k} N_0^{(s_1-100\th_1)(1-s)} \\
& +  2^{(\frac{2+s_1-100\th_1}{1-p}\cdot(\frac{dp}{4}-1) + (\frac1{2p}-\frac{1}{q_1})\cdot \frac{p}{1-p})k} N_0^{\frac{s_1-100\th_1}{1-p}(1-s)},
}
and
\EQ{
\norm{\nabla \V_N \U_{N_1}}_{L_{t,x}^2(I_k\times\R^d)} \lsm & \de_0 N_1N^{-1/2+(1-s)}  \\
& \cdot \Big(2^{\brkb{(\frac{dp}{4}-1)(1+(s_1-100\th_1)p) +\frac{1}{2}-\frac{p}{q_0}}k} N_0^{(s_1-100\th_1)(1-s)} \\
& \quad+ 2^{\brkb{(1+\frac{(2+s_1-100\th_1)p}{1-p})\cdot(\frac{dp}{4}-1) + (\frac{1}{2p}-\frac{1}{q_1})\cdot \frac{p}{1-p}}k} N_0^{\frac{s_1-100\th_1}{1-p}(1-s)}\Big).
}
\item
If \eqref{eq:fra-energy-bound-wv-high-high-5d-case-2} holds, then
\EQ{
\normb{\jb{\nabla}^{s}\U}_{L_t^qL_x^r(I_k\times\R^d)} \lsm & 2^{(\frac{dp}{4}-1)sk} N_0^{s(1-s)} +  2^{(\frac{2+s}{1-p}\cdot(\frac{dp}{4}-1) + (\frac1{2p}-\frac{1}{q_2})\cdot \frac{p}{1-p})k} N_0^{\frac{s}{1-p}(1-s)},
}
and
\EQ{
\norm{\nabla \V_N \U_{N_1}}_{L_{t,x}^2(I_k\times\R^d)} \lsm & \de_0 N_1N^{-1/2+(1-s)} \cdot\Big(2^{\brkb{(\frac{dp}{4}-1)(1+sp) +\frac{1}{2}-\frac{p}{q_0}}k} N_0^{s(1-s)} \\
& + 2^{\brkb{(1+\frac{(2+s)p}{1-p})\cdot(\frac{dp}{4}-1) + (\frac{1}{2p}-\frac{1}{q_2})\cdot \frac{p}{1-p}}k} N_0^{\frac{s}{1-p}(1-s)}\Big).
}
\end{enumerate}
\end{cor}

\subsection{Energy estimate when $p\goe \frac{2}{d-2}$}
\begin{proof}[Proof of Lemma \ref{lem:fra-energy-5d}]
We make the following decomposition:
\EQnnsub{
|\int_{0}^{T_0}\int_{\R^d}|\U|^p\U\wb{\V_t}\dd x\dd t| \lsm & \sum_{N\lsm N_1} N_1^{0+} \int_{0}^{T_0}\int_{\R^d} |\nabla \W| |\nabla \V_N||\U_{ N_1}|^p \dd x \dd t \label{eq:fra-energy-bound-wv-high-high-5d}\\
& + \sum_{N_1\ll N} N_1^{0+} \int_{0}^{T_0}\int_{\R^d} |\nabla \W| |\nabla \V_N||\U_{ N_1}|^p \dd x \dd t \label{eq:fra-energy-bound-wv-high-low-5d}\\
& + \sum_{N_1\loe N_2}\int_{0}^{T_0} \int_{\R^d} |\nabla\V_{N_1}||\nabla\V_{N_2}||\U_{\gsm N_2}|^p\dd x\dd t \label{eq:fra-energy-bound-vv-high-high-5d}\\
& + \sum_{N_1\loe N_2}\int_{0}^{T_0} \int_{\R^d} |\nabla\V_{N_1}||\nabla\V_{N_2}||\U_{\ll N_2}|^p\dd x\dd t \label{eq:fra-energy-bound-vv-high-low-5d}. 
}
Then, we will deal with \eqref{eq:fra-energy-bound-wv-high-high-5d}--\eqref{eq:fra-energy-bound-vv-high-low-5d} one by one.

$\bullet$ \textbf{Estimate of \eqref{eq:fra-energy-bound-wv-high-high-5d}}.
We first make a dyadic decomposition in $t$,
\EQ{
\eqref{eq:fra-energy-bound-wv-high-high-5d} \lsm \sum_{N\lsm N_1} \sum_{k\in\Z:2^k\loe T_0} N_1^{0+} \int_{2^{k-1}}^{2^k}\int_{\R^d} |\nabla \W| |\nabla \V_N||\U_{ N_1}|^p \dd x \dd t.
}
By H\"older's inequality, Lemma \ref{lem:fra-linear}, and \eqref{eq:fra-energy-bound-w},
\EQ{
& N_1^{0+} \int_{2^{k-1}}^{2^k}\int_{\R^d} |\nabla \W| |\nabla \V_N||\U_{ N_1}|^p \dd x \dd t \\
\lsm & N^{1-s}N_1^{-\frac12\th_1} \norm{\nabla\W}_{L_{t\in I_k}^\I L_x^2}  \norm{\jb{\nabla}^{s}\V_{N}}_{L_{t\in I_k}^2L_x^{\frac{2d}{d-2}}} \normb{N_1^{\frac{\th_1}{2p}+}\U_{N_1}}_{L_{t\in I_k}^{2p}L_x^{dp}}^{p} \\
\lsm & \de_0 N^{1-s}N_1^{-\frac12\th_1} 2^{(\frac{dp}{4}-1)k} N_0^{1-s}   \normb{N_1^{\frac{\th_1}{2p}+}\U_{N_1}}_{L_{t\in I_k}^{2p}L_x^{dp}}^{p},
}
and also by the choice of $\th_1$,
\EQ{
	\sum_{N\lsm N_1}N^{1-s}N_1^{-\frac12\th_1} \lsm \sum_{N\in2^\N}N^{-\frac12\th_1 + 1-s} \lsm 1.
}
Therefore, we have
\EQ{
\eqref{eq:fra-energy-bound-wv-high-high-5d} \lsm \de_0 N_0^{1-s}  \sum_{k\in\Z:2^k\loe T_0} 2^{(\frac{dp}{4}-1)k}  \sup_{N_1:N_1\gsm N_0} \normb{N_1^{\frac{\th_1}{2p}+}\U_{N_1}}_{L_{t\in I_k}^{2p}L_x^{dp}}^{p},
}
thereby it suffices to prove that
\EQn{\label{eq:fra-energy-bound-wv-high-high-5d-main}
\sum_{k\in\Z:2^k\loe T_0} 2^{(\frac{dp}{4}-1)k}  \sup_{N_1:N_1\gsm N_0} \normb{N_1^{\frac{\th_1}{2p}+}\U_{N_1}}_{L_{t\in I_k}^{2p}L_x^{dp}}^{p} \lsm N_0^{1-s}.
}
In the following, we will prove this in two subcases \eqref{eq:fra-energy-bound-wv-high-high-5d-case-1} and \eqref{eq:fra-energy-bound-wv-high-high-5d-case-2}, using different estimates.

First, we consider the case when \eqref{eq:fra-energy-bound-wv-high-high-5d-case-1} holds, namely
\EQ{
\frac{2}{d-2}<p<\frac4d\text{, when }d=5,6;\quad \frac12<p<\frac47\text{, when }d=7.
}
Recall the definition of $s_1,q_1,r_1$ in Lemma \ref{lem:fra-energy-global-spacetime-estimate-sobolev}: $s_1=\frac{1-p}{p}$, $\frac{1}{q_1}=\frac d4-\frac{s_1}2-\frac{1}{2p}+50d\th_1$, and $\frac{1}{r_1}:=\frac{s_1}{d}+\frac{1}{dp}-100\th_1$. Note that
\EQ{
\frac{\th_1}{2p} + d\cdot(\frac{1}{r_1}-\frac{1}{dp}) < s_1-100\th_1, 
}
then by Bernstein's inequality in $x$, H\"older's in $t$, Corollary \ref{cor:fra-energy-global-spacetime-estimate}, $s_1<\frac{s_1p}{1-p}$, and $-100(\frac{dp}{4}-1)>0$,
\EQ{
\normb{N_1^{\frac{\th_1}{2p}+}\U_{N_1}}_{L_{t\in I_k}^{2p}L_x^{dp}} \lsm & 2^{(\frac{1}{2p}-\frac{1}{q_1})k} \normb{N_1^{s_1-100\th_1}\U_{N_1}}_{L_{t\in I_k}^{q_1}L_x^{r_1}} \\
\lsm & 2^{\brkb{(\frac{dp}{4}-1)(s_1-100\th_1) + \frac{1}{2p}-\frac{1}{q_1}}k} N_0^{(s_1-100\th_1)(1-s)} \\
& +  2^{(\frac{2+s_1-100\th_1}{1-p}\cdot(\frac{dp}{4}-1) + (\frac1{2p}-\frac{1}{q_1})\cdot \frac{1}{1-p})k} N_0^{\frac{s_1-100\th_1}{1-p}(1-s)} \\
\lsm & N_0^{\frac{s_1-100\th_1}{1-p}(1-s)} \Big( 2^{\brkb{(\frac{dp}{4}-1)(s_1-100\th_1) + \frac{1}{2p}-\frac{1}{q_1}}k}  \\
& + 2^{(\frac{2+s_1}{1-p}\cdot(\frac{dp}{4}-1) + \frac{1}{1-p}\cdot(\frac1{2p}-\frac{1}{q_1}))k} \Big).
}
Therefore, by the definition of $s_1$, $q_1$,  $\frac{s_1p}{1-p}\le1$, $100-75dp>-200$, and $-\frac{50dp}{1-p}>-1000$,
\EQn{\label{eq:fra-energy-bound-wv-high-high-5d-main-1}
& \text{(LHS) of \eqref{eq:fra-energy-bound-wv-high-high-5d-main}} \\
\lsm &N_0^{( \frac{s_1p}{1-p} -10\th_1)(1-s)} \sum_{k\in\Z:2^k\loe T_0} \Big( 2^{\brkb{(\frac{dp}{4}-1)(1+p(s_1-100\th_1)) + \frac{1}{2}-\frac{p}{q_1}}k} \\
& \quad\qquad\qquad\qquad\qquad\qquad + 2^{((1+\frac{(2+s_1)p}{1-p})\cdot(\frac{dp}{4}-1) + \frac{p}{1-p}\cdot(\frac1{2p}-\frac{1}{q_1}) )k}  \Big) \\
\lsm & N_0^{(1 -10\th_1)(1-s)} \sum_{k\in\Z:2^k\loe T_0} \Big( 2^{\brkb{-\frac{dp^2}{4}+\frac34(d+2)p-\frac52+(100-75dp)\th_1}k} 
 +  2^{(\frac{(d-2)p-2}{4(1-p)} - \frac{50dp}{1-p}\th_1)k} \Big)
 \\
\lsm & N_0^{(1 -10\th_1)(1-s)} \sum_{k\in\Z:2^k\loe T_0} \brkb{ 2^{\brkb{-\frac{dp^2}{4}+\frac34(d+2)p-\frac52-200\th_1}k} 
+  2^{(\frac{(d-2)p-2}{4(1-p)} - 1000\th_1)k} }.
}
For the polynomial $-\frac{dp^2}{4}+\frac34(d+2)p-\frac52$ under \eqref{eq:fra-energy-bound-wv-high-high-5d-case-1}, we can check that it is bounded below by $>\frac49$ if $d=5$, $>\frac18$ if $d=6$, and $>\frac{7}{16}$ if $d=7$, then by the choice of $\th_1$,
\EQn{\label{eq:fra-energy-bound-wv-high-high-5d-main-1-com-1}
\sum_{k\in\Z:2^k\loe T_0} 2^{\brkb{-\frac{dp^2}{4}+\frac34(d+2)p-\frac52-200\th_1}k} \lsm 1.
}
Note that for $d\goe 5$, $p<4/d<1$, then for $p>\max\fbrk{\frac{2}{d-2},\frac12}$, we have that $\frac{(d-2)p-2}{4(1-p)}>0$. 
By the choice of $\th_1$,
\EQn{\label{eq:fra-energy-bound-wv-high-high-5d-main-1-com-2}
	\sum_{k\in\Z:2^k\loe T_0} 2^{(\frac{(d-2)p-2}{4(1-p)} - 1000\th_1)k}\lsm 1.
}
Applying \eqref{eq:fra-energy-bound-wv-high-high-5d-main-1-com-1} and \eqref{eq:fra-energy-bound-wv-high-high-5d-main-1-com-2} to \eqref{eq:fra-energy-bound-wv-high-high-5d-main-1},
\EQn{\label{eq:fra-energy-bound-wv-high-high-5d-main-1-1}
\text{(LHS) of \eqref{eq:fra-energy-bound-wv-high-high-5d-main}} 
\lsm & N_0^{(1 -10\th_1)(1-s)}.
}
This finishes the proof of \eqref{eq:fra-energy-bound-wv-high-high-5d-main} when \eqref{eq:fra-energy-bound-wv-high-high-5d-case-1} holds.

Second, we consider the case when \eqref{eq:fra-energy-bound-wv-high-high-5d-case-2} holds, namely
\EQ{
p=\frac{2}{d-2}&\text{, when }d=5,6;\quad \frac{2}{d-2}\loe p\loe \frac12\text{, when }d=7;\\
& \frac{2}{d-2}\loe p <\frac{2}{d-4}\text{, when }8\loe d\loe 11.
}
Denote that
\EQn{\label{defn:wtr2al1}
\frac{1}{\wt r_2}:= \frac{s}{d} + \frac{2-\th_1}{2dp}-\text{, and }\al_1:=\frac d2-\frac{d}{\wt r_2}.
}
We can check that under the condition \eqref{eq:fra-energy-bound-wv-high-high-5d-case-2}, $2<\wt r_2<\frac{2d}{d-2}$ and $0<\al_1<1$. Note that
\EQ{
\frac{\th_1}{2p}+ d(\frac{1}{\wt r_2}-\frac{1}{dp})<s,
}
then by choosing suitable implicit parameters, and using Bernstein's inequality, interpolation,  H\"older's inequality, and Corollary \ref{cor:fra-energy-global-spacetime-estimate},
\EQ{
& \normb{N_1^{\frac{\th_1}{2p}+}\U_{N_1}}_{L_{t\in I_k}^{2p}L_x^{dp}} \\
\lsm &  \normb{N_1^{s}\U_{N_1}}_{L_{t\in I_k}^{2p}L_x^{\wt r_2}} \\
\lsm &  \normb{\normo{N_1^{s}\U_{N_1}}_{L_x^2}^{1-\al_1} \normo{N_1^{s}\U_{N_1}}_{L_x^{\frac{2d}{d-2}}}^{\al_1} }_{L_{t\in I_k}^{2p}} \\
\lsm &  2^{(1-\al_1)s(\frac{dp}{4}-1)k} N_0^{(1-\al_1)s(1-s)} \normb{N_1^{s}\U_{N_1}}_{L_{t\in I_k}^{2\al_1p}L_x^{\frac{2d}{d-2}}}^{\al_1} \\
\lsm &  2^{((1-\al_1)s(\frac{dp}{4}-1)+ \frac{1}{2p}-\frac{\al_1}{2})k} N_0^{(1-\al_1)s(1-s)} \normb{N_1^{s}\U_{N_1}}_{L_{t\in I_k}^{2}L_x^{\frac{2d}{d-2}}}^{\al_1} \\
\lsm & 2^{((1-\al_1)s(\frac{dp}{4}-1)+ \frac{1}{2p}-\frac{\al_1}{2})k} N_0^{(1-\al_1)s(1-s)} \\
& \cdot\brkb{2^{(\frac{dp}{4}-1)sk} N_0^{s(1-s)} +  2^{(\frac{2+s}{1-p}\cdot(\frac{dp}{4}-1) + (\frac1{2p}-\frac{1}{q_2})\cdot \frac{p}{1-p})k} N_0^{\frac{s}{1-p}(1-s)}}^{\al_1}\\
\lsm &  2^{\brkb{s(\frac{dp}{4}-1)+ \frac{1}{2p}-\frac{\al_1}{2}}k}N_0^{s(1-s)} \\
& + 2^{\brkb{(1-\al_1)s(\frac{dp}{4}-1)+ \frac{1}{2p}-\frac{\al_1}{2}+\frac{\al_1(2+s)}{1-p}\cdot(\frac{dp}{4}-1) + (\frac1{2p}-\frac{1}{q_2})\cdot \frac{\al_1p}{1-p}}k}N_0^{((1-\al_1)s+\frac{s\al_1}{1-p})(1-s)},
}
where $\frac{1}{q_2}=\frac{d}{4}-\frac{s}{2}-\frac{1}{2p}$ is defined in Lemma \ref{lem:fra-energy-global-spacetime-estimate-sobolev}. Then recall the definition of $\al_1=\frac d2-s-\frac1p+\frac{\th_1}{2p}+$ in \eqref{defn:wtr2al1}, 
\EQn{\label{eq:fra-energy-bound-wv-high-high-5d-main-2}
& \text{(LHS) of \eqref{eq:fra-energy-bound-wv-high-high-5d-main}} \\
\lsm &  \sum_{k\in\Z:2^k\loe T_0} \Big( 2^{\brkb{(1+sp)(\frac{dp}{4}-1)+ \frac{1}{2}-\frac{\al_1p}{2}}k}N_0^{sp(1-s)} \\
& + 2^{\brkb{(1+(1-\al_1)sp)(\frac{dp}{4}-1)+ \frac{1}{2}-\frac{\al_1p}{2}+\frac{\al_1p(2+s)}{1-p}\cdot(\frac{dp}{4}-1) + (\frac1{2p}-\frac{1}{q_2})\cdot \frac{\al_1p^2}{1-p}}k}N_0^{((1-\al_1)s+\frac{s\al_1}{1-p})p(1-s)} \Big) \\
\lsm & N_0^{((1-\al_1)s+\frac{s\al_1}{1-p})p(1-s)} \cdot \sum_{k\in\Z:2^k\loe T_0} \Big( 2^{\brkb{(1+p)(\frac{dp}{4}-1)+ \frac{1}{2}-\frac{\al_1p}{2}-10(1-s)}k} \\
& + 2^{\brkb{(1+(1-\al_1)p+ \frac{3\al_1p}{1-p})(\frac{dp}{4}-1)+ \frac{1}{2}-\frac{\al_1p}{2} + (\frac1{2p}-\frac{1}{q_2})\cdot \frac{\al_1p^2}{1-p}-10(1-s)}k} \Big) \\
\lsm & N_0^{\brkb{\frac{(d-4)p^2}{2(1-p)}+100\th_1+}(1-s)}\\
&\cdot \sum_{k\in\Z:2^k\loe T_0}2^{(-100(1-s)-100\th_1-)k} \Big( 2^{\frac14p(dp-2)k} 
+ 2^{\brkb{\frac{d(d-4)p^3+(d^2-4d+8)p^2+(8-6d)p+8}{8(1-p)}}k} \Big).
}
First, since $p\goe\frac{2}{d-2}>\frac2d$, by the choice of $s$ and $\th_1$,
\EQn{\label{eq:fra-energy-bound-wv-high-high-5d-main-2-com-1}
\sum_{k\in\Z:2^k\loe T_0} 2^{\brko{\frac14p(dp-2)-100(1-s)-100\th_1}k} \lsm 1.
}
Note that $p<4/d<1$ for $d\goe 5$. We can also check that for $5\loe d\loe11$ and $\frac{2}{d-2}\loe p<\min\fbrk{\frac{2}{d-2},\frac{4}{d}}$,
\EQ{
d(d-4)p^3+(d^2-4d+8)p^2+(8-6d)p+8>0,
}
then by the choice of $s$ and $\th_1$,
\EQn{\label{eq:fra-energy-bound-wv-high-high-5d-main-2-com-2}
\sum_{k\in\Z:2^k\loe T_0} 2^{(\frac{d(d-4)p^3+(d^2-4d+8)p^2+(8-6d)p+8}{8(1-p)}-100(1-s)-100\th_1)k} \lsm 1.
}
Note also that for $d=5,6$, and $p=\frac{2}{d-2}$, $\frac{(d-4)p^2}{2(1-p)}\loe \frac23<1$; for $d=7$ and $\frac{2}{5}\loe p<\frac12$, $\frac{(d-4)p^2}{2(1-p)} \loe\frac34<1$; and for $8\loe d \loe 10$ and $\frac{2}{d-2}\loe p<\frac{2}{d-4}$, $\frac{(d-4)p^2}{2(1-p)} < \frac{(d-4)(\frac{2}{d-4})^2}{2(1-\frac{2}{d-4})} \loe 1$. Therefore, under the condition \eqref{eq:fra-energy-bound-wv-high-high-5d-case-2}, we have that $\frac{(d-4)p^2}{2(1-p)}<1$. Consequently, by the choice of $\th_1$ and the implicit parameter, 
\EQn{\label{eq:fra-energy-bound-wv-high-high-5d-main-2-com-3}
N_0^{\brkb{\frac{(d-4)p^2}{2(1-p)}+100\th_1+}(1-s)} \lsm N_0^{1-s}.
}
Applying \eqref{eq:fra-energy-bound-wv-high-high-5d-main-2-com-1}, \eqref{eq:fra-energy-bound-wv-high-high-5d-main-2-com-2}, and \eqref{eq:fra-energy-bound-wv-high-high-5d-main-2-com-3} to \eqref{eq:fra-energy-bound-wv-high-high-5d-main-2},
\EQ{
\text{(LHS) of \eqref{eq:fra-energy-bound-wv-high-high-5d-main}} 
\lsm  N_0^{1-s}.
}
Thus, we finish the proof of \eqref{eq:fra-energy-bound-wv-high-high-5d-main} when \eqref{eq:fra-energy-bound-wv-high-high-5d-case-2} holds.

Now, noting that \eqref{eq:fra-energy-bound-wv-high-high-5d-case-1} and \eqref{eq:fra-energy-bound-wv-high-high-5d-case-2} together imply \eqref{eq:fra-energy-bound-wv-high-high-5d-case}, we finish the proof of \eqref{eq:fra-energy-bound-wv-high-high-5d-main}. Then,
\EQn{\label{esti:fra-energy-bound-wv-high-high-5d}
\eqref{eq:fra-energy-bound-wv-high-high-5d} \lsm \de_0 N_0^{2(1-s)}.
}

$\bullet$ \textbf{Estimate of \eqref{eq:fra-energy-bound-wv-high-low-5d}}.
We make the dyadic decomposition in $t$,
\EQ{
\eqref{eq:fra-energy-bound-wv-high-low-5d} \lsm \sum_{N_1\ll N} \sum_{k\in\Z:2^k\loe T_0} N_1^{0+} \int_{2^{k-1}}^{2^k}\int_{\R^d} |\nabla \W| |\nabla \V_N||\U_{ N_1}|^p \dd x \dd t.
}
By H\"older's inequality,
\EQn{\label{eq:fra-energy-bound-wv-high-low-5d-1}
& N_1^{0+} \int_{2^{k-1}}^{2^k}\int_{\R^d} |\nabla \W| |\nabla \V_N||\U_{ N_1}|^p \dd x \dd t\\
\lsm & \norm{\nabla\W}_{L_{t\in I_k}^\I L_x^2} \norm{\nabla\V_N\W_{N_1}}_{L_{t\in I_k}^2L_x^2}^{\th_1} \norm{\nabla\V_{N}}_{L_{t\in I_k}^2L_x^{\frac{2d}{d-2}}}^{1-\th_1} \norm{\U_{N_1}}_{L_{t\in I_k}^{2(p-\th_1)}L_x^{\frac{d(p-\th_1)}{1-\th_1}}}^{p-\th_1}.
}
We then claim the estimate:
\EQn{\label{eq:fra-energy-bound-wv-high-low-5d-bi}
\norm{\nabla \V_N \U_{N_1}}_{L_{t\in I_k}^2L_x^2}^{\th_1} \lsm \de_0^{\th_1} N_1^{\th_1}N^{-1/2\th_1+\th_1(1-s)} 2^{-100\th_1k} N_0^{10\th_1(1-s)}.
}
In fact, by Corollary \ref{cor:fra-energy-global-spacetime-estimate}, if \eqref{eq:fra-energy-bound-wv-high-high-5d-case-1} holds, then
\EQ{
\norm{\nabla \V_N \U_{N_1}}_{L_{t\in I_k}^2L_x^2}^{\th_1} \lsm & \de_0^{\th_1} N_1^{\th_1}N^{-1/2\th_1+\th_1(1-s)}  \\
& \cdot \Big(2^{\brkb{(\frac{dp}{4}-1)(1+(s_1-100\th_1)p) +\frac{1}{2}-\frac{p}{q_0}}k} N_0^{(s_1-100\th_1)(1-s)} \\
& \quad+ 2^{\brkb{(1+\frac{(2+s_1-100\th_1)p}{1-p})\cdot(\frac{dp}{4}-1) + (\frac{1}{2p}-\frac{1}{q_1})\cdot \frac{p}{1-p}}k} N_0^{\frac{s_1-100\th_1}{1-p}(1-s)}\Big)^{\th_1} \\
\lsm & \de_0^{\th_1} N_1^{\th_1}N^{-1/2\th_1+\th_1(1-s)} 2^{-100\th_1k} N_0^{10\th_1(1-s)}.
}
Moreover, if \eqref{eq:fra-energy-bound-wv-high-high-5d-case-2} holds, then
\EQ{
\norm{\nabla \V_N \U_{N_1}}_{L_{t\in I_k}^2L_x^2}^{\th_1} \lsm & \de_0^{\th_1} N_1^{\th_1}N^{-1/2\th_1+\th_1(1-s)} \\  
& \cdot\Big(2^{\brkb{(\frac{dp}{4}-1)(1+sp) +\frac{1}{2}-\frac{p}{q_0}}k} N_0^{s(1-s)} \\
& + 2^{\brkb{(1+\frac{(2+s)p}{1-p})\cdot(\frac{dp}{4}-1) + (\frac{1}{2p}-\frac{1}{q_2})\cdot \frac{p}{1-p}}k} N_0^{\frac{s}{1-p}(1-s)}\Big)^{\th_1} \\
\lsm & \de_0^{\th_1} N_1^{\th_1}N^{-1/2\th_1+\th_1(1-s)} 2^{-100\th_1k} N_0^{10\th_1(1-s)}.
}
This finishes the proof of \eqref{eq:fra-energy-bound-wv-high-low-5d-bi}.

Then, by \eqref{eq:fra-energy-bound-wv-high-low-5d-1} and \eqref{eq:fra-energy-bound-wv-high-low-5d-bi},
\EQ{
& N_1^{0+} \int_{2^{k-1}}^{2^k}\int_{\R^d} |\nabla \W| |\nabla \V_N||\U_{ N_1}|^p \dd x \dd t\\
\lsm & \de_0 2^{(\frac{dp}{4}-1 - 100\th_1)k} N_0^{(1+10\th_1)(1-s)} N^{-\frac12\th_1 +(1-s)} N_1^{0-} \normb{N_1^{\frac{\th_1}{p-\th_1}+}\U_{N_1}}_{L_{t\in I_k}^{2(p-\th_1)}L_x^{\frac{d(p-\th_1)}{1-\th_1}}}^{p-\th_1}.
}
Applying this inequality and using $\sum_{N_1\ll N}N^{-\frac12\th_1 +(1-s)} N_1^{0-} \lsm 1$, we have
\EQ{
\eqref{eq:fra-energy-bound-wv-high-low-5d} \lsm \de_0  N_0^{(1+10\th_1)(1-s)} \sum_{k\in\Z:2^k\loe T_0} 2^{(\frac{dp}{4}-1 - 100\th_1)k} \sup_{N_1:N_1\gsm N_0} \normb{N_1^{\frac{\th_1}{p-\th_1}+}\U_{N_1}}_{L_{t\in I_k}^{2(p-\th_1)}L_x^{\frac{d(p-\th_1)}{1-\th_1}}}^{p-\th_1},
}
thus it suffices to prove
\EQn{\label{eq:fra-energy-bound-wv-high-low-5d-main}
\sum_{k\in\Z:2^k\loe T_0} 2^{(\frac{dp}{4}-1 - 100\th_1)k} \sup_{N_1:N_1\gsm N_0} \normb{N_1^{\frac{\th_1}{p-\th_1}+}\U_{N_1}}_{L_{t\in I_k}^{2(p-\th_1)}L_x^{\frac{d(p-\th_1)}{1-\th_1}}}^{p-\th_1} \lsm N_0^{(1-10\th_1)(1-s)}.
}
In the following, we will prove this in two subcases \eqref{eq:fra-energy-bound-wv-high-high-5d-case-1} and \eqref{eq:fra-energy-bound-wv-high-high-5d-case-2}, similar as the proof of \eqref{eq:fra-energy-bound-wv-high-high-5d-main}.

First, we prove \eqref{eq:fra-energy-bound-wv-high-low-5d-main} under the condition \eqref{eq:fra-energy-bound-wv-high-high-5d-case-1}:
\EQ{
\frac{2}{d-2}<p<\frac4d\text{, when }d=5,6;\quad \frac12<p<\frac47\text{, when }d=7.
}
Recall the definition of $s_1,q_1,r_1$ in Lemma \ref{lem:fra-energy-global-spacetime-estimate-sobolev}: $s_1=\frac{1-p}{p}$, $\frac{1}{q_1}=\frac d4-\frac{s_1}2-\frac{1}{2p}+50d\th_1$, and $\frac{1}{r_1}:=\frac{s_1}{d}+\frac{1}{dp}-100\th_1$.
Note that
\EQ{
\frac{\th_1}{p-\th_1} +d(\frac{1}{r_1}-\frac{1-\th_1}{d(p-\th_1)}) <s_1-100\th_1.
}
then by Bernstein's inequality in $x$,  H\"older's in $t$, and Corollary \ref{cor:fra-energy-global-spacetime-estimate},
\EQ{
\normb{N_1^{\frac{\th_1}{p-\th_1}+}\U_{N_1}}_{L_{t\in I_k}^{2(p-\th_1)}L_x^{\frac{d(p-\th_1)}{1-\th_1}}} \lsm &  2^{(\frac{1}{2(p-\th_1)}-\frac{1}{q_1})k} \normb{N_1^{s_1-100\th_1}\U_{N_1}}_{L_{t\in I_k}^{q_1}L_x^{r_1}} \\
\lsm & N_0^{\frac{s_1-100\th_1}{1-p}(1-s)} \Big( 2^{\brkb{(\frac{dp}{4}-1)(s_1-100\th_1) + \frac{1}{2p}-\frac{1}{q_1}-10\th_1}k}  \\
& + 2^{(\frac{2+s_1}{1-p}\cdot(\frac{dp}{4}-1) + \frac{1}{1-p}\cdot(\frac1{2p}-\frac{1}{q_1})-10\th_1)k} \Big).
}
Then, applying this inequality to the (LHS) of \eqref{eq:fra-energy-bound-wv-high-low-5d-main}, and by the definition of $s_1$, $q_1$,  $\frac{s_1p}{1-p}\le1$, $100-75dp>-200$, and $-\frac{50dp}{1-p}>-1000$,
\EQn{\label{eq:fra-energy-bound-wv-high-low-5d-main-1}
& \text{(LHS) of \eqref{eq:fra-energy-bound-wv-high-low-5d-main}} \\
\lsm &N_0^{( \frac{s_1p}{1-p} -10\th_1)(1-s)} \sum_{k\in\Z:2^k\loe T_0} \Big( 2^{\brkb{(\frac{dp}{4}-1)(1+p(s_1-100\th_1)) + \frac{1}{2}-\frac{p}{q_1}-10\th_1}k} \\
& \quad\qquad\qquad\qquad\qquad\qquad + 2^{((1+\frac{(2+s_1)p}{1-p})\cdot(\frac{dp}{4}-1) + \frac{p}{1-p}\cdot(\frac1{2p}-\frac{1}{q_1}) -10\th_1)k}  \Big) \\
\lsm & N_0^{(1 -10\th_1)(1-s)} \sum_{k\in\Z:2^k\loe T_0} \Big( 2^{\brkb{-\frac{dp^2}{4}+\frac34(d+2)p-\frac52+(100-75dp)\th_1-10\th_1}k} 
+  2^{(\frac{(d-2)p-2}{4(1-p)} - \frac{50dp}{1-p}\th_1-10\th_1)k} \Big)
\\
\lsm & N_0^{(1 -10\th_1)(1-s)} \sum_{k\in\Z:2^k\loe T_0} \brkb{ 2^{\brkb{-\frac{dp^2}{4}+\frac34(d+2)p-\frac52-210\th_1}k} 
+  2^{(\frac{(d-2)p-2}{4(1-p)} - 1010\th_1)k} }.
}
We observe that this is the same as the case in \eqref{eq:fra-energy-bound-wv-high-high-5d-main-1} up to a $2^{-10\th_1}$-factor, thus applying \eqref{eq:fra-energy-bound-wv-high-high-5d-main-1-com-1} and \eqref{eq:fra-energy-bound-wv-high-high-5d-main-1-com-2} to \eqref{eq:fra-energy-bound-wv-high-low-5d-main-1},
\EQ{
\text{(LHS) of \eqref{eq:fra-energy-bound-wv-high-low-5d-main}}
\lsm  N_0^{(1 -10\th_1)(1-s)}.
}
This gives \eqref{eq:fra-energy-bound-wv-high-low-5d-main} when \eqref{eq:fra-energy-bound-wv-high-high-5d-case-1} holds.

Next, we prove \eqref{eq:fra-energy-bound-wv-high-low-5d-main} under the condition \eqref{eq:fra-energy-bound-wv-high-high-5d-case-2}:
\EQ{
p=\frac{2}{d-2}&\text{, when }d=5,6;\quad \frac{2}{d-2}\loe p\loe \frac12\text{, when }d=7;\\
& \frac{2}{d-2}\loe p <\frac{2}{d-4}\text{, when }8\loe d\loe 11.
}
Denote that
\EQ{
\frac{1}{\wt r_2}:= \frac{s}{d} + \frac{1-2\th_1}{d(p-\th_1)}+\text{, and }\al_1:=\frac d2-\frac{d}{r_2}.
}
We can check that under the condition \eqref{eq:fra-energy-bound-wv-high-high-5d-case-2}, $2<r_2<\frac{2d}{d-2}$ and $0<\al_1<1$. Note that
\EQ{
\frac{\th_1}{p-\th_1}+ d(\frac{1}{r_2}-\frac{1-\th_1}{d(p-\th_1)})<s,
}
then by Bernstein's, interpolation,  H\"older's inequalities, and Corollary \ref{cor:fra-energy-global-spacetime-estimate},
\EQ{
& \normb{N_1^{\frac{\th_1}{p-\th_1}+}\U_{N_1}}_{L_{t\in I_k}^{2(p-\th_1)}L_x^{\frac{d(p-\th_1)}{1-\th_1}}} \\
\lsm &  \normb{N_1^{s}\U_{N_1}}_{L_{t\in I_k}^{2(p-\th_1)}L_x^{r_2}} \\
\lsm &  \normb{\normo{N_1^{s}\U_{N_1}}_{L_x^2}^{1-\al_1} \normo{N_1^{s}\U_{N_1}}_{L_x^{\frac{2d}{d-2}}}^{\al_1} }_{L_{t\in I_k}^{2(p-\th_1)}} \\
\lsm &  2^{(1-\al_1)s(\frac{dp}{4}-1)k} N_0^{(1-\al_1)s(1-s)} \normb{N_1^{s}\U_{N_1}}_{L_t^{2\al_1(p-\th_1)}L_x^{\frac{2d}{d-2}}}^{\al_1} \\
\lsm &  2^{((1-\al_1)s(\frac{dp}{4}-1)+ \frac{1}{2(p-\th_1)}-\frac{\al_1}{2})k} N_0^{(1-\al_1)s(1-s)} \normb{N_1^{s}\U_{N_1}}_{L_t^{2}L_x^{\frac{2d}{d-2}}}^{\al_1} \\
\lsm &  2^{\brkb{s(\frac{dp}{4}-1)+ \frac{1}{2p}-\frac{\al_1}{2}-10\th_1}k}N_0^{((1-\al_1)s+s\al_1)(1-s)} \\
& + 2^{\brkb{(1-\al_1)s(\frac{dp}{4}-1)+ \frac{1}{2p}-\frac{\al_1}{2}+\frac{\al_1(2+s)}{1-p}\cdot(\frac{dp}{4}-1) + (\frac1{2p}-\frac{1}{q_2})\cdot \frac{\al_1p}{1-p}-10\th_1}k}N_0^{((1-\al_1)s+\frac{s\al_1}{1-p})(1-s)},
}
where $q_2$ is defined in Lemma \ref{lem:fra-energy-global-spacetime-estimate-sobolev}. Then, similar to \eqref{eq:fra-energy-bound-wv-high-high-5d-main-2}, recall the definition of $\al_1$, 
\EQn{\label{eq:fra-energy-bound-wv-high-low-5d-2-1}
& \text{(LHS) of \eqref{eq:fra-energy-bound-wv-high-low-5d-main}} \\
\lsm & \sum_{k\in\Z:2^k\loe T_0} 2^{(\frac{dp}{4}-1 - 100\th_1)k} N_0^{10\th_1(1-s)} \sup_{N_1:N_1\gsm N_0} \normb{N_1^{\frac{\th_1}{p-\th_1}+}\U_{N_1}}_{L_t^{2(p-\th_1)}L_x^{\frac{d(p-\th_1)}{1-\th_1}}}^{p-\th_1} \\
\lsm & N_0^{\brkb{\frac{(d-4)p^2}{2(1-p)}+100\th_1+}(1-s)} \\
& \cdot \sum_{k\in\Z:2^k\loe T_0}2^{(-100(1-s)-110\th_1-)k} \Big( 2^{\brko{\frac14p(dp-2)}k} 
+ 2^{\brkb{\frac{d(d-4)p^3+(d^2-4d+8)p^2+(8-6d)p+8}{8(1-p)}}k} \Big),
}
which is up to a $2^{-10\th_1}$-factor compared to \eqref{eq:fra-energy-bound-wv-high-high-5d-main-2}. Therefore, applying \eqref{eq:fra-energy-bound-wv-high-high-5d-main-2-com-1}, \eqref{eq:fra-energy-bound-wv-high-high-5d-main-2-com-2}, and \eqref{eq:fra-energy-bound-wv-high-high-5d-main-2-com-3} to \eqref{eq:fra-energy-bound-wv-high-low-5d-2-1}, we have that 
\EQ{
\text{(LHS) of \eqref{eq:fra-energy-bound-wv-high-low-5d-main}} 
\lsm  N_0^{(1-10\th_1)(1-s)}.
}
This gives \eqref{eq:fra-energy-bound-wv-high-low-5d-main} when \eqref{eq:fra-energy-bound-wv-high-high-5d-case-2} holds.

Now, noting that \eqref{eq:fra-energy-bound-wv-high-high-5d-case-1} and \eqref{eq:fra-energy-bound-wv-high-high-5d-case-2} together imply \eqref{eq:fra-energy-bound-wv-high-high-5d-case}, we finish the proof of \eqref{eq:fra-energy-bound-wv-high-low-5d-main}. Then,
\EQn{\label{esti:fra-energy-bound-wv-high-low-5d}
	\eqref{eq:fra-energy-bound-wv-high-low-5d} \lsm \de_0 N_0^{2(1-s)}.
}

$\bullet$ \textbf{Estimate of \eqref{eq:fra-energy-bound-vv-high-high-5d}}. First, we make a dyadic decomposition in $t$:
\EQ{
\eqref{eq:fra-energy-bound-vv-high-high-5d} \lsm \sum_{N_1\loe N_2,N_2\lsm N}\sum_{k\in\Z:2^k\loe T_0} N^{0+} \int_{2^{k-1}}^{2^k} \int_{\R^d} |\nabla\V_{N_1}||\nabla\V_{N_2}||\U_{N}|^p\dd x\dd t.
}

By H\"older's inequality and Lemma \ref{lem:fra-linear},
\EQ{
& N^{0+} \int_{2^{k-1}}^{2^k}\int_{\R^d} |\nabla \V_{N_1}| |\nabla \V_{N_2}||\U_{ N}|^p \dd x \dd t \\
\lsm & N_1^{1-s} N_2^{1-s} N^{-\frac12\th_1} \norm{|\nabla|^s\V_{N_1}}_{L_{t\in I_k}^\I L_x^2}  \norm{|\nabla|^{s}\V_{N_2}}_{L_{t\in I_k}^2L_x^{\frac{2d}{d-2}}} \normb{N^{\frac{\th_1}{2p}+}\U_{N}}_{L_{t\in I_k}^{2p}L_x^{dp}}^{p} \\
\lsm & \de_0^2 N_1^{1-s} N_2^{1-s} N^{-\frac12\th_1}  \normb{N^{\frac{\th_1}{2p}+}\U_{N}}_{L_{t\in I_k}^{2p}L_x^{dp}}^{p}.
}
By the choice of $\th_1$, we have
\EQ{
\sum_{N_1\loe N_2,N_2\lsm N} N_1^{1-s} N_2^{1-s} N^{-\frac12\th_1} \lsm \sum_{N_2:N_2\gsm N_0} N_2^{2(1-s)-\frac12\th_1}  \lsm 1.
}
Then, by $1\lsm 2^{(\frac{dp}{4}-1)k}$ and \eqref{eq:fra-energy-bound-wv-high-high-5d-main},
\EQn{\label{esti:fra-energy-bound-vv-high-high-5d}
\eqref{eq:fra-energy-bound-vv-high-high-5d} \lsm & \de_0^2  \sum_{k\in\Z:k\lsm \log T_0}\sup_{N\in2^\N}  \normb{N^{\frac{\th_1}{2p}+}\U_{N}}_{L_{t\in I_k}^{2p}L_x^{dp}}^{p} \\
\lsm & \de_0^2  \sum_{k\in\Z:k\lsm \log T_0}\sup_{N\in2^\N} 2^{(\frac{dp}{4}-1)k} \normb{N^{\frac{\th_1}{2p}+}\U_{N}}_{L_{t\in I_k}^{2p}L_x^{dp}}^{p} \\
\lsm & \de_0^2 N_0^{1-s}.
}

$\bullet$ \textbf{Estimate of \eqref{eq:fra-energy-bound-vv-high-low-5d}}. First, we make a dyadic decomposition in $t$:
\EQ{
\eqref{eq:fra-energy-bound-vv-high-low-5d} \lsm \sum_{N_1\loe N_2,N\ll N_2} \sum_{k\in\Z:2^k\loe T_0} N^{0+} \int_{2^{k-1}}^{2^k} \int_{\R^d} |\nabla\V_{N_1}||\nabla\V_{N_2}||\U_{N}|^p\dd x\dd t.
}
By H\"older's inequality, Lemma \ref{lem:fra-linear}, and \eqref{eq:fra-energy-bound-wv-high-low-5d-bi},
\EQ{
& N^{0+} \int_{2^{k-1}}^{2^k} \int_{\R^d} |\nabla\V_{N_1}||\nabla\V_{N_2}||\U_{N}|^p\dd x\dd t \\
\lsm & N^{0+} \norm{\nabla\V_{N_1}}_{L_{t\in I_k}^\I L_x^2} \norm{\nabla\V_{N_2}\U_{N}}_{L_{t\in I_k}^2L_x^2}^{\th_1} \norm{\nabla\V_{N_2}}_{L_{t\in I_k}^2L_x^{\frac{2d}{d-2}}}^{1-\th_1} \norm{\U_{N}}_{L_{t\in I_k}^{2(p-\th_1)}L_x^{\frac{d(p-\th_1)}{1-\th_1}}}^{p-\th_1} \\
\lsm & \de_0^2 N^{0+}\cdot  N_1^{1-s} \cdot  N^{\th_1}N_2^{-1/2\th_1+\th_1(1-s)} 2^{-100\th_1k} N_0^{10\th_1(1-s)} \\
& \cdot N_2^{(1-\th_1)(1-s)} \norm{\U_{N}}_{L_{t\in I_k}^{2(p-\th_1)}L_x^{\frac{d(p-\th_1)}{1-\th_1}}}^{p-\th_1} \\
\lsm & \de_0^2 2^{-100\th_1k} N_1^{1-s} N_2^{-\frac12\th_1 + 1-s} N_0^{10\th_1(1-s)} N^{0-} \normb{N^{\frac{\th_1}{p-\th_1}+}\U_{N}}_{L_{t\in I_k}^{2(p-\th_1)}L_x^{\frac{d(p-\th_1)}{1-\th_1}}}^{p-\th_1}.
}
By the choice of $\th_1$, we have
\EQ{
\sum_{N_1\loe N_2,N\ll N_2} N_1^{1-s} N_2^{-\frac12\th_1 + 1-s} N^{0-} \lsm \sum_{N_2:N_2\gsm N_0} N_2^{-\frac12\th_1+2(1-s)}  \lsm 1.
}
Then, by $1\lsm 2^{(\frac{dp}{4}-1)k}$ and \eqref{eq:fra-energy-bound-wv-high-low-5d-main},
\EQn{\label{esti:fra-energy-bound-vv-high-low-5d}
\eqref{eq:fra-energy-bound-vv-high-low-5d} \lsm & \de_0^2 N_0^{10\th_1(1-s)} \sum_{k\in\Z:2^k\loe T_0} 2^{-100\th_1k} \sup_{N:N\gsm N_0} \normb{N^{\frac{\th_1}{p-\th_1}+}\U_{N}}_{L_{t\in I_k}^{2(p-\th_1)}L_x^{\frac{d(p-\th_1)}{1-\th_1}}}^{p-\th_1} \\
\lsm & \de_0^2 N_0^{10\th_1(1-s)} \sum_{k\in\Z:2^k\loe T_0} 2^{(\frac{dp}{4}-1 - 100\th_1)k} \sup_{N_1:N_1\gsm N_0} \normb{N_1^{\frac{\th_1}{p-\th_1}+}\U_{N_1}}_{L_{t\in I_k}^{2(p-\th_1)}L_x^{\frac{d(p-\th_1)}{1-\th_1}}}^{p-\th_1} \\
\lsm & \de_0^2 N_0^{1-s}.
}

Finally, Lemma \ref{lem:fra-energy-5d} follows by \eqref{esti:fra-energy-bound-wv-high-high-5d},
\eqref{esti:fra-energy-bound-wv-high-low-5d},
\eqref{esti:fra-energy-bound-vv-high-high-5d}, and
\eqref{esti:fra-energy-bound-vv-high-low-5d}.
\end{proof}

\vskip 1.5cm
\section{Almost sure scattering}\label{sec:almost-sure-scattering}
\vskip .5cm

Let $u$ be the solution to \eqref{eq:nls-random}. We denote 
\EQ{
v=S(t)f^\om\text{, and } \V(t,x)=\mathcal Tv(t,x).
}
Since we are considering a fixed equation, the exponent $p$ does not affect the subsequent argument, and can be viewed as a constant. Therefore, in the following, we will omit the dependence on $p$.

\subsection{Almost sure Strichartz estimate for $Jv$ and $\V$}
In this subsection, we derive the almost sure Strichartz estimate for $Jv$ and $\V$. Fix a sufficiently small parameter $\ep>0$, and denote 
\EQ{
X^{-\ep}(I) :=
L^{q_1}_tL^{r_1}_x(I\times \R^d),
}
where $I\subset\R$, $\frac{1}{q_1}=\frac{d}{2(d+2)} + \frac\ep2$, and $\frac{1}{r_1}=\frac{d}{2(d+2)}$. Here, we use the notation $X^{-\ep}$ to indicate that it is a $\dot H_x^{-\ep}$-level spacetime norm, due to the fact that 
$$
 \frac{2}{q_1}+\frac{d}{r_1}=\frac d2+\ep.
$$
Then, we define $Y$ norm for the linear solution $v=S(t)f^\om$:
\EQ{
	\norm{v}_{Y} := \norm{Jv}_{X^{-\ep}([0,1])} + \norm{\jb{\nabla}^2 \mathcal Tv }_{ L_t^\I L_x^{p+2}(\R\times\R^d)}.
} 
The main result in this subsection is the following linear estimate.
\begin{prop}\label{prop:bound-v-Y}
	Let $d\goe 1$, $0<\ep\loe \frac{d+4}{d+2}$, $\frac2d<p<\frac4d$, and $f^\om$ be defined by Definition \ref{defn:randomization}. Then, there exist $c>0$ and $C>0$ independent of $\norm{f}_{L_x^2(\R^d)}$, such that for any $\la>0$,
	\EQ{
		\PP\brkb{\fbrkb{\om\in\Om: \norm{v}_{Y}>\la}} \loe C e^{-c\la^2\norm{f}_{L_x^2(\R^d)}^{-2}}.
	}
	Moreover, for almost every $\om\in\Om$,
	\EQ{
		\norm{v}_{Y}<\I.
	}
\end{prop} 

Proposition \ref{prop:bound-v-Y} follows directly by Lemmas  \ref{lem:probability-estimate},  \ref{lem:local-Jv}, and \ref{lem:local-V-T}.
\begin{lem}[$L^2$ subcritical estimate]\label{lem:local-Jv}
Let $d\goe 1$, $0<\ep\loe \frac{d+4}{d+2}$, and $f^\om$ be defined by Definition \ref{defn:randomization}. 
Then, for any $\rho\goe 2$,
\EQ{
\norm{JS(t)f^\om}_{L_\om^\rho X^{-\ep}(\Om\times[0,1])} \lsm \sqrt{\rho} \norm{f}_{L_x^2(\R^d)}.
}
\end{lem}
\begin{proof}
We restrict the variables $\om\in\Om$, $j\in\N$, $t\in[0,1]$, and $x\in\R^d$. 
Let $(q_1,r_1)$ such that $\frac{1}{q_1}=\frac{d}{2(d+2)} + \frac\ep2$ and $\frac{1}{r_1}=\frac{d}{2(d+2)}$. For any $\rho\goe \frac{2(d+2)}{d}$, by Minkowski inequality,
\EQ{
\norm{JS(t)f^\om}_{L_\om^\rho L_t^{q_1} L_x^{r_1}} = & \norm{S(t)(xf^\om)}_{L_\om^\rho L_t^{q_1} L_x^{r_1}} \\
\lsm & \sqrt{\rho} \norm{S(t)(xf^\om)}_{L_t^{q_1} L_x^{r_1}L_\om^\rho} \\
\lsm & \sqrt{\rho} \norm{S(t)(x\psi_jf)}_{L_t^{q_1} L_x^{r_1}l_j^2}\\
\lsm & \sqrt{\rho} \norm{S(t)(x\psi_jf)}_{l_j^2L_t^{q_1} L_x^{r_1}}.
}
If $2\loe\rho< \frac{2(d+2)}{d}$,
we first use H\"older's inequality in $\om$,
\EQ{
\norm{JS(t)f^\om}_{L_\om^\rho L_t^{q_1} L_x^{r_1}} \lsm \norm{JS(t)f^\om}_{L_\om^{\frac{2(d+2)}{d}} L_t^{q_1} L_x^{r_1}},
}
then by similar argument as above,
\EQ{
\norm{JS(t)f^\om}_{L_\om^{\frac{2(d+2)}{d}} L_t^{q_1} L_x^{r_1}} \lsm_d \norm{S(t)(x\psi_jf)}_{l_j^2L_t^{q_1} L_x^{r_1}} \lsm_d \sqrt{\rho} \norm{S(t)(x\psi_jf)}_{l_j^2L_t^{q_1} L_x^{r_1}}.
}
Now, we omit the dependence on $d$, and obtain for all $\rho\goe2$,
\EQ{
\norm{JS(t)f^\om}_{L_\om^\rho L_t^{q_1} L_x^{r_1}} \lsm \sqrt{\rho} \norm{S(t)(x\psi_jf)}_{l_j^2L_t^{q_1} L_x^{r_1}}.
}

Fix $j\in\N$. First, by Strichartz estimate, 
\EQn{\label{esti:bound-Jv-str}
\norm{S(t)(x\psi_jf)}_{L_{t,x}^{r_1}} \lsm \norm{x\psi_jf}_{L_x^2}.
}
Second, using the dispersive inequality,
\EQ{
\norm{S(t)(x\psi_jf)}_{L_{x}^{r_1}} \lsm t^{-\frac{d}{d+2}}\norm{x\psi_jf}_{L_x^{r_1'}}.
}
Noting that $t\in[0,1]$, this implies that  
\EQn{\label{esti:bound-Jv-dispersive}
\norm{S(t)(x\psi_jf)}_{L_t^1L_{x}^{r_1}} \lsm \norm{x\psi_jf}_{L_x^{r_1'}}.
}
Then, interpolating between \eqref{esti:bound-Jv-str} and \eqref{esti:bound-Jv-dispersive}, it infers that 
\EQ{
\norm{S(t)(x\psi_jf)}_{L_t^{q_1} L_x^{r_1}} \lsm \norm{x\psi_jf}_{L_x^{r_2'}},
}
where $r_2$ satisfies
$$
\frac1{r_2}=\frac12-\frac{\ep}{d+4}.
$$

Now, we turn to estimate $\norm{x\psi_jf}_{L_x^{r_2'}}$. For each $j\in\N$, there exists $N\in2^\N$ such that $|x|\sim N$, and the volume of the support set for $\psi_j$ is $CN^{-da}$. Using H\"older's inequality and by the choice of $a$,
\EQ{
\norm{x\psi_jf}_{L_x^{r_2'}} \lsm N\cdot N^{-da(\frac{1}{r_2'}-\frac{1}{2})} \norm{\psi_jf}_{L_x^{2}} \lsm N^{1-\frac{d}{d+4}a\ep} \norm{\psi_jf}_{L_x^{2}} \lsm \norm{\psi_jf}_{L_x^{2}}.
}
This gives that 
\EQ{
\norm{S(t)(x\psi_jf)}_{L_t^{q_1} L_x^{r_1}} \lsm \norm{\psi_jf}_{L_x^{2}}.
}
Therefore,
\EQ{
\norm{S(t)(x\psi_jf)}_{l_j^2L_t^{q_1} L_x^{r_1}} \lsm \norm{\psi_jf}_{l_j^2L_x^{2}}\lsm \norm{f}_{L_x^2}.
}
This finishes the proof.
\end{proof}

Next, we recall an explicit formula of $\mathcal V$, which plays an important role in our analysis. 
\begin{lem}\label{lem:explicit-V}
Let $v=S(t)f^\om$, $\V=\TT v$, and $V(t,x):=e^{- \frac12 i t|x|^2}\overline{f^\om}(x)$, 
then 
\EQn{\label{esti:bound-V-representation-nabla}
\V(t,x) =\wh{V}(t,-x).
}
\end{lem}
\begin{proof}
This lemma follows directly from Lemma \ref{lem:J-T}:
\EQ{
\V(t,x) = \TT S(t) f^\om = S(t) \F^{-1}(\wb{f^\om}) = \F(e^{-\frac12it|\cdot|^2}\wb{f^\om})(-x).
}
\end{proof}

Applying Lemma \ref{lem:explicit-V}, we have the following lemma.
\begin{lem}[Improved weighted estimate] \label{lem:local-V-T}
Let $d\goe 1$, $f^\om$ be defined by Definition \ref{defn:randomization}, $v=S(t)f^\om$, and $\V=\TT v$.  Then, for any $\rho\goe 2$,
\EQ{
\norm{\jb{\nabla}^2\V}_{L_\om^\rho L_t^\I L_x^{p+2}(\Om\times\R\times\R^d)} \lsm \sqrt{\rho}\norm{f}_{L_x^2(\R^d)}.
}
\end{lem}
\begin{proof}
In the proof of this lemma, we write the spatial variable of $\V$ as $\xi$ to avoid confusion, namely $\V(t,\xi)$. We also restrict the variables $\om\in\Om$, $t\in\R$, and $x,\xi\in\R^d$.  Recall the definition of $f^\om$ in \eqref{eq:randomization}, i.e.  $f^\om=\sum_{j\in\N} g_j(\om)\psi_j f$. Then by Lemma \ref{lem:explicit-V},
\begin{align*}
V(t,x)
&=\sum_{j\in\N} g_j(\om)\psi_j(x) e^{-\frac12i t|x|^2}  \wb{f} (x)\\
&= \sum_{j\in\N} g_j(\om)\psi_j(x) h_1(x),
\end{align*}
where we denote $h_1(x):= e^{- \frac12 i t|x|^2} \wb{f}$, and 
\EQ{
\V(t,\xi) = \sum_{j\in\N} g_j(\om)\>\widehat{\psi_j h_1}(-\xi),
}

Next, we give a ``Bernstein's inequality" in $t$:
\EQn{\label{eq:local-V-T-bernstein}
\norm{\V_N}_{L_t^{\I}L_\xi^{p+2}(\R\times\R^d)} \lsm N^{2\ep_0} \norm{\V_N}_{L_t^{\frac{1}{\ep_0}}L_\xi^{p+2}(\R\times\R^d)}.
}
Let $\ep_0>0$ be a sufficiently small absolute constant. For any $t_1,t_2\in\R$ and $N\in2^\N$, by the fundamental theorem of calculus, and $\pd_t\V=i\De \V$, we get that
\EQ{
\norm{\V_N(t_2)}_{L_\xi^{p+2}} \lsm & \norm{\V(t_1)}_{L_\xi^{p+2}} + \int_{t_1}^{t_2}\norm{\pd_s\V_N(s)}_{L_\xi^{p+2}} \dd s \\
\lsm & \norm{\V(t_1)}_{L_\xi^{p+2}} + N^2|t_2-t_1|^{1-\ep_0}\norm{\V_N}_{L_t^{\frac{1}{\ep_0}}L_\xi^{p+2}(\R\times\R^d)}.
}
Then,
\EQ{
N^{-2}\norm{\V_N(t_2)}_{L_\xi^{p+2}}^{\frac{1}{\ep_0}} \lsm & \int_{t_2-N^{-2}}^{t_2+N^{-2}}\norm{\V(t_1)}_{L_\xi^{p+2}}^{\frac{1}{\ep_0}} \dd t_1 \\
& + N^{\frac{2}{\ep_0}}\int_{t_2-N^{-2}}^{t_2+N^{-2}}|t_2-t_1|^{\frac{1-\ep_0}{\ep_0}}\dd t_1\norm{\V_N}_{L_t^{\frac{1}{\ep_0}}L_\xi^{p+2}(\R\times\R^d)}^{\frac{1}{\ep_0}} \\
\lsm & \norm{\V_N}_{L_t^{\frac{1}{\ep_0}}L_\xi^{p+2}(\R\times\R^d)}^{\frac{1}{\ep_0}}.
}
Then, \eqref{eq:local-V-T-bernstein} follows by taking the supremum in $t_2$ in the above inequality.

Now, we consider the case when $\rho\goe 1/\ep_0$. By Littlewood-Paley decomposition, Minkowski's inequality and \eqref{eq:local-V-T-bernstein},
\EQ{
\norm{\jb{\nabla}^2\V}_{L_t^{\I} L_\xi^{p+2}} \lsm \norm{\jb{\nabla}^{2+2\ep_0}\V_N}_{l_N^2L_t^{\frac{1}{\ep_0}} L_\xi^{p+2}}.
}
By Minkowski's and Hausdorff-Young's inequalities,
\EQ{
\norm{\jb{\nabla}^2\V}_{ L_t^{\I} L_\xi^{p+2}}
\lsm & \norm{\jb{\nabla}^{2+2\ep_0}\V_N}_{l_N^2L_t^{\frac{1}{\ep_0}} L_\xi^{p+2} L_\om^\rho} \\
\lsm & \normb{\sum_{j\in\N}g_j(\om)P_N\F\brkb{\jb{x}^{2+2\ep_0}\psi_j h_1}(-\xi)}_{l_N^2L_t^{\frac{1}{\ep_0}}L_\xi^{p+2}L_\om^\rho} \\
\lsm & \sqrt\rho\normb{\F\brkb{\jb{x}^{2+2\ep_0}\ph_N\psi_j h_1}(-\xi)}_{l_N^2L_t^{\frac{1}{\ep_0}}L_\xi^{p+2} l_j^2} \\ 
\lsm &  \sqrt\rho\normb{\F\brkb{\jb{x}^{2+2\ep_0}\ph_N\psi_j h_1}(-\xi)}_{l_j^2 l_N^2L_t^{\frac{1}{\ep_0}}L_\xi^{p+2}} \\ 
\lsm & \sqrt\rho\norm{\jb{x}^{2+2\ep_0}\ph_N\psi_j h_1}_{l_j^2l_N^2 L_x^{\frac{p+2}{p+1}} } \\
\sim &  \sqrt\rho\norm{\jb{x}^{2+2\ep_0}\ph_N\psi_j f}_{l_j^2l_N^2 L_x^{\frac{p+2}{p+1}}}.
}
If $2\loe \rho < 1/\ep_0$, using H\"older's inequality in $\om$,
\EQ{
\norm{\jb{\nabla}^2\V}_{L_\om^\rho L_t^\I L_\xi^{p+2}} \lsm \norm{\jb{\nabla}^2\V}_{L_\om^{1/\ep_0} L_t^\I L_\xi^{p+2}}.
}
Then, similar as above, 
\EQ{
\norm{\jb{\nabla}^2\V}_{L_\om^{1/\ep_0} L_t^\I L_\xi^{p+2}} & \lsm \sqrt{1/\ep_0} \norm{\jb{x}^{2+2\ep_0}\ph_N\psi_j h_1}_{l_j^2 l_N^2 L_x^{\frac{p+2}{p+1}}} \\
& \lsm \sqrt{\rho}\cdot\sqrt{\frac{1}{\rho\ep_0}} \norm{\jb{x}^{2+2\ep_0}\ph_N\psi_j h_1}_{l_j^2 l_N^2 L_x^{\frac{p+2}{p+1}}} \\
& \lsm_{\ep_0}  \sqrt{\rho}\norm{\jb{x}^{2+2\ep_0}\ph_N\psi_j h_1}_{l_j^2 l_N^2 L_x^{\frac{p+2}{p+1}}}.
}
Therefore, we have that for all $\rho\goe2$,
\EQ{
\norm{\jb{\nabla}^2\V}_{L_\om^{\rho} L_t^\I L_\xi^{p+2}} \lsm_{\ep_0} \sqrt\rho\norm{\jb{x}^{2+2\ep_0}\ph_N\psi_j f}_{l_j^2l_N^2 L_x^{\frac{p+2}{p+1}}}.
}

For each $j\in \N$, there exists $N$ such that $|x|\sim N$ and the volume of the support set for $\psi_j$ is $CN^{-da}$. Note that $a\ge \frac{5(p+2)}{dp}$, then by choosing suitable $\ep_0$ and H\"older's inequality,
\EQ{
\normb{\jb{x}^{2+2\ep_0}\ph_N \psi_j f}_{ L_x^{\frac{p+2}{p+1}} }
\lsm  N^{2+2\ep_0-\frac{dpa}{2(p+2)}} \norm{\ph_N\psi_jf}_{L_x^2} 
\lsm  \norm{\ph_N\psi_jf}_{L_x^2},
}
which yields that 
\EQ{
\norm{\jb{\nabla}^2\V}_{L_\om^\rho L_t^\I L_\xi^{p+2}}
\lsm \sqrt\rho \norm{\ph_N\psi_jf}_{l_N^2l_j^2 L_x^2}
\lsm \sqrt\rho\|f\|_{L^2}.
}
This finishes the proof.
\end{proof}

\subsection{Reduction to the deterministic problem}\label{sec:reduction}
Let $u$ be the solution of \eqref{eq:nls}, and let
\EQn{\label{def:v-w}
	v=S(t)u_0\text{, and }w=u-v.
}
Then, $w$ satisfies the equation
\EQn{
	\label{eq:nls-w}
	\left\{ \aligned
	&i\pd_t w + \frac12\De w = |u|^p u, \\
	& w(0,x) = 0.
	\endaligned
	\right.
}
Now, we denote  
\EQn{\label{def-mathcal-UVW}
	\U(t,x):=\mathcal Tu(t,x)\text{, }\V(t,x):=\mathcal Tv(t,x)\text{, and  }\W(t,x):=\mathcal Tw(t,x).
}
Then, the related final data is defined by
\EQ{
\U_+:=\F^{-1}\wb u_0 \text{, }\V_+:=\F^{-1}\wb u_0 \text{, and }\W_+:=0.
}
Moreover, $\W(t,x)$ satisfies the equation
\EQn{\label{eq:nls-W}
i\pd_t \W + \frac12 \De \W = t^{\frac{dp}{2}-2}|\U|^p\U.
}
%

Now, we reduce the proof of Theorem \ref{thm:main} to the following deterministic problem:
\begin{prop}\label{prop:deterministic}
Let $d\goe 1$, $A>1$, and $\frac{2}{d}<p<\frac 4d$. Let also $v$ and $w$ be defined as above. Assume that the initial data $u_0\in L_x^2(\R^d)$ and $v$ satisfies
\EQn{\label{Hyth-determined}
\norm{u_0}_{L_x^{2}(\R^d)} + \norm{v}_{Y}\loe A.
}
Then, there exists a global solution $w\in C(\R;L_x^2(\R^d))$ to \eqref{eq:nls-w} such that there exists $w_\pm\in L_x^2(\R^d)$ with
\EQ{
\lim_{t\ra\pm\I} \norm{w(t)-S(t)w_{\pm}}_{L_x^2(\R^d)}=0.
}
\end{prop}

Applying this proposition, we give the proof of Theorem \ref{thm:main}.
\begin{proof}[Proof of Theorem \ref{thm:main}]
Let 
\EQ{
	u(t)=S(t)f^\om + w(t),
}
with $u_0=f^\om$. Then, $w$ satisfies the equation \eqref{eq:nls-w}, and 
\EQ{
v(0)=f^\om\text{, }w(0)=0\text{, and }v=S(t)f^\om.
}

By Proposition \ref{prop:bound-v-Y}, we have for almost every $\om\in\Om$,
\EQ{
	\norm{v}_{Y}<\I.
}
Moreover, by Lemma \ref{lem:large-deviation}, for any $\rho\goe 2$,
\EQ{
\norm{f^\om}_{L_\om^\rho L_x^2} \lsm \sqrt{\rho}\norm{\psi_jf}_{L_j^2 L_x^2} \lsm\sqrt{\rho} \norm{f}_{L_x^2}.
}
Therefore, we also have for almost every $\om\in\Om$
\EQ{
\norm{f^\om}_{L_x^2(\R^d)} <+\I.
}
Since $w_0=0$, we can apply Proposition \ref{prop:deterministic} to obtain the global well-posedness and scattering of $w\in C(\R; L_x^2(\R^d))$ for almost every $\om\in\Om$.
\end{proof}

\subsection{Energy estimates near the infinity time}\label{sec:Finialdata}
Recall the modified energy in \eqref{defn:pseudo-conformal-energy}:
\EQ{
	\E(t) = \frac{1}{4}t^{2-\frac{dp}{2}}\int_{\R^d}|\nabla\W(t,x)|^2 \dx + \frac{1}{p+2} \int_{\R^d}|\U(t,x)|^{p+2} \dx.
}
Note that $\norm{\V}_{L_t^\I L_x^{p+2}} \lsm \norm{v}_Y\lsm A$, then by Proposition \ref{prop:scattering-criterion}, it suffices to prove
\EQ{
\sup_{0<t\le T_0} \E(t)\lsm A^2,
}
for some $T_0>0$. This will be finished in Proposition \ref{prop:finialdata} and \ref{prop:energy-pseudo}. We first have that 
\begin{prop}\label{prop:finialdata}
If the hypothesis \eqref{Hyth-determined} holds, then there exists $T_0=T_0(\norm{u_0}_{L_x^2})>0$ such that 
\begin{align}
\E(T_0)\lesssim A^2.
\end{align}
\end{prop}

Using \eqref{H1-T}, it suffices to show the analogous estimate on $Jw(\frac{1}{T_0})$, which is given by the following lemma. 
\begin{lem}\label{lem:local-Jw}
If the hypothesis \eqref{Hyth-determined} holds, then there exists $\delta=\de(\norm{u_0}_{L_x^2})>0$ such that 
\EQ{
\sup_{t\in[0,\de]}\norm{Jw(t)}_{L_x^2} \loe A.
}
\end{lem}
\begin{proof}
Here we adopt the argument  and the notations used in the proof of Lemma \ref{lem:Ju-local}. We restrict all the spacetime norms on $[0,\de]\times \R^d$. Then by Lemma \ref{lem:J-T}, \eqref{def:v-w}, and 	\eqref{eq:nls-w}, we have that 
\EQ{
Jw(t)=-i\int_{0}^t S(t-s)J(|u(s)|^pu(s))\ds.
}
By Lemma \ref{lem:strichartz} and H\"older's inequality, there exists an absolute constant $C_0>0$ such that
\EQn{\label{421}
\norm{Jw}_{L^\infty_tL_x^2}+\norm{Jw}_{S^0} \loe C_0 \norm{u^pJu}_{L_t^{q_0'} L_x^{r_0'}} \loe C_0 \norm{u}_{S^0}^p \norm{Ju}_{L_t^{q_1} L_x^{r}},
}
where the parameters $r_0,q_0,q_1,r$ are defined in the proof of Lemma \ref{lem:Ju-local}:
\begin{align*}
	\frac1{r_0}=\frac{d+4-dp}{2(d+2)};\quad 
	\frac1{q_0}=\frac{d(dp-2)}{4(d+2)};\quad 
	\frac1{q_1}=\frac{2d+8-(d+2)dp}{4(d+2)};\quad 
	\frac1q=\frac1r=\frac{d}{2(d+2)}.
\end{align*}
Note that 
\EQ{
\norm{Ju}_{L_t^{q_1} L_x^{r}} \loe \norm{Jw}_{L_t^{q_1} L_x^{r}} + \norm{Jv}_{L_t^{q_1} L_x^{r}}.
}
Moreover, when $\frac2d<p<\frac4d$, we have that $q_1<\frac{2(d+2)}{d}$. 
Using H\"older's inequality in $t$,
it gives
\EQ{
 \norm{Jw}_{L_t^{q_1} L_x^{r}}\loe \de^{\frac{1}{q_1}-\frac{d}{2(d+2)}} \norm{Jw}_{S^0}.
}
Furthermore, choosing $\ep=\frac{4-dp}{2}$, it holds that 
$$
\norm{Jv}_{L_t^{q_1} L_x^{r}}\loe \norm{Jv}_{X^{-\ep}}\loe A.
$$
Then, combining with \eqref{421}, these give that 
\EQ{
\norm{Jw}_{L^\infty_tL_x^2}+\norm{Jw}_{X^0} \loe C_0 \norm{u}_{X^0}^p\big( \de^{\frac{1}{q_1}-\frac{d}{2(d+2)}} \norm{Jw}_{X^0} + A\big).
}
By Lemma \ref{lem:local-l2}, we can set $\delta=\de(\norm{u_0}_{L_x^2})>0$ small enough such that 
$$
C_0\de^{\frac{1}{q_1}-\frac{d}{2(d+2)}}\norm{u}_{X^0}^p\le \frac12,
$$ 
which gives that 
\EQ{
\norm{Jw}_{L^\infty_tL_x^2} \le A.
}
This finishes the proof.
\end{proof}

\begin{proof}[Proof of Proposition \ref{prop:finialdata}]
Set $T_0=\frac1\delta$, where $\delta$ is determined in Lemma \ref{lem:local-Jw}. 
By \eqref{H1-T}, mass conservation law and Lemma  \ref{lem:local-Jw}, we have that 
\begin{align*} 
\big\|\W(t)\big\|_{L^2_x}&=\big\|w(\frac1t)\big\|_{L^2_x}
\le \big\|u(\frac1t)\big\|_{L^2_x}+\big\|v(\frac1t)\big\|_{L^2_x}
 =2\|u_0\|_{L^2}\le 2A;
 \end{align*}
 and 
 \begin{align*}
\|\nabla \W(T_0)\|_{L^2_x}&=\big\|Jw(\de)\big\|_{L^2_x}\le 2A.
\end{align*}
This proposition then follows by Gagliardo-Nirenberg's inequality in Lemma \ref{lem:GN}. 
\end{proof}

\subsection{Energy estimate towards the origin} \label{sec:energy}
Now, to finish the proof of Theorem \ref{thm:main}, it suffices to prove that:
\begin{prop}\label{prop:energy-pseudo}
Let $T_0$ be determined in Proposition  \ref{prop:finialdata}. Under the hypothesis \eqref{Hyth-determined}, 
\EQ{
\sup_{0<t\le T_0} \E(t) \loe C(T_0,A).
}
\end{prop}
\begin{proof}
By Proposition \ref{prop:finialdata} and Gagliardo-Nirenberg's inequality,
\EQn{\label{eq:bound-W-energy-de}
	\normb{\nabla \W(T_0,\cdot)}_{L_x^2(\R^d)}^2 + \normb{ \W(T_0,\cdot)}_{L_x^{p+2}(\R^d)}^{p+2} \lsm A^2.
}
This combining with the hypothesis \eqref{Hyth-determined} gives that 
\EQ{
\E(T_0) \lsm A^2.
}
By \eqref{eq:nls-W} and integration by parts,
\EQ{
\frac{\dd}{\dd t}\E(t) = & \re\int_{\R^d}|\U|^p\U\wb{\V_t}\dx + (2-\frac {dp}2)t^{1-\frac{dp}2} \int_{\R^d} |\nabla \W|^2 \dd x \\ 
\goe & \re\int_{\R^d}|\U|^p\U\wb{\Delta \V}\dx.
}
Therefore, using H\"older's inequality,
\EQ{
0 \loe \frac{\dd}{\dd t}\E(t) + \big\|\U\big\|_{L^{p+2}_x}^{p+1}\big\|\Delta\V\big\|_{L^{p+2}_x}
\loe \frac{\dd}{\dd t}\E(t) + C \big\|\Delta\V\big\|_{L^{p+2}_x}\E(t)^{\frac{p+1}{p+2}}.
}
Then,
\EQ{
0\loe \frac{\dd}{\dd t}\brk{\E(t)^{\frac{1}{p+2}}}+ C \normb{\De\V(t)}_{L_x^{p+2}}.
}
Then, integrating in $t$ and by the hypothesis \eqref{Hyth-determined}, it infers that 
\EQ{
\sup_{0<t\loe T_0}\E(t)\loe & \E(T_0) + C\brkb{\int_0^{T_0} \normb{\De\V(t)}_{L_x^{p+2}}\dd t}^{p+2}  \\
\loe & CA^2 + CAT_0^{p+2} \\
\loe & C(T_0,A).
} 
This finishes the proof.
\end{proof}

\section*{Acknowledgement}
The authors are very grateful to Professor Kenji Nakanishi for introducing them to the scattering problem with fractional weighted condition and for providing several insightful discussions that have improved this paper.

%
%

\end{document}